\documentclass[12pt]{amsart}
\usepackage{amssymb}
\usepackage{graphicx}
\usepackage[cmtip,all]{xy}
\usepackage{bbm}
\usepackage{comment}
\usepackage{mathrsfs}
\usepackage{mathtools}
\usepackage{xcolor}
\usepackage[dvipdfmx]{hyperref}
\usepackage{autobreak}
\hypersetup{
 colorlinks,
 linkcolor={blue!50!black},
 citecolor={blue!50!black},
 urlcolor={blue!80!black}
}

\everymath{\displaystyle}

\calclayout

\def\DefineSymbol#1#2{\newcommand{#1}{{\mathrm {#2}}}}
\def\DefineCategory#1#2{\newcommand{#1}{{\mathrm {#2}}}}

\theoremstyle{plain}
	\newtheorem{theorem}{Theorem}[section]
	\newtheorem{lemma}[theorem]{Lemma}
	\newtheorem{proposition}[theorem]{Proposition}
	\newtheorem{corollary}[theorem]{Corollary}
	\newtheorem{conjecture}[theorem]{Conjecture}
	
\theoremstyle{definition}
	\newtheorem{definition}[theorem]{Definition}
	\newtheorem{lemma and definition}[theorem]{Lemma and Definition}
	\newtheorem{example}[theorem]{Example}
	
	\newtheorem{question}[theorem]{Question}
	
	\newtheorem{construction}[theorem]{Construction}
	\newtheorem{notation}[theorem]{Notation}
    \newtheorem{claim}{Claim}
\theoremstyle{remark}
	\newtheorem{remark}[theorem]{Remark}
	\numberwithin{equation}{section}

\DefineSymbol{\pr}{pr}
\DefineSymbol{\id}{id}
\DefineSymbol{\const}{const}
\DefineSymbol{\op}{op}
\DefineSymbol{\diag}{diag}
\DefineSymbol{\proet}{pro\acute{e}t}
\DefineSymbol{\cond}{cond}
\DefineSymbol{\cont}{cont}
\DefineSymbol{\conti}{conti}
\DefineSymbol{\Cond}{Cond}
\DefineSymbol{\dCond}{dCond}
\DefineSymbol{\disc}{disc}
\DefineSymbol{\Tot}{Tot}
\DefineSymbol{\triv}{triv}
\DefineSymbol{\Bun}{Bun}
\DefineSymbol{\adic}{adic}
\DefineSymbol{\rig}{rig}
\DefineSymbol{\nc}{nc}
\DefineSymbol{\nuc}{nuc}
\DefineSymbol{\cpt}{cpt}
\DefineSymbol{\an}{an}
\DefineSymbol{\alg}{alg}
\DefineSymbol{\la}{la}
\DefineSymbol{\rla}{rla}
\DefineSymbol{\dep}{dep}
\DefineSymbol{\red}{red}
\DefineSymbol{\ab}{ab}

\DeclareMathOperator{\Hom}{Hom}

\DeclareMathOperator{\Ker}{Ker}

\let\Im\relax
\DeclareMathOperator{\Im}{Im}

\DeclareMathOperator{\Aut}{Aut}
\DeclareMathOperator{\Gal}{Gal}

\DeclareMathOperator{\cofib}{cofib}
\DeclareMathOperator{\fib}{fib}

\DeclareMathOperator{\Spec}{Spec}
\DeclareMathOperator{\AnSpec}{AnSpec}
\DeclareMathOperator{\Spa}{Spa}

\DeclareMathOperator{\Spf}{Spf}

\DeclareMathOperator{\VB}{VB}
\DeclareMathOperator{\intHom}{\underline{Hom}}
\DeclareMathOperator*{\Rvarprojlim}{\mathit{R}\varprojlim}

\newcommand{\lra}{\longrightarrow}

\newcommand{\heart}{\heartsuit}

\newcommand{\hotimes}{\mathbin{\hat{\otimes}}}

\DefineCategory{\Set}{Set}
\DefineCategory{\Ab}{Ab}
\DefineCategory{\Ring}{Ring}
\DefineCategory{\Mod}{Mod}
\DefineCategory{\Rep}{Rep}
\DefineCategory{\coMod}{coMod}
\DefineCategory{\Alg}{Alg}
\DefineCategory{\Ch}{Ch}
\DefineCategory{\Mon}{Mon}
\DefineCategory{\CMon}{CMon}
\DefineCategory{\PCoh}{PCoh}
\DefineCategory{\Perf}{Perf}
\DefineCategory{\FP}{FP}
\DefineCategory{\Aff}{Aff}
\DefineCategory{\cAff}{cAff}
\DefineCategory{\AnRing}{AnRing}
\DefineCategory{\Ani}{Ani}
\DefineCategory{\CAlg}{CAlg}
\DefineCategory{\AlgAff}{AlgAff}
\DefineCategory{\AlgAffPair}{AlgAffPair}
\DefineCategory{\AlgAnSp}{AlgAnSp}
\DefineCategory{\AnAdic}{AnAdic}

\renewcommand{\Bbb}{\mathbb{B}}

\newcommand{\Ebb}{\mathbb{E}}

\newcommand{\Gbb}{\mathbb{G}}

\newcommand{\Lbb}{\mathbb{L}}

\newcommand{\Nbb}{\mathbb{N}}

\newcommand{\Pbb}{\mathbb{P}}
\newcommand{\Qbb}{\mathbb{Q}}

\newcommand{\Zbb}{\mathbb{Z}}

\newcommand{\Acal}{\mathcal{A}}
\newcommand{\Bcal}{\mathcal{B}}
\newcommand{\Ccal}{\mathcal{C}}
\newcommand{\Dcal}{\mathcal{D}}

\newcommand{\Ncal}{\mathcal{N}}
\newcommand{\Ocal}{\mathcal{O}}
\newcommand{\Pcal}{\mathcal{P}}

\newcommand{\Rcal}{\mathcal{R}}

\newcommand{\Wcal}{\mathcal{W}}

\newcommand{\Afrak}{\mathfrak{A}}

\newcommand{\Xfrak}{\mathfrak{X}}

\newcommand{\mfrak}{\mathfrak{m}}
\newcommand{\nfrak}{\mathfrak{n}}

\newcommand{\pfrak}{\mathfrak{p}}

\renewcommand{\tilde}{\widetilde}
\renewcommand{\hat}{\widehat}
\renewcommand{\bar}{\overline}

\begin{document}

\title{$(\varphi,\Gamma)$-modules over relatively discrete algebras}
\author{Yutaro Mikami}
\date{\today}
\address{Graduate School of Mathematical Sciences, University of Tokyo, 3-8-1 Komaba, Meguro-ku, Tokyo 153-8914, Japan}
\email{y-mikmi@g.ecc.u-tokyo.ac.jp}
\begin{abstract}
In this paper, we study $(\varphi,\Gamma)$-modules over rings which are ``combinations of discrete algebras and affinoid $\Qbb_p$-algebras'', and prove basic results: the existence of a fully faithful functor from the category of Galois representations, the deperfection of $(\varphi,\Gamma)$-modules over perfect period rings, the dualizability of the cohomology of $(\varphi,\Gamma)$-modules, and the classification of $(\varphi,\Gamma)$-modules of rank $1$.
This work is motivated by the categorical $p$-adic local Langlands correspondence for locally analytic representations, as proposed by Emerton-Gee-Hellmann, and the $GL_1$ case, as formulated and proved by Rodrigues Jacinto-Rodr\'{\i}guez Camargo.
\end{abstract}
\maketitle

\setcounter{tocdepth}{1}

\tableofcontents

\section*{Introduction}
\subsection{Background}
Let $K$ be a finite extension of $\Qbb_p$.
The \textit{$p$-adic local Langlands correspondence} (for $GL_n(K)$) is a correspondence between unitary Banach representations of $GL_n(K)$ and $p$-adic representations of the absolute Galois group $G_K=\Gal(\bar{K}/K)$ of $K$.
When $GL_n(K)=GL_2(\Qbb_p)$, then the correspondence was constructed by Colmez in \cite{Col10}.
A key ingredient of the construction is the theory of $(\varphi,\Gamma_K)$-modules.
\'{E}tale $(\varphi,\Gamma_K)$-modules were introduced by Fontaine in \cite{Fontaine90}, and he proved that the category of \'{e}tale $(\varphi,\Gamma_K)$-modules is equivalent to the category of $p$-adic representations of $G_K$.
Berger introduced $(\varphi,\Gamma_K)$-modules over the Robba ring $\Rcal_K$ in \cite{Ber02} building on works of Cherbonnier-Colmez and Kedlaya \cite{CC98, Ked04}.
The category of $(\varphi,\Gamma_K)$-modules over $\Rcal_K$ contains the category of \'{e}tale $(\varphi,\Gamma_K)$-modules (or equivalently the category of $p$-adic representations of $G_K$) as a full subcategory.
An irreducible $p$-adic representation of $G_K$ might be reducible in the category of $(\varphi,\Gamma_K)$-modules over $\Rcal_K$. 
Indeed there are many such representations (e.g., trianguline representations), and these representations play an important role in the $p$-adic local Langlands correspondence for $GL_2(\Qbb_p)$.
It is natural to ask whether there is a correspondence between $(\varphi,\Gamma_K)$-modules over $\Rcal_K$ of rank $n$ and some representations of $GL_n(K)$.
In fact, Colmez constructed a correspondence between $(\varphi,\Gamma_{\Qbb_p})$-modules over $\Rcal_{\Qbb_p}$ of rank $2$ and locally analytic representations of $GL_2(\Qbb_p)$ in \cite{Col16}, which we call the \textit{$p$-adic local Langlands correspondence for locally analytic representations}.

Recently, Emerton-Gee-Hellmann proposed a \textit{categorical $p$-adic local Langlands correspondence} in \cite{EGH23}.
We explain the categorical $p$-adic local Langlands correspondence for locally analytic representations.
To formulate it, we use a rigid analytic moduli stack $\Xfrak_{n,K}$ of $(\varphi,\Gamma_K)$-modules over the Robba ring of rank $n$, which is called the \textit{(analytic) Emerton-Gee stack}.
Let $\Rep_{\square}^{\la}(GL_n(K))$ denote the stable $\infty$-category of locally analytic representations of $GL_n(K)$ defined in \cite[Definition 3.2.1]{RJRC23}, and let $\Dcal(\Xfrak_{n,K})$ denote the stable $\infty$-category of solid quasi-coherent complexes on $\Xfrak_{n,K}$.

\begin{conjecture}[{\cite[Conjecture 6.2.4, Remark 6.2.9 (d)]{EGH23}}]
There exists an exact functor
$$\Afrak_{GL_n(K)}^{\rig} \colon \Rep_{\square}^{\la}(GL_n(K)) \to \Dcal(\Xfrak_{n,K})$$
satisfying many good properties.
\end{conjecture}
\begin{remark}
More precisely, it seems necessary to restrict to the subcategories consisting of objects satisfying suitable finiteness conditions, but we will not pursue this in this paper.
\end{remark}

However, it was pointed out in \cite[Remark 6.2.9 (c)]{EGH23} that the functor is not expected to be fully faithful.
When $n=1$, this problem was resolved by Rodrigues Jacinto-Rodr\'{\i}guez Camargo in \cite{RJRC23} by using a modification of $\Xfrak_{1,K}$.
We explain it.
Let $K_0$ be the maximal unramified extension of $\Qbb_p$ in $K$.
Every $(\varphi,\Gamma_K)$-module of rank $1$ over the Robba ring $\Rcal_{K,A}$ on an affinoid $K_0$-algebra $A$ is given by a continuous character $K^{\times} \to A^{\times}$ up to a twist by a line bundle on $\Spa(A)$ (\cite[Theorem 6.2.14]{KPX14}).
Therefore, we obtain a simple description of $\Xfrak_{1,K}$
$$\Xfrak_{1,K}\cong[(\Wcal\times \Gbb_m^{\an})/\Gbb_m^{\an}]$$
with trivial $\Gbb_m^{\an}$-action, where $\Wcal$ is the rigid analytic moduli space of continuous characters of $\Ocal_K^{\times}$, and where $\Gbb_m^{\an}$ is the rigid analytic multiplicative group.
They used a modification $\Xfrak_{1,K}^{\mathrm{mod}}$ of $\Xfrak_{1,K}$ which is given by 
$$\Xfrak_{1,K}^{\mathrm{mod}}=[(\Wcal\times \Gbb_m^{\alg})/\Gbb_m^{\alg}],$$
where $\Gbb_m^{\alg}=\AnSpec (K_0[T^{\pm 1}],\Ocal_{K_0})_{\square}$ is an analytic space in the sense of Clausen-Scholze \cite{AG}, and where we regard $\Wcal$ as an analytic space.
\begin{theorem}[{\cite[Theorem 4.4.4]{RJRC23}}]
     There is a natural equivalence of stable $\infty$-categories
     $$\prod_{\Zbb}\Rep_{\square}^{\la}(K^{\times})\overset{\sim}{\to}\Dcal(\Xfrak_{1,K}^{\mathrm{mod}}).$$
\end{theorem}
We note that the left-hand side can be regarded as the stable $\infty$-category of quasi-coherent complexes on the analytic stack $\Bun_{K^{\times}}^{\la}=\coprod_{\Zbb}[\ast/K^{\times,\la}]$.
The modification of $\Xfrak_{1,K}$ depends on the explicit presentation of $\Xfrak_{1,K}$, so it is natural to ask the following question.
\begin{question}\label{ques:moduli interpretation}
Is there a moduli interpretation of $\Xfrak_{1,K}^{\mathrm{mod}}$?
\end{question}
A natural candidate is a moduli stack of $(\varphi,\Gamma_K)$-modules with more general coefficients such as $A[T^{\pm 1}]$ where $A$ is an affinoid $K_0$-algebra.
In this paper, we consider $(\varphi,\Gamma_K)$-modules with more general coefficients as above.

\subsection{Statements of the main results}
Let us now describe what is carried in this paper.

\subsubsection*{Comparison of some definitions of $(\varphi,\Gamma)$-modules}
We use four kinds of families of $p$-adic period rings (cf.\ \cite{Ber08B-pair}), and we do not use the Robba ring (i.e., ``limit of $p$-adic period rings'') (Remark \ref{rem:the reason why family}).
The first one is $\tilde{B}_{\bar{K}}=\{\tilde{B}_{\bar{K}}^{I}\}_{I}$ which is a family of ring of functions on rational open subsets of $\Spa(W(\Ocal_{\hat{\bar{K}}}^{\flat}))\setminus \{p[p^{\flat}]=0\}$.
It has a Frobenius action $\varphi$ and a continuous $G_K$-action.
By replacing $\bar{K}$ with $K_{\infty}=K(\zeta_{p^{\infty}})$, we obtain $\tilde{B}_{K_{\infty}}=\{\tilde{B}_{K_{\infty}}^{I}\}_{I}$, which has a Frobenius action $\varphi$ and a continuous $\Gamma_K=\Gal(K_{\infty}/K)$-action. 
Since $\Gamma_K$ is a $p$-adic Lie group, we obtain the subrings $B_{K,\infty}=\{\tilde{B}_{K_{\infty}}^{I,\la}\}_{I}$ of $\Gamma_K$-locally analytic vectors, which was studied in \cite{Ber16}.
The final one is a deperfection $B_{K}=\{B_{K}^{I}\}_{I}$ of $\tilde{B}_{K_{\infty}}$.
Then we can define a notion of \textit{$(\varphi,G_K)$-modules over $\tilde{B}_{\bar{K}}$} (resp.\ \textit{$(\varphi,\Gamma_K)$-modules over $\tilde{B}_{K_{\infty}}$}, $B_{K,\infty}$, $B_{K}$) as a family of finite projective $\tilde{B}_{\bar{K}}^{I}$-modules (resp.\ $\tilde{B}_{K_{\infty}}^{I}$-modules, $B_{K,\infty}^{I}$-modules, $B_K^{I}$-modules) $M=\{M^{I}\}_{I}$ with $\varphi$-actions and $G_K$-actions (resp.\ $\Gamma_K$-actions) (Definition \ref{def:perfect module}, Definition \ref{def:family of phi Gamma module perfect la}, Definition \ref{def:family of phi Gamma module}).
By the work of Berger, Fargues-Fontaine, Kedlaya, and Porat, the categories of these modules are canonically equivalent (cf.\ \cite[Theorem 5.1.1, Theorem 5.1.5]{EGH23}).
We generalize it for more general coefficient rings.
Let $A$ be an algebraic-affinoid $\Qbb_{p,\square}$-algebra (e.g., $A=\Qbb_p[T^{\pm 1}]$), see Definition \ref{def:algebraic-affinoid algebra}.
By replacing $\tilde{B}_{\bar{K}}^{I}$ with $\tilde{B}_{\bar{K},A}^{I}=\tilde{B}_{\bar{K}}^{I}\otimes_{\Qbb_{p,\square}}A$, we can define a notion of \textit{$(\varphi,G_K)$-modules over $\tilde{B}_{\bar{K},A}$}.
Similarly, we can define a notion of \textit{$(\varphi,\Gamma_K)$-modules over $\tilde{B}_{K_{\infty},A}$} (resp.\ $B_{K,\infty,A}$, $B_{K,A}$).

Then we can show the following theorem by a standard argument.

\begin{theorem}[{Theorem \ref{thm:Galois representation fully faithful functor}}]
    There exists a natural fully faithful functor from the category of continuous $G_K$-representations on finite projective $A$-modules to the category of $(\varphi,G_K)$-modules over $\tilde{B}_{\bar{K},A}$.
\end{theorem}

The following is one of the main theorems.
\begin{theorem}[{Theorem \ref{thm:FF curve descent}, Theorem \ref{thm:deperfection}}]
    The following categories are canonically equivalent.
    \begin{enumerate}
        \item The category of $(\varphi,G_K)$-modules over $\tilde{B}_{\bar{K},A}$.
        \item The category of $(\varphi,\Gamma_K)$-modules over $\tilde{B}_{K_{\infty},A}$.
        \item The category of $(\varphi,\Gamma_K)$-modules over $B_{K,\infty,A}$.
        \item The category of $(\varphi,\Gamma_K)$-modules over $B_{K,A}$.
    \end{enumerate}
\end{theorem}

We note that this theorem is well-known to experts when $A$ is an affinoid $\Qbb_p$-algebra (cf.\ \cite[5.1]{EGH23}).
We sketch the proof when $A$ is not an affinoid $\Qbb_p$-algebra.
First, we explain the categorical equivalence of the categories (1) and (2).
It basically follows from a descent along $\tilde{B}_{K_{\infty},A}^{I} \to \tilde{B}_{\bar{K},A}^{I}$, which is a ``profinite Galois cover with Galois group $H_K=\Gal(\bar{K}/K_{\infty})$''. 
We use the following theorem to prove that for $M\in \Dcal((\tilde{B}_{K_{\infty},A}^{I},\Zbb_p)_{\square})$, if $M\otimes_{(\tilde{B}_{K_{\infty},A}^{I},\Zbb_p)_{\square}}^{\Lbb}\tilde{B}_{\bar{K},A}^{I}$ is finite projective over $\tilde{B}_{\bar{K},A}^{I}$, then $M$ is finite projective over $\tilde{B}_{K_{\infty},A}^{I}$.
We note that it is not clear because $(\tilde{B}_{K_{\infty},A}^{I},\Zbb_p)_{\square}$ is not Fredholm.

\begin{theorem}[{Theorem \ref{thm:finite projective non-Fredholm}}]
    Let $A$ be a Jacobson-Tate algebra (e.g., $\tilde{B}_{K_{\infty}}^{I}$), and $B$ be a (condensed) animated relatively discrete $A$-algebra.
    Then $M\in \Dcal((B,\Zbb_p)_{\square})$ is a finite projective $B$-module if and only if it satisfies the following conditions:
    \begin{enumerate}
        \item The object $M$ is dualizable.
        \item The objects $M$ and $M^{\vee}=R\intHom_{B}(M,B)$ belong to $\Dcal(B_{\square})^{\leq 0}$.
    \end{enumerate}
\end{theorem}

Next, we explain the categorical equivalence of the categories (2), (3), and (4).
When $A$ is an affinoid $\Qbb_p$-algebra, then it can be proved by the Tate-Sen methods.
However, it cannot be applied naively when $A$ is not an affinoid $\Qbb_p$-algebra.
To resolve this problem, we use locally analytic vectors for $\Gamma_K$-actions.
This idea is motivated by \cite{Ber16,Por24}.
Based on this idea, we can reduce to the case when $A$ is an affinoid $\Qbb_p$-algebra. 

\subsubsection*{Dualizability of cohomology of $(\varphi,\Gamma_K)$-modules}
We explain the dualizability of the cohomology of $(\varphi,\Gamma_K)$-modules.
Our argument is essentially similar to the one in \cite{Bel24}.
First, we define a cohomology theory for 
\begin{itemize}
    \item $(\varphi,G_K)$-modules over $\tilde{B}_{\bar{K},A}$, and
    \item $(\varphi,\Gamma_K)$-modules over $\tilde{B}_{K_{\infty},A}$, $B_{K,\infty,A}$, and $B_{K,A}$
\end{itemize} 
by using an analog of the Herr complexes introduced by Herr in \cite{Herr98}, see Definition \ref{def:phi cohomology perfect} and Definition \ref{def:phi cohomology imperfect}.
Then we prove that these cohomology are equivalent by using an analog of \cite[Lemma 3.7]{Bel24}.
As a corollary of the proof, we obtain a simple presentation of the $(\varphi,\Gamma_K)$-cohomology $R\Gamma_{\varphi,\Gamma_K}(N)$ for a $(\varphi,\Gamma_K)$-module $N=\{N^{I}\}$ over $B_{K,A}$ as follows (Corollary \ref{cor:phi Gamma cohomology simple presentation}):
$$R\Gamma_{\varphi,\Gamma_K}(N)\simeq R\Gamma(\Gamma_K,\fib(\Phi-1\colon N^{[p^{-k},p^{-l}]}\to N^{[p^{-k},p^{-l-1}]})),$$
which is an analog of \cite[Proposition 3.8]{Bel24}.
By using this, we obtain the following theorem.
\begin{theorem}[{Theorem \ref{thm:dualizability}}]
    Let $N$ be a $(\varphi,\Gamma_K)$-module over $B_{K,A}$.
    Then the $(\varphi,\Gamma_K)$-cohomology $R\Gamma_{\varphi,\Gamma_K}(N)$ of $N$ is a dualizable object in $\Dcal(A_{\square})$.
\end{theorem}

\begin{remark}
    It is natural to ask whether Tate local duality as stated in \cite[Theorem 4.4.5]{KPX14} holds when $A$ is an algebraic-affinoid $\Qbb_{p,\square}$-algebra.
    However, the argument used in loc. cit., reducing to the case that $A$ is a finite extension of $\Qbb_p$, does not work here, since Lemma 4.1.6 in loc. cit. is not true in this setting; see Example \ref{ex:nonzero pointwise zero}.
    Tate local duality has been proved in \cite{Mikami25} using 6-functor formalism on analytic stacks, but we do not pursue this approach in the present paper.
\end{remark}

\subsubsection*{The classification of $(\varphi,\Gamma_K)$-modules of rank $1$}
Let $K_0$ be the maximal unramified extension of $\Qbb_p$ in $K$.
We prove the following classification result (under a certain mild assumption), which provides an affirmative answer to Question \ref{ques:moduli interpretation}.
\begin{theorem}[{Theorem \ref{thm:classification rank 1}}]
    Let $A$ be an algebraic-affinoid $K_{0,\square}$-algebra, and $M$ be a $(\varphi,\Gamma_K)$-module over $\tilde{B}_{K,A}$ of rank $1$ satisfying a certain freeness condition (see Theorem \ref{thm:classification rank 1} for details).
    Then there exists a character $\delta \colon K^{\times} \to A^{\times}$ and a finite projective $A$-module $L$ of rank $1$ such that $M\cong \tilde{B}_{K,A}(\delta)\otimes_{A}L$, where $\tilde{B}_{K,A}(\delta)$ is the $(\varphi,\Gamma_K)$-module obtained from the character $\delta$ (cf.\ Definition \ref{def:character type}).
\end{theorem}

\begin{remark}
    The freeness condition in the above theorem can be removed by using the $p$-adic monodromy theorem (i.e., ``de Rham'' implies ``potentially semistable'', \cite{Mikami25m}), but we will not pursue this in the present paper.
\end{remark}

We outline the proof of this theorem.
The basic strategy is to reduce to the case where $A$ is an affinoid $K_0$-algebra.
First, we reduce to the case $A=R_f=R[1/f]$, where $R$ is an affinoid $K_0$-algebra and $f\in R$.
The slogan of the proof is that \emph{the action of $\Gamma_K$ is analytic, while the action of $\varphi$ is algebraic.}
Based on this slogan, we prove that after suitably modifying the action of $\varphi$, a $(\varphi,\Gamma_K)$-module over $\tilde{B}_{K,R_f}$ of rank $1$ descends to one over $\tilde{B}_{K,R}$.
Then, the theorem follows from the classification result in \cite{KPX14}.

\subsection{Outline of the paper}
This paper is organized as follows.
In the former part of Section 1, we discuss finite projective modules over non-Fredholm analytic animated rings. 
In the latter part of Section 1, we collect some auxiliary results about group cohomology and locally analytic representations in the context of condensed mathematics.
In Section 2, we define algebraic-affinoid (analytic) $\Qbb_{p,\square}$-algebra (e.g., $\Qbb_p[T^{\pm 1}]$), and prove basic properties of them.
In the first part of Section 3, we compare $(\varphi,G_K)$-modules and $(\varphi,\Gamma_K)$-modules over perfect period rings.
In the second part of Section 3, we compare $(\varphi,\Gamma_K)$-modules over perfect period rings and imperfect period rings.
In Section 4, we define and compare the cohomology of $(\varphi,\Gamma_K)$-modules over perfect period rings and imperfect period rings.
Then we prove the dualizability of this cohomology.
In Section 5, we classify $(\varphi,\Gamma_K)$-modules of rank $1$.
In Appendix, we give examples illustrating the difference between dualizable complexes and perfect complexes.
\subsection{Convention}

\begin{itemize}
\item
All rings, including condensed ones, are assumed unital and commutative.
\item 
In contrast to \cite{And21, Mann22}, we use the term \textit{ring} to refer to an \textit{ordinary ring} (not an animated ring). 
Moreover, we use the symbol $-\otimes-$ to refer to a \textit{non-derived tensor product} and use the symbol $-\otimes^{\Lbb}-$ to refer to a \textit{derived tensor product}.
Similarly, we adopt analogous notation for Hom and limit.
\item
We use the terms \textit{f-adic ring} and \textit{affinoid pair} rather than \textit{Huber ring} and \textit{Huber pair}.
\item For an f-adic ring $A$, we denote the ring of power-bounded elements of $A$ by $A^{\circ}$.
\item For a complete non-archimedean field $K$, we use the term \textit{affinoid $K$-algebra} to refer to a \textit{topological $K$-algebra topologically of finite type over $K$}.
\item
We denote the simplex category by $\Delta$, which is the full subcategory of the category of totally ordered sets consisting of the totally ordered sets $[n]=\{0,\ldots,n\}$ for all $n \geq 0$. 
Moreover, for every $m \geq 0$, we denote the full subcategory of $\Delta$ consisting of $[n]$ for all $0 \leq n \leq m$ by $\Delta_{\leq m}$.
\item Throughout this paper, all radii $r$ and $s$ are assumed to be rational numbers.
\end{itemize}

\subsection{Convention and notation about condensed mathematics}
In this paper, we use condensed mathematics.
We summarize the notations and conventions related to condensed mathematics.

\begin{itemize}
\item We often identify a compactly generated topological set, ring, group, etc. $X$ whose points are closed (i.e., $X$ is $T1$) with a condensed set, ring, group, etc. $\underline{X}$ associated to $X$. 
    It is justified by \cite[Proposition 1.7]{CM}.
    If there is no room for confusion, we simply write $X$ for $\underline{X}$.
\item For a condensed object $X$, we denote the underlying object of $X$ by $X(\ast)$, and denote the underlying discrete condensed object of $X$ by $X_{\disc}$. 
    Then $X_{\disc}$ is the condensed object associated with $X(\ast)$ with the discrete topology.
    We can naturally identify these objects, however, we distinguish between them to avoid confusion.
\item For a condensed set $X$, we simply write $x\in X$ to mean $x\in X(\ast)$.
\item 
    In contrast to \cite{Mann22}, we use the term \textit{ring} to refer to an \textit{ordinary ring} (not a condensed animated ring). 
    Sometimes we use the term \textit{discrete ring} (resp.\ \textit{discrete animated ring}) to refer to an ordinary ring (resp.\ animated ring) in order to emphasize that it is not a condensed one. We also use the term \textit{static ring} (resp.\ \textit{static analytic ring})  to refer to an ordinary ring (resp.\ analytic ring) in order to emphasize that it is not an animated one.
    \item We use the terms ``analytic animated ring'' and ``uncompleted analytic animated ring'' according to \cite{Mann22}.
    \item For an uncompleted analytic animated ring $\Acal$, we denote the underlying condensed animated ring of $\Acal$ by $\underline{\Acal}$.
    \item For an uncompleted analytic animated ring $\Acal$, an object $M\in \Dcal(\underline{\Acal})$ is said to be \textit{$\Acal$-complete} if it lies in $\Dcal(\Acal)$.    
    \item Let $\Zbb_{\square}$ and $\Zbb_{p,\square}$ denote the analytic rings defined in \cite[Example 7.3]{CM}. 
    Moreover, for a usual ring $A$, let $(A,A)_{\square}$ denote the analytic ring defined in \cite[Definition 2.9.1]{Mann22}.
    We note that the analytic ring structure of $\Zbb_{p,\square}$ is induced from $\Zbb_{\square}$.
    \item For a condensed animated ring $A$ and for a morphism of usual rings $B\to \pi_0A(\ast)$, let $(A,B)_{\square}$ denote the condensed animated ring $A$ with the induced analytic ring structure from $(B,B)_{\square}$.
    For details, see \cite[Definition 2.1]{Mikami23}.
    When $B=\Zbb$, then we simply write $A_{\square}$ for $(A,\Zbb)_{\square}$.
    \item For a profinite set $S$ and an object $M\in \Dcal(\Zbb_{\square})$, we write $C(S,M)=R\intHom_{\Zbb}(\Zbb_{\square}[S],M)$. 
    If $M$ is static then $C(S,M)$ is static, since $\Zbb_{\square}[S]$ is a projective $\Zbb_{\square}$-module.
    If $M$ is a $\Zbb_{\square}$-module associated to a compactly generated Hausdorff $\Zbb$-module $N$ (that is, $\underline{N}=M$), then $C(S,M)$ is a $\Zbb_{\square}$-module associated to the module $C(S,N)$ of continuous functions $S\to N$ endowed with the compact-open topology.
    \item For an analytic animated ring $\Rcal$, let $\Dcal^{\nuc}(\Rcal)$ denote the full subcategory of $\Dcal(\Rcal)$ spanned by nuclear objects.
\end{itemize}

\subsection{Notation}
Let $K$ be a finite extension of $\Qbb_p$.
We fix an algebraic closure $\bar{K}$ and its completion $C$.
Let $K_{\infty}=K(\zeta_{p^{\infty}})$ be the cyclotomic extension of $K$, and we write $\Gamma_K=\Gal(K_{\infty}/K)$, $H_K=\Gal(\bar{K}/K_{\infty})$, and $G_K=\Gal(\bar{K}/K)$ for the Galois group of $K_{\infty}/K$, $\bar{K}/K_{\infty}$, and $\bar{K}/K$.
We write $\chi \colon G_K \to \Zbb_p^{\times}$ for the $p$-adic cyclotomic character of $K$.
Let $W_K$ be the Weil group of $K$.
By using the local class field theory, we identify $W_K^{\ab}$ with $K^{\times}$ where a uniformizer of $K^{\times}$ corresponds to a geometric Frobenius.

\subsection*{Acknowledgements}
The author is grateful to Yoichi Mieda for his support during the studies of the author.
A part of this work was carried out during the author's stay at the University of M\"{u}nster under the Mathematics M\"{u}nster programme for Visiting Doctoral Researchers, and the author thanks Eugen Hellmann for his hospitality and support during the stay.
This work was supported by JSPS KAKENHI Grant Number JP23KJ0693.

\section{Preliminaries on condensed mathematics}
In this section, we recall and prove some facts about condensed mathematics which will be used later.
For the foundations of condensed mathematics, see \cite{CM, AG, CC, And21, Mann22, RJRC22}.
\subsection{Faithfully flat descent of finite projective modules}
In this subsection, we prove faithfully flat descent of finite projective modules in the condensed setting.
In this paper, we primarily deal with analytic ring structures induced from $\Zbb_{\square}$, so it is not difficult to prove descent of quasi-coherent complexes.
However, since the analytic animated rings which we consider are not necessarily Fredholm (see Definition \ref{def:Fredholm}), it is difficult to characterize finite projective modules in the stable $\infty$-category of quasi-coherent complexes.
We carry it out for special non-Fredholm analytic animated rings.

First, we define finite projective modules.
Let $\Acal$ be an analytic animated ring.
\begin{definition}
    A \textit{finite projective $\Acal$-module} is an object $M\in \Dcal(\Acal)$ which is a direct summand of $\underline{\Acal}^n$ for some $n\geq 0$.
\end{definition}

\begin{definition}\label{def:condensification functor}
    Let $A$ be a condensed animated ring, and let $A(\ast)$ denote the underlying animated ring of $A$. 
    Then we have a functor $$\Cond_{A}\colon \Dcal(A(\ast)) \to \Dcal(A);\; M \to A\otimes_{A_{\disc}}^{\Lbb} \underline{M},$$
    where $\underline{M}$ is the condensed object associated to $M$ (see Convention).
    We call it the \textit{condensification functor}.
    By the same argument as in \cite[Theorem 5.9]{And21}, we can show that it is fully faithful.
    If $A$ is the underlying condensed animated ring of $\Acal$, then we can show that the image of the functor $\Cond_{\underline{\Acal}}$ is contained in $\Dcal(\Acal)$.
    Therefore, we obtain a fully faithful functor
    $$\Cond_{\Acal}\coloneqq\Cond_{\underline{\Acal}}\colon \Dcal(\underline{\Acal}(\ast)) \to \Dcal(\Acal),$$
    which is also called the condensification functor.
\end{definition}

\begin{remark}\label{rem:cond functor}
    By the proof of \cite[Theorem 5.9]{And21}, we get $\Cond_{\Acal}(M)(\ast)\cong M$.
\end{remark}

\begin{definition}\label{def:relatively discrete}
    \begin{enumerate}
        \item An object $M \in \Dcal(\Acal)$ is \textit{relatively discrete} (over $\Acal$) if $M$ lies in the essential image of the condensification functor $\Cond_{\Acal}$.
        \item A (condensed) animated $\Acal$-algebra $B$ is \textit{relatively discrete} (over $\Acal$) if it is relatively discrete as an object of $\Dcal(\Acal)$.
        \item An object $M \in \Dcal(\Acal)$ is a \textit{perfect complex} (of $\Acal$-modules) if $M$ lies in the essential image of $\Perf(\underline{\Acal}(\ast))$ under the condensification functor $\Cond_{\Acal}$, where $\Perf(\underline{\Acal}(\ast))$ is the category of perfect complexes of $\underline{\Acal}(\ast)$-modules.
    \end{enumerate}
\end{definition}

\begin{remark}
    Since the definition of being relatively discrete does not depend on the analytic ring structure of $\Acal$, we also use the terminology ``relatively discrete over $\underline{\Acal}$''.
\end{remark}

\begin{remark}\label{rem:rd extension}
Since the condensification functor is fully faithful, for $M,N\in \Dcal(A(\ast))$, we have $\Hom_{A(\ast)}(M,N[1])\cong \Hom_{A}(\Cond_{A}(M),\Cond_{A}(N)[1])$.
Therefore, the class of relatively discrete objects is closed under extensions in $\Dcal(\Acal)$.
\end{remark}

\begin{remark}
    Let $\Pcal$ be a property of $\underline{\Acal}(\ast)$-modules (resp.\ $\underline{\Acal}(\ast)$-algebras).
    Then the property $\Pcal$ can naturally be defined for relatively discrete modules (resp.\ algebras) over $\Acal$ as well.
    For example, suppose that $\Acal$ is static.
    Then a relatively discrete and finitely generated $\Acal$-module means a relatively discrete $\Acal$-module whose underlying $\Acal(\ast)$-module is finitely generated.
    When it is clear from the context that the modules (resp.\ algebras) are relatively discrete, we simply omit the phrase “relatively discrete.”
\end{remark}

\begin{remark}\label{rem:nuclear}
    It is easy to show that $M\in \Dcal(\Acal)$ is a perfect complex if and only if it lies in the full subcategory generated under retracts, shifts, and finite colimits by $\underline{\Acal}$. 
    Similarly, $M\in \Dcal(\Acal)$ is relatively discrete if and only if it lies in the full subcategory generated under retracts, shifts, and small colimits by $\underline{\Acal}$. 
    In particular, if $M\in \Dcal(\Acal)$ is relatively discrete, then it is nuclear by \cite[Proposition 13.13]{CM}.
\end{remark}

\begin{remark}
    For a perfect complex $M\in \Dcal(\Acal)$, the dual $M^{\vee}=R\intHom_{\Acal}(M,\underline{\Acal})$ is also a perfect complex, and there is an equivalence
    \begin{align*}
        M^{\vee}(\ast)\simeq R\Hom_{\underline{\Acal}(\ast)}(M(\ast),\underline{\Acal}(\ast)).
    \end{align*}
\end{remark}

\begin{example}
    A polynomial ring $\underline{\Qbb_p}[T]$ over the condensed ring $\underline{\Qbb_p}$ is relatively discrete over $\Qbb_{p,\square}$.
\end{example}

\begin{lemma}\label{lem:finite projective module criterion}
    Let $M\in \Dcal(\Acal)$ be a perfect complex.
    Then $M$ is a finite projective module if and only if both $M$ and $M^{\vee}=R\intHom_{\Acal}(M,\underline{\Acal})$ lie in $\Dcal(\Acal)^{\leq 0}$.
\end{lemma}
\begin{proof}
    The only if direction is clear, so we prove the if direction.
    Since $M$ is a perfect complex and $M$ belongs to $\Dcal(\Acal)^{\leq 0}$, we can take a morphism $f\colon \underline{\Acal}^n \to M$ such that $\cofib(f)\in \Dcal(\Acal)^{\leq -1}$.
    Then we have a fiber sequence
    $$R\intHom_{\Acal}(M,\underline{\Acal}^n) \to R\intHom_{\Acal}(M,M) \to R\intHom_{\Acal}(M,\cofib(f)).$$
    Since $M$ is a perfect complex and $M^{\vee}$ belongs to $\Dcal(\Acal)^{\leq 0}$, we have $$R\intHom_{\Acal}(M,\cofib(f))\simeq M^{\vee}\otimes_{\Acal}^{\Lbb} \cofib(f)\in \Dcal(\Acal)^{\leq -1}.$$
    Therefore, $\id_M \in H^0(R\intHom_{\Acal}(M,M))$ has a lift to $H^0(R\intHom_{\Acal}(M,\underline{\Acal}^n))$, which proves the lemma.
\end{proof}

\begin{lemma}\label{lem:finite projective module criterion 2}
    Let $M\in \Dcal(\Acal)$ be a bounded above object.
    Then $M\in \Dcal(\Acal)$ is a finite projective $\Acal$-module if and only if $M\otimes_{\Acal}^{\Lbb}\pi_0(\Acal)\in \Dcal(\pi_0(\Acal))$ is a finite projective $\pi_0(\Acal)$-module.
\end{lemma}
\begin{proof}
    The only if direction is clear, so we prove the if direction.
    First, we prove that $M$ is connective.
    We assume that there exists $i>0$ such that $H^i(M)\neq 0$.
    Since $M$ is bounded above, we can take the largest such $i$.
    Then we have $H^i(M)=H^i(M\otimes_{\Acal}^{\Lbb}\pi_0(\Acal))$.
    Since the right-hand side is zero, it is a contradiction.
    Let us take a surjection $\pi_0(\underline{\Acal})^n \to M\otimes_{\Acal}^{\Lbb}\pi_0(\Acal)$.
    Since the cofiber of the morphism $M\to M\otimes_{\Acal}^{\Lbb}\pi_0(\Acal)$ belongs to $\Dcal(\Acal)^{\leq -1}$, we can take a lift $\underline{\Acal}^n\to M$ of $\underline{\Acal}^n\to \pi_0(\underline{\Acal})^n \to M\otimes_{\Acal}^{\Lbb}\pi_0(\Acal)$.
    We set $N=\cofib(\underline{\Acal}^n\to M)\in \Dcal(\Acal)^{\leq -1}$.
    Then we have a fiber sequence
    $$R\intHom_{\Acal}(M,\underline{\Acal}^n) \to R\intHom_{\Acal}(M,M) \to R\intHom_{\Acal}(M,N).$$
    It is enough to show $R\intHom_{\Acal}(M,N)\in \Dcal(\Acal)^{\leq -1}$.
    For $m\geq 1$, $$R\intHom_{\Acal}(M,H^{-m}(N)[m])\cong R\intHom_{\pi_0(\Acal)}(M\otimes_{\Acal}^{\Lbb}\pi_0(\Acal),H^{-m}(N)[m])$$ is concentrated in cohomological degree $-m$.
    Hence $R\intHom_{\Acal}(M,\tau^{\geq -m}N)$ lies in $\Dcal(\Acal)^{[-m,-1]}$, and $\{R\intHom_{\Acal}(M,\tau^{\geq -m}N)\}_m$ is a Postnikov tower.
    Therefore, we obtain $$R\intHom_{\Acal}(M,N)\cong R\varprojlim_{m}R\intHom_{\Acal}(M,\tau^{\geq -m}N)\in \Dcal(\Acal)^{\leq -1}.$$
\end{proof}

On the other hand, by the general theory of closed symmetric monoidal stable $\infty$-categories, we have a notion of dualizable objects.

\begin{definition}\label{def:dualizable}
    An object $M \in \Dcal(\Acal)$ is \textit{dualizable} if there exists an object $M^{\prime}\in\Dcal(\Acal)$ with morphisms $\alpha \colon \underline{\Acal}\to M^{\prime}\otimes_{\Acal}^{\Lbb}M$ and $\beta \colon M\otimes_{\Acal}^{\Lbb}M^{\prime} \to\underline{\Acal}$ such that the compositions
    \begin{align*}
        &M\overset{\id\otimes\alpha}{\longrightarrow} M\otimes_{\Acal}^{\Lbb}M^{\prime}\otimes_{\Acal}^{\Lbb}M\overset{\beta\otimes\id}{\longrightarrow}M, \\
        &M^{\prime}\overset{\alpha\otimes\id}{\longrightarrow} M^{\prime}\otimes_{\Acal}^{\Lbb}M\otimes_{\Acal}^{\Lbb}M^{\prime}\overset{\id\otimes\beta}{\longrightarrow}M^{\prime}
    \end{align*}
    are the identity.
\end{definition}

\begin{remark}
    If $M$ is dualizable, then $M^{\prime}$ in the definition is unique and it is equivalent to $R\intHom_{\Acal}(M,\underline{\Acal})$.
\end{remark}

\begin{lemma}[{\cite[Proposition 9.3]{CC}}]\label{lem:dualizable compact nuclear}
    An object $M \in \Dcal(\Acal)$ is dualizable if and only if it is compact and nuclear as an object of $\Dcal(\Acal)$.
\end{lemma}

\begin{corollary}\label{cor:dualizable nuclear}
    The inclusion $\Dcal^{\nuc}(\Acal)\subset \Dcal(\Acal)$ induces a categorical equivalence between the full subcategories of dualizable objects in $\Dcal^{\nuc}(\Acal)$ and $\Dcal(\Acal)$.
\end{corollary}

In general, every perfect complex of $\Dcal(\Acal)$ is dualizable, however, the converse is not necessarily true.

\begin{definition}[{\cite[Definition 9.7]{CC}}]\label{def:Fredholm}
    The analytic animated ring $\Acal$ is \textit{Fredholm} if any dualizable object of $\Dcal(\Acal)$ is relatively discrete.
\end{definition}

\begin{remark}
    Let $\Acal$ be a Fredholm analytic animated ring.
    Then every dualizable object of $\Dcal(\Acal)$ is a perfect complex (\cite[Proposition 9.2]{CC}).
\end{remark}

\begin{example}
    \begin{enumerate}
        \item Let $A$ be a (discrete) animated ring, $A^+\subset \pi_0(A)$ is an integrally closed subring.
        Then the analytic animated ring $(A,A^+)_{\square}$ is Fredholm, which follows from \cite[Proposition 2.9.7 (i)]{Mann22}.
        \item Let $(A,A^+)$ be a complete Tate affinoid pair.
        Then the analytic ring $(A,A^+)_{\square}$ is Fredholm, which follows from \cite[Theorem 5.50]{And21}.
        We note that this theorem is true without the assumption that $A$ is sheafy.
        \item Let $\Qbb_{p}[T]$ denote a polynomial ring over the condensed ring $\Qbb_p$.
        Then the analytic ring $\Qbb_p[T]_{\square}$ is not Fredholm, see also Appendix \ref{appendix:nonFredholm}.
    \end{enumerate}
\end{example}

\begin{definition}
    Let $f \colon \Acal \to \Bcal$ be a morphism of analytic animated rings.
    Then we say that $f$ satisfies \textit{$\ast$-descent} if the natural functor
    $$\Dcal(\Acal) \to \varprojlim_{[n]\in \Delta}\Dcal(\Bcal^n)$$ 
    is an equivalence of $\infty$-categories, where $\Bcal^{\bullet}$ is the \v{C}ech nerve of $\Acal \to \Bcal$.
    We also say that $f$ satisfies \textit{universal $\ast$-descent} if for any morphism $\Acal\to \Ccal$ of analytic animated rings, $\Ccal\to \Ccal\otimes_{\Acal}^{\Lbb} \Bcal$ satisfies $\ast$-descent.
\end{definition}

The following proposition is proved in \cite[VI.1]{GRLLC24}.
For the reader's convenience, we provide a proof here.

\begin{proposition}\label{prop:descent finite projective Fredholm}
    Let $f \colon \Acal \to \Bcal$ be a morphism of analytic animated rings which satisfies $\ast$-descent.
     \begin{enumerate}
        \item For $M\in \Dcal(\Acal)$, if $M\otimes_{\Acal}^{\Lbb} \Bcal$ is a dualizable $\Bcal$-module, then $M$ is also a dualizable $\Acal$-module.
        \item For $M\in \Dcal(\Acal)$, if $M\otimes_{\Acal}^{\Lbb} \Bcal$ is a finite projective $\Bcal$-module and $M$ is relatively discrete in $\Dcal(\Acal)$, then $M$ is also a finite projective $\Acal$-module.
    \end{enumerate}
\end{proposition}

\begin{remark}
   If $\Acal$ is Fredholm, then by (1), the condition in (2) that $M$ is relatively discrete in $\Dcal(\Acal)$ is automatic from the assumption that $M\otimes_{\Acal}^{\Lbb} \Bcal$ is a finite projective $\Bcal$-module.
\end{remark}
\begin{proof}
    Part (1) follows from the same argument as in \cite[Proposition 6.12]{Mann22b}.
    Let us prove Part (2).
    Since $M$ is relatively discrete in $\Dcal(\Acal)$, $M$ is a perfect complex of $\Acal$-modules.
    In particular, $M$ is bounded above, so we may assume that $\Acal$ is static by Lemma \ref{lem:finite projective module criterion 2}.
    We set $A=\underline{\Acal}(\ast)$, which is a classical ring.
    Let us prove that $M(\ast)$ lies in $\Dcal(A)^{\leq 0}$.
    We assume that there exists $i\geq 1$ such that $H^i(M(\ast))\neq 0$, and we take the largest such $i$.
    Then there exists a maximal ideal $\mfrak\subset A$ such that $H^i(M(\ast)\otimes_{A}^{\Lbb}A/\mfrak)\neq 0$.
    We define an analytic ring $\Acal/\mfrak$ as $\Cond_{\Acal}(A/\mfrak)$ with the induced analytic ring structure from $\Acal$.
    We set $\Bcal/\mfrak=\Acal/\mfrak\otimes_{\Acal}^{\Lbb} \Bcal$.
    Since $\Acal \to \Acal/\mfrak$ is steady by \cite[Proposition 13.9]{AG}, there is an equivalence
    $$\underline{\Bcal/\mfrak}\simeq \underline{\Acal/\mfrak}\otimes_{\Acal}^{\Lbb} \Bcal.$$
    Therefore, $\Bcal/\mfrak$ is not equivalent to zero since $\Acal\to\Bcal$ satisfies $\ast$-descent.
    Since $A/\mfrak$ is a field, $H^i(M(\ast)\otimes_{A}^{\Lbb}A/\mfrak)[-i]$ is a direct summand of $M(\ast)\otimes_{A}^{\Lbb}A/\mfrak$ and $H^i(M(\ast)\otimes_{A}^{\Lbb}A/\mfrak)$ is a nonzero finite free $A/\mfrak$-module.
    Therefore, $H^i(M\otimes_{\Acal}^{\Lbb}\Acal/\mfrak)[-i]$ is a direct summand of $M\otimes_{\Acal}^{\Lbb}\Acal/\mfrak$ and $H^i(M\otimes_{\Acal}^{\Lbb}\Acal/\mfrak)$ is isomorphic to $(\underline{\Acal/\mfrak})^n$ for some $n>0$.
    We have an equivalence
    $$(M\otimes_{\Acal}^{\Lbb}\Bcal)\otimes_{\Bcal}^{\Lbb} \Bcal/\mfrak\simeq (M\otimes_{\Acal}^{\Lbb}\Acal/\mfrak)\otimes_{\Acal/\mfrak}^{\Lbb} \Bcal/\mfrak.$$
    From the above argument, we get $H^i((M\otimes_{\Acal}^{\Lbb}\Acal/\mfrak)\otimes_{\Acal/\mfrak}^{\Lbb} \Bcal/\mfrak)\neq 0$.
    On the other hand, since $M\otimes_{\Acal}^{\Lbb}\Bcal$ is a finite projective $\Bcal$-module, $(M\otimes_{\Acal}^{\Lbb}\Bcal)\otimes_{\Bcal}^{\Lbb} \Bcal/\mfrak$ is also a finite projective $\Bcal/\mfrak$-module.
    It is a contradiction.

    By applying the same argument to the dual $M^{\vee}=R\intHom_{\Acal}(M,\underline{\Acal})$, we find that $M^{\vee}(\ast)=M(\ast)^{\vee}=R\Hom_A(M(\ast),A)$ is also lies in $\Dcal(A)^{\leq 0}$.
    Therefore, by Lemma \ref{lem:finite projective module criterion}, $M(\ast)$ is a finite projective $A$-module, and $M$ is also a finite projective $\Acal$-module.
\end{proof}

Next, we consider descent of finite projective modules over non-Fredholm analytic animated rings for special cases.

\begin{lemma}\label{lem:condensification functor}
    Let $A$ be a noetherian complete Tate ring, and we regard it as a condensed ring.
    \begin{enumerate}
        \item The condensification functor $\Cond_{A}\colon \Dcal(A(\ast)) \to \Dcal(A_{\square})$ is $t$-exact.
        \item Let $M$ be a (static) $A(\ast)$-module, and $S$ be an extremally disconnected set.
        Then $f \in \Cond_A(M)(S)$ is equal to $0$ if and only if $f(s)=0$ for all $s\in S$, where $f(s)\in M$ is the image of $f$ under the morphism $\Cond_A(M)(S)\to \Cond_A(M)(\{s\})=M$. 
        \item Let $\{M_i\}_{i\in I}$ be a family of (static) $A(\ast)$-modules.
        Then the natural morphism $\Cond_A\left(\prod_{i\in I}M_i\right) \to \prod_{i\in I} \Cond_A(M_i)$ is injective.
    \end{enumerate}
\end{lemma}
\begin{proof}
    First, we prove (1).
    It is enough to show that for any finitely presented $A(\ast)$-module $M \in \Mod_{A(\ast)}$, $\Cond_{A}(M)$ is static.
    Since $A(\ast)$ is noetherian, $M$ is quasi-isomorphic to a complex of finite free $A(\ast)$-modules.
    Therefore, $\Cond_{A}(M)$ is static by \cite[Theorem 5.9]{And21}.

    Next, we prove (2). 
    We can write $M$ as a filtered colimit $M=\varinjlim_{\lambda\in \Lambda} M_{\lambda}$ of finitely presented submodules $M_{\lambda}\subset M$.
    Then we have $\Cond_A(M)\cong \varinjlim_{\lambda\in \Lambda} \Cond_A(M_{\lambda})$ and $\Cond_A(M_{\lambda}) \to \Cond_A(M)$ is injective by (1).
    Moreover, there exists $\lambda\in \Lambda$ such that $f\in \Cond_A(M)(S)$ belongs to the image of $\Cond_A(M_{\lambda})(S) \hookrightarrow \Cond_A(M)(S)$.
    Therefore, we may assume that $M$ is finitely presented.
    We endow $M$ with the natural topology. 
    Then $\Cond_A(M)$ is the condensed module associated to $M$ by \cite[Lemma 5.23]{And21}, so $M(S)$ is isomorphic to the module of continuous functions on $S$ with values in $M$, which proves (2).

    Finally, we prove (3).
    By (2), it is enough to show that $$\left(\Cond_A\left(\prod_{i\in I}M_i\right)\right)(\ast) \to \left(\prod_{i\in I} \Cond_A(M_i)\right)(\ast)=\prod_{i\in I}(\Cond_A(M_i)(\ast))$$ is injective.
    It is an isomorphism by Remark \ref{rem:cond functor}.
\end{proof}

\begin{lemma}\label{lem:rd sub}
	Let $R$ be a noetherian complete Tate ring, $A$ be a relatively discrete $R$-algebra, and $M$ be a relatively discrete $A$-module.
	Then an $A$-submodule $N\subset M$ is also relatively discrete over $A$.
\end{lemma}
\begin{proof}
	We note that an $A$-module $L$ is relatively discrete over $A$ if and only if it is relatively discrete over $R$.
	Let us prove the natural morphism
	$$f\colon \Cond_A(N(\ast))\to N$$
	is an isomorphism.
	Since the composition
	$$\Cond_A(N(\ast))\overset{f}{\to} N\to M$$
	is injective by Lemma \ref{lem:condensification functor} (1), $f$ is also injective.
	Let us prove that $f$ is surjective.
	By Lemma \ref{lem:condensification functor} (2), for a relatively discrete $A$-module $L$, an extremally disconnected set $S$ and $x\in L(S)$, if $x(s)=0$ for any $s\in S$, then $x=0$.
	By applying this claim to $L=M/\Cond_A(N(\ast))$, we get the claim.
\end{proof}

\begin{remark}
Let $R$, $A$, and $M$ be as in Lemma \ref{lem:rd sub}.
From this lemma, an $A$-submodule $N\subset M$ corresponds bijectively to an $A(\ast)$-submodule $N(\ast)\subset M(\ast)$.
We will often make use of this correspondence implicitly.
For example, we identify the set of ideals of $A(\ast)$ and $A$, and use the same notation when this causes no ambiguity.
\end{remark}

\begin{lemma}\label{lem:compact projective object and prod}
    Let $A$ be a noetherian complete Tate ring, and $B$ be a relatively discrete $A$-algebra.
    Let $S$ be an extremally disconnected set, and we take a set $I$ such that $\Zbb_{\square}[S]\cong\prod_I \Zbb$.
    Then $B_{\square}[S]$, which is equal to $B[S]\otimes_{\Zbb}^{\Lbb} \Zbb_{\square}$ by definition, is static and isomorphic to $B \otimes_{\Zbb_{\square}}^{\Lbb} \Zbb_{\square}[S]$.
    Moreover, the natural morphism $$B_{\square}[S]\cong B \otimes_{\Zbb_{\square}}^{\Lbb} \Zbb_{\square}[S]\cong B \otimes_{\Zbb_{\square}}^{\Lbb} \prod_I \Zbb \to \prod_I B$$ is injective.
\end{lemma}
\begin{proof}
    We can show $B[S]\otimes_{\Zbb}^{\Lbb} \Zbb_{\square}\cong B \otimes_{\Zbb_{\square}}^{\Lbb} \Zbb_{\square}[S]$ by the same argument as in \cite[Lemma 3.25]{And21}.
    Therefore, it is enough to show that for a relatively discrete (static) $A$-module $M$, $M\otimes_{\Zbb_{\square}}^{\Lbb}\prod_I\Zbb$ is static and $$M\otimes_{\Zbb_{\square}}\prod_I \Zbb\to \prod_I M$$ is injective.
    Since $M(\ast)$ can be written as a filtered colimit of finitely presented $A(\ast)$-submodules, we may assume that $M(\ast)$ is finitely presented.
    Then the same argument as in \cite[Theorem 3.27]{And21} works well.
\end{proof}
\begin{remark}
    We note that $\prod_I\Zbb$ is a compact projective object of $\Mod_{\Zbb_{\square}}$, but it is not necessarily flat (\cite[Lemma 2.9.35]{Mann22}).
    Therefore, it is not clear that $M\otimes_{\Zbb_{\square}}^{\Lbb}\prod_I\Zbb$ is static.
\end{remark}

\begin{definition}[{\cite[Definition 3.1]{Lou17}}]\label{def:Jacobson-Tate}
    Let $A$ be a strongly noetherian complete Tate ring.
    We say that $A$ is a Jacobson-Tate ring if it satisfies the following conditions:
    \begin{enumerate}
        \item For any maximal ideal $\mfrak \subset A$, $A/\mfrak$ is a complete non-archimedean field.
        \item For any $A$-algebra $B$ topologically of finite type and any maximal ideal $\nfrak$ of $B$, the restriction ideal of $\nfrak$ to $A$ is also a maximal ideal of $A$.
    \end{enumerate}
\end{definition}

\begin{example}
An affinoid algebra over a complete non-archimedean field is a Jacobson-Tate ring.
\end{example}

\begin{theorem}\label{thm:finite projective non-Fredholm}
    Let $A$ be a Jacobson-Tate algebra, and $B$ be an animated relatively discrete $A$-algebra.
    Then $M\in \Dcal(B_{\square})$ is a finite projective $B_{\square}$-module if and only if it satisfies the following conditions:
    \begin{enumerate}
        \item $M$ is dualizable.
        \item Both $M$ and $M^{\vee}=R\intHom_{B}(M,B)$ belong to $\Dcal(B_{\square})^{\leq 0}$.
    \end{enumerate}
\end{theorem}

\begin{proof}
    The only if direction is clear, so it suffices to prove the if direction.
    We may assume that $B$ is static by Lemma \ref{lem:finite projective module criterion 2}.
    By using a limit argument (\cite[Lemma 2.7.4]{Mann22}), we may assume that $B$ is of finite type over $A$.
    Let $\mfrak$ be a maximal ideal of $B$, and $\nfrak\subset A$ be the restriction ideal of $\mfrak$ to $A$.
    For $f \in B$, we write $B_f=B[1/f]$, which is also relatively discrete over $A$.
    We write $B_{\mfrak}=\varinjlim_{f\in B\setminus \mfrak}B_{f}$, which is also relatively discrete over $A$.
    For the time being, we assume that $M\otimes_{B_{\square}}^{\Lbb}B_{\mfrak \square}$ a finite projective $B_{\mfrak\square}$-module.
    Then, by \cite[Lemma 2.7.4]{Mann22}, there exists $f \in B\setminus \mfrak$ such that $M\otimes_{B_{\square}}^{\Lbb}B_{f\square}$ is a finite projective $B_{f\square}$-module.
    Therefore, there exists an open cover $\Spec(B(\ast))=\bigcup_{i=1}^n D(f_i)$ such that $M\otimes_{B_{\square}}^{\Lbb}B_{f_i\square}$ is a finite projective $B_{f_i\square}$-module.
    Since $M$ can be written as a finite limit of $M\otimes_{B_{\square}}^{\Lbb}B_{f_{i_1}\cdots f_{i_m}\square}$, $M$ is relatively discrete over $B_{\square}$. 
    Therefore, $M$ is a perfect complex over $B_{\square}$ by \cite[Proposition 9.2]{CC}, and $M$ is a finite projective $B_{\square}$-module by Lemma \ref{lem:finite projective module criterion}.

    Let us prove that $M\otimes_{B_{\square}}^{\Lbb}B_{\mfrak\square}$ is a finite projective $B_{\mfrak\square}$-module.
    Since $A$ is a Jacobson ring by \cite[Proposition 3.3(3)]{Lou17}, $\nfrak$ is a maximal ideal of $A$.
    Moreover, $B/\mfrak$ is relatively discrete over $A/\nfrak$, and is a finite extension of $A/\nfrak$.
    Therefore, $B/\mfrak$ is the condensed ring associated with a complete non-archimedean field.
    Hence, we may assume that $B$ is a noetherian local ring with a maximal ideal $\mfrak$ such that $B/\mfrak$ is a condensed ring associated with a complete non-archimedean field.
    Since $(B/\mfrak)_{\square}$ is Fredholm, $M\otimes_{B_{\square}}^{\Lbb}(B/\mfrak)_{\square}$ is a finite free $(B/\mfrak)_{\square}$-module.
    We take an isomorphism $M\otimes_{B_{\square}}^{\Lbb}(B/\mfrak)_{\square} \overset{\sim}{\to} (B/\mfrak)^n$.
    We have an equivalence
    $$R\intHom_{B}(M,\cofib(B^n\to (B/\mfrak)^n))\simeq M^{\vee}\otimes_{B_{\square}}^{\Lbb}\cofib(B^n\to (B/\mfrak)^n)$$
    from the condition (1), and it belongs to $\Dcal(B_{\square})^{\leq -1}$ from the condition (2).
    Hence, the morphism
    $$M\to M\otimes_{B_{\square}}^{\Lbb}(B/\mfrak)_{\square} \to (B/\mfrak)^n$$ has a lift $M\to B^n$, which has a retraction.
    Therefore, there exists a $B_{\square}$-module $L$ such that $M\cong L\oplus B^n$ and $L\otimes_{B_{\square}}^{\Lbb}(B/\mfrak)_{\square}=L\otimes_{B_{\square}}^{\Lbb}B/\mfrak=0$.
    Moreover, $L$ satisfies the conditions (1) and (2) in the theorem.
    We show $L=0$.
    Since $L$ is compact and connective, there exists an extremally disconnected set $S$ and a morphism $B_{\square}[S]\to L$ such that $B_{\square}[S]\to H^0(L)$ is surjective, where we note that $B_{\square}[S]$ is static by Lemma \ref{lem:compact projective object and prod}.
    For any $N\in \Dcal(B_{\square})^{\leq 0}$, $R\intHom_{B}(L,N)=L^{\vee}\otimes_{B_{\square}}^{\Lbb}N \in \Dcal(B_{\square})^{\leq 0}$ by the conditions (1) and (2).
    Therefore, the morphism $B_{\square}[S]\to L$ admits a section.
    In particular, $L$ is static, and by Lemma \ref{lem:compact projective object and prod} we obtain an injection $L\to \prod_I B$ for a set $I$ such that $\Zbb_{\square}[S]\cong\prod_I \Zbb$.
    Moreover, we have $L\otimes_{B_{\square}}B/\mfrak=0$, and by induction we obtain $L\otimes_{B_{\square}}B/\mfrak^n=0$ for $n\geq 1$.
    Since $B(\ast)$ is noetherian, the natural morphism $B(\ast)\to \prod_{n}B(\ast)/\mfrak^n$ is injective by Krull's intersection theorem. 
    Therefore, $B=\Cond_A(B(\ast)) \to \Cond_A(\prod_{n}B(\ast)/\mfrak^n)$ is injective, and 
    $$B\to \Cond_A\left(\prod_{n}B(\ast)/\mfrak^n\right)\to \prod_n \Cond_A\left(B(\ast)/\mfrak^n \right)=\prod_n B/\mfrak^n$$ 
    is also injective by Lemma \ref{lem:condensification functor}, where we note $B/\mfrak^n = \Cond_A(B(\ast)/\mfrak^n)$.
    We have the following commutative diagram:
    $$
    \xymatrix{
        L\ar[r]\ar@{^{(}->}[d] &\prod_n \left(L\otimes_{B_{\square}}B/\mfrak^n\right)\ar[d]\\
        \prod_I B \ar@{^{(}->}[r] & \prod_n \left(\prod_I B/\mfrak^n\right).
            }
    $$
    From the above, the morphism $L\to \prod_n \left(L\otimes_{B_{\square}}B/\mfrak^n\right)=0$ is injective, and therefore we obtain $L=0$.
\end{proof}

\begin{definition}
    Let $\Acal$ be an analytic animated ring.
    An object $M\in \Dcal(\Acal)$ is \textit{flat over $\Acal$} if $M$ is bounded above and for every static object $N\in \Dcal(\Acal)^{\heart}$, $M\otimes_{\Acal}^{\Lbb}N$ is also static.
    Moreover $M\in \Dcal(\Acal)$ is \textit{faithfully flat over $\Acal$} if $M$ is flat and the functor $M\otimes_{\Acal}^{\Lbb}- \colon \Dcal(\Acal) \to \Dcal(\Acal)$ is conservative.
\end{definition}

We prove the following lemmas for later use.

\begin{lemma}\label{lem:faithfully flat quotient}
    Let $\Acal$ be an analytic animated ring, and $B$ be an animated $\Acal$-algebra.
    Then $B$ is faithfully flat over $\Acal$ if and only if $\cofib(\underline{\Acal}\to B)\in \Dcal(\Acal)$ is flat over $\Acal$.
\end{lemma}
\begin{proof}
    We write $B/\underline{\Acal}=\cofib(\underline{\Acal}\to B)$.
    First, we prove the if direction.
    We take a static object $N\in \Dcal(\Acal)^{\heart}$.
    Then we have a fiber sequence
    \begin{align}\label{eq4}
        N \to N\otimes_{\Acal}^{\Lbb} B \to N\otimes_{\Acal}^{\Lbb}B/\underline{\Acal}.
    \end{align}
    Since $N$ and $N\otimes_{\Acal}^{\Lbb}B/\underline{\Acal}$ are static, $N\otimes_{\Acal}^{\Lbb} B$ is also static.
    Therefore, $B$ is flat over $\Acal$.
    We prove faithfulness.
    Since $B$ is flat over $\Acal$, it is enough to show that for every static object $N\in \Dcal(\Acal)^{\heart}$, $N\otimes_{\Acal}^{\Lbb}B=0$ implies $N=0$.
    It follows from the fiber sequence (\ref{eq4}), since $H^{-1}(N\otimes_{\Acal}^{\Lbb}B/\underline{\Acal})=0$ for a static object $N\in \Dcal(\Acal)^{\heart}$.

    Next, we prove the only if direction.
    It is clear that $B/\underline{\Acal}$ is bounded above.
    We take a static object $N\in \Dcal(\Acal)^{\heart}$.
    From the fiber sequence (\ref{eq4}), it is enough to show that $N \to N\otimes_{\Acal}^{\Lbb} B$ is injective.
    Since $B$ is faithfully flat over $\Acal$, it reduces to showing that $N\otimes_{\Acal}^{\Lbb} B \to N\otimes_{\Acal}^{\Lbb} B\otimes_{\Acal}^{\Lbb} B$ is injective.
    It follows from the fact that $B\to B\otimes_{\Acal}^{\Lbb} B$ has a retraction.
\end{proof}

\begin{lemma}\label{lem:faithfully flat colimit}
    Let $\Acal_0\to \Acal_1 \to \cdots \to \Acal_n \to\cdots $ be a diagram of analytic animated rings such that the analytic animated ring structure on $\Acal_n$ is induced from $\Acal_0$ for $n\geq 0$.
    We write $\Acal_{\infty}=\varinjlim_n \Acal_n$.
    \begin{enumerate}
        \item Let $M$ be an object of $\Dcal(\Acal_{\infty})$.
        If $M$ is flat over $\Acal_n$ for all $n$, then $M$ is also flat over $\Acal_{\infty}$.
        \item Let $B$ be an animated $\Acal_{\infty}$-algebra.
        If $B$ is faithfully flat over $\Acal_n$ for all $n$, then $B$ is also faithfully flat over $\Acal_{\infty}$.
    \end{enumerate}
\end{lemma}
\begin{proof}
    For every static object $N\in \Dcal(\Acal)^{\heart}$, we have $N\otimes_{\Acal_{\infty}}^{\Lbb}M\simeq \varinjlim_{n} N\otimes_{\Acal_{n}}^{\Lbb}M$, which proves (1).
    We prove (2).
    By Lemma \ref{lem:faithfully flat quotient}, it is enough to show that for a static object $N\in \Dcal(\Acal_{\infty})^{\heart}$, $N\otimes_{\Acal_{\infty}}^{\Lbb} \underline{\Bcal}/\underline{\Acal_{\infty}}$ is static, where $\underline{\Bcal}/\underline{\Acal_{\infty}}=\cofib(\underline{\Acal_{\infty}}\to\underline{\Bcal})$.
    We have the following equivalences:
    \begin{align*}
        N\otimes_{\Acal_{\infty}}^{\Lbb} \underline{\Bcal}/\underline{\Acal_{\infty}}&\simeq \varinjlim_{n} N\otimes_{\Acal_{n}}^{\Lbb} \underline{\Bcal}/\underline{\Acal_{\infty}}\\
        &\simeq \varinjlim_{n} \varinjlim_{m\geq n} N\otimes_{\Acal_{n}}^{\Lbb} \underline{\Bcal}/\underline{\Acal_{m}}\\
        &\simeq \varinjlim_{n} N\otimes_{\Acal_{n}}^{\Lbb} \underline{\Bcal}/\underline{\Acal_{n}}.
    \end{align*}
    By Lemma \ref{lem:faithfully flat quotient}, $\underline{\Bcal}/\underline{\Acal_{n}}$ is flat over $\Acal_n$, so we find that $ N\otimes_{\Acal_{\infty}}^{\Lbb} \underline{\Bcal}/\underline{\Acal_{\infty}}$ is static.
\end{proof}

\begin{proposition}\label{prop:ff descent}
    Let $\Acal$ be an analytic animated ring, and $B$ be an animated faithfully flat $\Acal$-algebra. 
    Let $\Bcal$ denote the condensed animated ring $B$ with the analytic ring structure induced from $\Acal$.
    Then $\Acal\to\Bcal$ satisfies $\ast$-descent.
\end{proposition}
\begin{proof}
    It follows from the proof of \cite[Theorem 2.8.18(i)]{Mann22}.
    We note that the assumption used in loc. cit. that $\Acal\to\Bcal$ is steady is not necessary because the analytic ring structure of $\Bcal$ is induced from $\Acal$.
\end{proof}

\begin{theorem}\label{thm:descent of finite projective modules}
    Let $A$ be a Jacobson-Tate algebra, and $B$ be a (condensed) animated relatively discrete $A$-algebra.
    Let $C$ be an animated faithfully flat $B_{\square}$-algebra.
    For $M\in \Dcal(B_{\square})$, if $M\otimes_{B_{\square}}^{\Lbb}C_{\square}$ is a finite projective $C_{\square}$-module, then $M$ is also a finite projective $B_{\square}$-module.
\end{theorem}
\begin{proof}
    By Proposition \ref{prop:ff descent} and Proposition \ref{prop:descent finite projective Fredholm} (1), $M$ is dualizable over $B_{\square}$.
    Moreover $M$ belongs to $\Dcal(B_{\square})^{\leq 0}$ since $C$ is faithfully flat over $B_{\square}$ and $M\otimes_{B_{\square}}^{\Lbb}C\cong M\otimes_{B_{\square}}^{\Lbb}C_{\square}$ belongs to $\Dcal(C_{\square})^{\leq 0}$.
    By applying the same argument to $M^{\vee}$, we find that $M^{\vee}$ also belongs to $\Dcal(B_{\square})^{\leq 0}$.
    Therefore, the claim follows from Theorem \ref{thm:finite projective non-Fredholm}.
\end{proof}

We provide examples of (faithfully) flat modules.
\begin{example}\label{ex:flat module}
    \begin{enumerate}
        \item Let $V$ be a Banach $\Qbb_p$-module in the usual sense.
        Then the associated condensed $\Qbb_{p,\square}$-module $\underline{V}$ is flat over $\Qbb_{p,\square}$.
        It follows from \cite[Lemma 3.21]{RJRC22}.
        \item For every profinite set $S$, $\Qbb_{p,\square}[S]$ is flat over $\Qbb_{p,\square}$ by \cite[Lemma 3.21]{RJRC22}.
        In particular, for a static $\Qbb_{p,\square}$-algebra $A$, $$A_{\square}[S]= A[S]\otimes_{\Qbb_{p}}^{\Lbb} \Qbb_{p,\square}\cong A\otimes_{\Qbb_{p,\square}}^{\Lbb} \Qbb_{p,\square}[S]\cong A\otimes_{\Qbb_{p,\square}} \Qbb_{p,\square}[S]$$ is static, where the second isomorphism follows from the same argument as in the proof of \cite[Lemma 3.25]{And21}.
        \item Let $A$ be a noetherian Banach $\Qbb_p$-algebra.
        We regard $A$ as a condensed ring, and let $V$ be a relatively discrete (static) $A_{\square}$-module.
        Then $V$ is flat over $\Qbb_{p,\square}$.
        In fact, $V$ can be written as a filtered colimit of relatively discrete and finitely presented $A_{\square}$-modules.
        Since every relatively discrete and finitely presented $A_{\square}$-module is associated with a Banach $\Qbb_p$-module, it is flat over $\Qbb_{p,\square}$.
        \item Let $A$ be a Banach $\Qbb_p$-algebra and $I$ be a set.
        We take a ring of definition $A_0$ of $A$, and we define $\hat{\bigoplus_I}A$ as $\hat{\bigoplus_I}A=(\bigoplus_I A_0)^{\wedge}_p [1/p]$, where $(-)^{\wedge}_p$ stands for the $p$-adic completion.
        Then it is flat over $A_{\square}$.
        It follows from (1) and the isomorphism $\hat{\bigoplus_I}A \cong \left(\hat{\bigoplus_I}\Qbb_p\right)\otimes_{\Qbb_{p,\square}}^{\Lbb} A$, which can be proved by using \cite[Proposition 2.12.10]{Mann22} (cf.\ the proof of \cite[Proposition 2.14]{Mikami23}).
    \end{enumerate} 
\end{example}


\subsection{Continuous actions and group cohomology}
In this subsection, we explain ``continuous action'' and ``group cohomology'' in the context of condensed mathematics.
Let $G$ be a locally profinite group, and let $A$ be a (static) $\Qbb_{p,\square}$-algebra.
Then $A_{\square}[G]=\varinjlim_{S\subset G}A_{\square}[S]$, where the colimit is taken over all compact open subsets $S$ of $G$, has a natural (non-commutative) condensed ring structure defined from the group structure of $G$.
\begin{lemma}
    Let $M$ be a (static) $A_{\square}$-module.
    Then an $A_{\square}[G]$-module structure on $M$ is equivalent to an $A$-linear action $\underline{G}\times M \to M$ of the condensed group $\underline{G}$ associated to $G$ on the $A$-module $M$.
\end{lemma}
\begin{proof}
    It can be proved in the same way as in \cite[Lemma 4.19]{RJRC22}.
\end{proof}

\begin{remark}
    We assume that $A$ is a Banach $\Qbb_p$-algebra.
    Let $M$ be a Banach $A$-module.
    Then a continuous $A$-linear $G$-action on $M$ in the usual sense (i.e., a continuous morphism $G\times M\to M$ satisfying usual conditions) is equivalent to an $A$-linear action $\underline{G}\times \underline{M} \to \underline{M}$ of the condensed group $\underline{G}$ on the condensed $\underline{A}$-module $\underline{M}$.
\end{remark}

\begin{definition}
    We regard $A_{\square}[G]$ as an $\Ebb_1$-algebra object in $\Dcal(\Qbb_{p,\square})$.
    We define $\Rep_{A_{\square}}(G)$ as the $\infty$-category of left $A_{\square}[G]$-module objects in $\Dcal(\Qbb_{p,\square})$.
    When $A=\Qbb_p$, then we simply write $\Rep_{\square}(G)$ for $\Rep_{\Qbb_{p,\square}}(G)$.
    Let $M$ be an object of $\Dcal(A_{\square})$.
    A \textit{continuous $A$-linear $G$-action} on $M$ is a lift of $M$ along the forgetful functor $\Rep_{A_{\square}}(G)\to \Dcal(A_{\square})$.
    We often refer to an object of $\Rep_{A_{\square}}(G)$ as an object of $\Dcal(A_{\square})$ with a continuous $A$-linear $G$-action.
    If there is no room for confusion, we often omit ``$A$-linear'' from our terminology.
\end{definition}

\begin{remark}
    The $\infty$-category $\Rep_{\square}(G)$ becomes a symmetric monoidal $\infty$-category by defining the tensor product of $V,W\in \Rep_{\square}(G)$ as $V\otimes_{\Qbb_{p,\square}}^{\Lbb}W$ with the diagonal $G$-action.
    Then, $A$ with the trivial $G$-action can be regarded as an $\Ebb_{\infty}$-algebra object in $\Rep_{\square}(G)$.
    Moreover, there is a natural equivalence of $\infty$-categories
    $$\Mod_{A}(\Rep_{\square}(G))\cong \Rep_{A_{\square}}(G),$$
    where $\Mod_{A}(\Rep_{\square}(G))$ is the $\infty$-category of $A$-module objects in $\Rep_{\square}(G)$.
\end{remark}

\begin{definition}[{\cite[Definition 5.1]{RJRC22}}]
    Let $M$ be an object of $\Dcal(A_{\square})$ with a continuous $G$-action.
    We define the \textit{continuous group cohomology} $R\Gamma(G,M)$ of $M$ as $R\Gamma(G,M)=R\intHom_{A_{\square}[G]}(A,M)\in\Dcal(A_{\square})$, where we regard $A$ as an object of $\Rep_{A_{\square}}(G)$ by the trivial action. 
    If $M$ is static, then we write 
    $$M^G=\Gamma(G,M)=H^0(R\Gamma(G,M))=\intHom_{A_{\square}[G]}(A,M)$$ 
    for the submodule of $G$-invariant elements.
\end{definition}

\begin{remark}
    We assume that $A$ is a Banach $\Qbb_p$-algebra.
    Let $M$ be a Banach $A$-module with a continuous $G$-action in the usual sense.
    By \cite[Lemma 5.2(1)]{RJRC22}, the condensed continuous group cohomology $H^i(R\Gamma(G,M))$ defined above is isomorphic to the condensed $A_{\square}$-module associated to the usual continuous cohomology $H^i_{\cont}(G,M)$ defined from the complex
    $$0\to M\to C(G,M) \to C(G^2,M)\to\cdots.$$
\end{remark}

\begin{proposition}\label{prop:compact nuclear group cohomology}
We assume that $G$ is a compact $p$-adic Lie group over $\Qbb_p$.
Let $M$ be an object of $\Dcal(A_{\square})$ with a continuous $G$-action.
If $M$ is compact (resp.\ nuclear) as an object of $\Dcal(A_{\square})$, then $R\Gamma(G,M)$ is also compact (resp.\ nuclear) as an object of $\Dcal(A_{\square})$.
\end{proposition}
\begin{proof}
    We note that $G$ has an open normal uniform pro-$p$ subgroup $G_0$ and that $R\Gamma(G,M)$ is equivalent to $R\Gamma(G/G_0,R\Gamma(G_0,M))$.
    Therefore we may assume that $G$ is a uniform pro-$p$ group or a finite group.
    First, we assume that $G$ is a uniform pro-$p$ group of dimension $d$.
    Then we have the Lazard-Serre projective resolution 
    $$0\to A_{\square}[G]^{\binom{d}{d}} \to \cdots \to A_{\square}[G]^{\binom{d}{i}}\to \cdots \to A_{\square}[G]^{\binom{d}{0}}\to 0$$
    of $A$ with the trivial $G$-action by \cite[Theorem 5.7]{RJRC22}.
    Therefore, $R\Gamma(G,M)$ can be written as a finite limit of $M$.
    Since compact objects (resp.\ nuclear objects) of $\Dcal(A_{\square})$ are stable under finite limits, it proves the proposition.

    Next, we assume that $G$ is a finite group.
    Then $\Qbb_p$ with the trivial $G$-action is a direct summand of $\Qbb_{p,\square}[G]$.
    Therefore, $A$ with the trivial $G$-action is a direct summand of $A_{\square}[G]$, so $R\Gamma(G,M)$ is a direct summand of $M$.
    Since compact objects (resp.\ nuclear objects) of $\Dcal(A_{\square})$ are stable under direct summands, it proves the proposition.
\end{proof}

\begin{lemma}\label{lem:fixed part and tensor products}
    We assume that $G$ is a compact $p$-adic Lie group over $\Qbb_p$.
    Let $V$ be an object of $\Dcal(A_{\square})$ with a continuous $G$-action, and $W$ be an object of $\Dcal(A_{\square})$.
    The (condensed) group $G$ acts on $V\otimes_{A_{\square}}^{\Lbb}W$ by $g\otimes \id$ for $g\in G$.
    Then we have a natural equivalence 
    $$R\Gamma(G,V\otimes_{A_{\square}}^{\Lbb}W)\simeq R\Gamma(G,V)\otimes_{A_{\square}}^{\Lbb}W.$$
    In particular, if $V$ is static and $W$ is static and flat over $A_{\square}$, then we have a natural isomorphism
    $$\Gamma(G,V\otimes_{A_{\square}}W)\cong \Gamma(G,V)\otimes_{A_{\square}}W.$$
\end{lemma}
\begin{proof}
    For the same reason as in Proposition \ref{prop:compact nuclear group cohomology}, we may assume that $G$ is a uniform pro-$p$ group or a finite group.
    If $G$ is a uniform pro-$p$ group, then the claim easily follows from the Lazard-Serre resolution \cite[Theorem 5.7]{RJRC22}.  
    If $G$ is a finite group, then the claim follows from the fact that $A$ with the trivial $G$-action is a direct summand of $A_{\square}[G]$
\end{proof}

We also explain semilinear $G$-actions in the condensed setting.
We endow $A$ with an action $\rho \colon G\times A \to A$ of the condensed group $G$ by $\Qbb_{p,\square}$-algebra automorphisms.
Then we define the twisted group ring structure on $A_{\square}[G]$, and let $A_{\square}^t[G]$ denote this ring to distinguish it from the usual group ring $A_{\square}[G]$.
Roughly speaking, the twisted group ring structure is given by $ga=\rho(g, a)g$ for $g\in G$ and $a\in A$.

\begin{lemma}\label{lem:semilinear action and twisted group ring}
    Let $M$ be a (static) $A_{\square}$-module.
    Then an $A_{\square}^t[G]$-module structure on $M$ is equivalent to an $A$-semilinear action $\underline{G}\times M \to M$ of the condensed group $G$ on the $A$-module $M$.
\end{lemma}
\begin{proof}
    It can be proved in the same way as in \cite[Lemma 4.19]{RJRC22}.
\end{proof}

\begin{definition}
    We regard $A_{\square}^t[G]$ as an $\Ebb_1$-algebra object in $\Dcal(\Qbb_{p,\square})$.
    We define $\Rep_{A_{\square}^t}(G)$ as the $\infty$-category of left $A_{\square}^t[G]$-module objects in $\Dcal(\Qbb_{p,\square})$.
    Let $M$ be an object of $\Dcal(A_{\square})$.
    A continuous $A$-semilinear $G$-action on $M$ is a lift of $M$ along the forgetful functor $\Rep_{A_{\square}^t}(G)\to \Dcal(A_{\square})$.
    We often refer to an object $\Rep_{A_{\square}^t}(G)$ as an object of $\Dcal(A_{\square})$ with a continuous $A$-semilinear $G$-action.
\end{definition}

\begin{remark}
    We can regard $A$ as an $\Ebb_{\infty}$-algebra object in $\Rep_{\square}(G)$.
    Then the $\infty$-category $\Rep_{A_{\square}^t}(G)$ is equivalent to the $\infty$-category $\Mod_A(\Rep_{\square}(G))$ of $A$-module objects in $\Rep_{\square}(G)$.
    Let us prove this.
    Let $F_1\colon \Dcal(\Qbb_{p,\square})\to \Rep_{\square}(G)$ be the scalar extension functor along $\Qbb_p\to \Qbb_{p,\square}[G]$, and $F_2\colon \Rep_{\square}(G)\to \Mod_A(\Rep_{\square}(G))$ be the scalar extension functor along the morphism $\Qbb_p\to A$ of $\Ebb_{\infty}$-algebras in $\Rep_{\square}(G)$.
    Then, there is a $\Dcal(\Qbb_{p,\square})$-linear adjunction
    $$F=F_2\circ F_1\colon \Dcal(\Qbb_{p,\square})\rightleftarrows \Mod_A(\Rep_{\square}(G))\colon G,$$
    where $G$ is the forgetful functor.
    Since $G$ is conservative and preserves all small colimits, this adjunction is monadic by the Barr-Beck-Lurie theorem \cite[Theorem 4.7.3.5]{HA}.
    By an explicit calculation, the associated $\Ebb_1$-algebra object $G\circ F(\Qbb_p)$ in $\Dcal(\Qbb_{p,\square})$ is isomorphic to $A^t_{\square}[G]$.
    Since both functors $F$ and $G$ are $\Dcal(\Qbb_{p,\square})$-linear, the associated monad $G\circ F$ is given as $A^t_{\square}[G]\otimes_{\Qbb_{p,\square}}^{\Lbb}-\colon \Dcal(\Qbb_{p,\square})\to \Dcal(\Qbb_{p,\square})$ by \cite[Lemma A.4.7]{Mann22}.
    Therefore, we get an equivalence of $\infty$-categories $\Mod_A(\Rep_{\square}(G))\simeq \Rep_{A_{\square}^t}(G)$.
\end{remark}

From now on, we assume that $A$ is nuclear as a $\Qbb_{p,\square}$-module and that $G$ is compact.
Then the continuous $G$-action $\Qbb_{p,\square}[G]\otimes_{\Qbb_{p,\square}} A \to A$ corresponds to $$\rho\colon A \to C(G,A)\cong C(G,\Qbb_{p})\otimes_{\Qbb_{p,\square}}A\cong C(G,\Qbb_{p})\otimes_{\Qbb_{p,\square}}^{\Lbb}A$$ by adjunction, where we note that $C(G,\Qbb_{p})$ is flat over $\Qbb_{p,\square}$ (Example \ref{ex:flat module}), and it gives an action of analytic spaces in the sense of Clausen-Scholze
$$\AnSpec(C(G,\Qbb_p)_{\square})\times \AnSpec A_{\square} \to \AnSpec A_{\square}.$$
Then we get a solid $\Dcal$-stack $[\AnSpec A_{\square}/\AnSpec(C(G,\Qbb_p)_{\square})]$ (see \cite[Definition 3.2.7]{RC24} for the definition of solid $\Dcal$-stacks). 
The symmetric monoidal $\infty$-category $\Dcal([\AnSpec A_{\square}/\AnSpec(C(G,\Qbb_p)_{\square})])$ is described as
\begin{align*}
    \Dcal([\AnSpec A_{\square}/\AnSpec(C(G,\Qbb_p)_{\square})])\cong \varprojlim_{[n]\in \Delta}\Dcal(\AnSpec C(G^n,\Qbb_p)_{\square}\times \AnSpec A_{\square}),
\end{align*}
where $G^n$ is the $n$-fold product of $G$. 
This $\infty$-category is not necessarily equivalent to $\Rep_{A_{\square}^t}(G)$.
It is observed from the fact that for an object $M\in \Dcal(\Qbb_{p,\square})$, the natural morphism $$C(G,\Qbb_{p})\otimes_{\Qbb_{p,\square}}^{\Lbb}M \to R\intHom_{\Qbb_p}(\Qbb_{p,\square}[G],M)$$ is not necessarily an equivalence.
Therefore, we expect that there is an equivalence of $\infty$-categories 
$$\Dcal([\AnSpec A_{\square}/\AnSpec(C(G,\Qbb_p)_{\square})])^{\nuc}\simeq \Rep_{A_{\square}^t}(G)^{\nuc},$$
where $(-)^{\nuc}$ stands for the subcategory consisting of objects that are nuclear as objects of $\Dcal(A_{\square})$ (equivalently, nuclear as objects of $\Dcal(\Qbb_{p,\square})$ (\cite[Lemma 2.9]{Mikami23})).
For our purpose, it is enough to consider static modules, so we prove the above equivalence for the subcategories of static objects and compare their cohomology.

\begin{lemma}\label{lem:stacky definition}
    Let $\Rep_{A_{\square}^t}(G)^{\nuc,\heart}$ denote the subcategory of $\Rep_{A_{\square}^t}(G)$ consisting of nuclear and static objects, and $\Dcal([\AnSpec A_{\square}/\AnSpec(C(G,\Qbb_p)_{\square})])^{\nuc,\heart}$ denote the subcategory of $\Dcal([\AnSpec A_{\square}/\AnSpec(C(G,\Qbb_p)_{\square})])$ consisting of objects which are static and nuclear as objects of $\Dcal(A_{\square})$.
    Then we have an equivalence of symmetric monoidal categories $$\Rep_{A_{\square}^t}(G)^{\nuc,\heart} \simeq \Dcal([\AnSpec A_{\square}/\AnSpec(C(G,\Qbb_p)_{\square})])^{\nuc,\heart}.$$
\end{lemma}
\begin{proof}
To simplify the notation, we write 
$$\Dcal(A_{\square}/\underline{G})=\Dcal([\AnSpec A_{\square}/\AnSpec(C(G,\Qbb_p)_{\square})]).$$
The $G$-action on $A$ gives a cosimplicial ring $\{C(G^n,A)\}_{[n]\in \Delta}$.
For a morphism $[0]\to [n]$ in $\Delta$, the corresponding morphism $A_{\square}\to C(G^n,A)_{\square}$ is flat, so the subcategory $\Dcal(A_{\square}/\underline{G})^{\heart}$ of $\Dcal(A_{\square}/\underline{G})$ consisting of objects which are static as objects of $\Dcal(A_{\square})$ is equivalent to the symmetric monoidal $\infty$-category $\varprojlim_{[n]\in\Delta} \Dcal({C(G^n,A)_{\square}})^{\heart}$, where we note $C(G^n,A)\cong C(G^n,\Qbb_{p})\otimes_{\Qbb_{p,\square}}^{\Lbb}A$.
This limit is equivalent to the limit in the $2$-category of symmetric monoidal categories, so it is equivalent to $\varprojlim_{[n]\in\Delta_{\leq 2}} \Dcal({C(G^n,A)_{\square}})^{\heart}$ by \cite[Proposition 4.3.5]{DAG8}.
Therefore, the symmetric monoidal category
$\Dcal(A_{\square}/\underline{G})^{\nuc,\heart}$ is equivalent to a symmetric monoidal category of (static) nuclear $A_{\square}$-modules $M$ with $\rho_M\colon M\to C(G,\Qbb_p)\otimes_{\Qbb_{p,\square}}M\cong C(G,M)$ satisfying the following conditions:
\begin{enumerate}
    \item The morphism $\rho_M$ is $A$-linear, where the $A$-module structure on $C(G,M)$ is given by $\rho\colon A\to C(G,A)$.
    \item The diagram
    $$
    \xymatrix{
        M\ar[r]^-{\rho_M}\ar[d]^-{\rho_M} & C(G,M)\ar[d]\\
        C(G,M)\ar[r] &C(G^2,M)
    }
    $$
    is commutative, where the right vertical morphism is induced by the multiplication morphism $G^2=G\times G \to G$ and the lower horizontal morphism is defined by $\phi\mapsto ((g,h)\mapsto g\phi(h))$.
    \item The composition morphism $M \overset{\rho_M}{\longrightarrow} C(G,M) \overset{e_1}{\longrightarrow}M$ is the identity, where $e_1 \colon C(G,M)\to M$ is the evaluation morphism at $1$.
\end{enumerate} 
A tensor product of $(M,\rho_M)$ and $(N,\rho_N)$ is given by
    \begin{align*}
        M\otimes_{\Qbb_{p,\square}}N \overset{\rho_M\otimes\rho_N}{\longrightarrow} &
        C(G,\Qbb_p)\otimes_{\Qbb_{p,\square}}M\otimes_{\Qbb_{p,\square}}C(G,\Qbb_p)\otimes_{\Qbb_{p,\square}}N \\
        \longrightarrow &C(G,\Qbb_p)\otimes_{\Qbb_{p,\square}}M\otimes_{\Qbb_{p,\square}}N,
    \end{align*}
where the second morphism is induced by the multiplication morphism 
$$C(G,\Qbb_p)\otimes_{\Qbb_{p,\square}}C(G,\Qbb_p)\to C(G,\Qbb_p).$$
A morphism $\rho_M \colon M \to C(G,M)$ corresponds to $\sigma_M\colon G\times M\to M$ by adjunction, and it is easy to show that $\rho_M$ satisfies the conditions (1), (2), and (3) above if and only if $\sigma_M$ is a continuous $A$-semilinear $G$-action.
Therefore, we obtain an equivalence of categories
$$\Rep_{A_{\square}^t}(G)^{\nuc,\heart} \simeq \Dcal(A_{\square}/\underline{G})^{\nuc,\heart}.$$
It is easy to check that this equivalence upgrades to an equivalence of symmetric monoidal categories.
\end{proof}
\begin{remark}
    If $G$ is a compact $p$-adic Lie group, then we can prove that there is an equivalence of symmetric monoidal $\infty$-categories 
    $$\Rep_{A_{\square}^t}(G)^{\nuc}\simeq\Dcal([\AnSpec A_{\square}/\AnSpec(C(G,\Qbb_p)_{\square})])^{\nuc}$$
    by the similar argument as in \cite[Proposition 2.46]{Mikami25}.
    We omit the details.
\end{remark}

\begin{corollary}\label{cor:stacky cohomology comparison}
    Let $M$ be a nuclear static $A_{\square}$-module equipped with a continuous $A$-semilinear $G$-action, and $\{M_n\}_{[n]\in\Delta}\in \Dcal([\AnSpec A_{\square}/\AnSpec(C(G,\Qbb_p)_{\square})])^{\nuc,\heart}$ be the object corresponding to $M$ under the equivalence in Lemma \ref{lem:stacky definition}.
    Then we have an equivalence $R\Gamma(G,M)\cong \Rvarprojlim_{[n]\in \Delta}M_n$
    in $\Dcal(\Qbb_{p,\square})$.
\end{corollary}

\begin{proof}
    We have the Bar resolution $\{\Qbb_{p,\square}[G^{n+1}]\}_{[n]\in\Delta}$ and there is an equivalence 
    $$R\Gamma(G,M)\simeq \varprojlim_{[n]\in\Delta}R\intHom_{\Qbb_{p,\square}[G]}(\Qbb_{p,\square}[G^{n+1}],M).$$
    Since $M$ is nuclear, the cosimplicial object $\{R\intHom_{\Qbb_{p,\square}[G]}(\Qbb_{p,\square}[G^{n+1}],M)\}_{[n]\in \Delta}$ is equivalent to $\{M_n\}_{[n]\in\Delta}$.
\end{proof}

For later use, we prove lemmas about semilinear representations on relatively discrete $A$-modules.
We assume that $A$ is a noetherian Banach $\Qbb_p$-algebra, and continue to assume that $G$ is compact.


\begin{proposition}\label{prop:relatively discrete action filtered colimit}
Let $M$ be a relatively discrete (static) $A$-module with a continuous $A$-semilinear $G$-action.
Then there exists a directed family $\{M_{\lambda}\}$ of finitely presented $A$-submodules which are stable under the $G$-action such that $M=\varinjlim_{\lambda}M_{\lambda}$.
\end{proposition}
\begin{proof}
Since $A$ is noetherian, $M$ can be written as a filtered colimit $M=\varinjlim_{\lambda} M_{\lambda}$ of finitely presented $A$-submodules $M_{\lambda}$ of $M$.
It is enough to show that for each $\lambda$, $M_{\lambda}$ is contained in some finitely presented $A$-submodule which is stable under the $G$-action.
Since $M_{\lambda}$ is a finitely presented $A$-module and $A^t_{\square}[G]$ is a compact left $A_{\square}$-module, $A_{\square}^t[G] \otimes_{A_{\square}} M_{\lambda}$ is also a compact (left) $A_{\square}$-module (in the category of static $A_{\square}$-modules).
Therefore, the image of 
$$A_{\square}^t[G] \otimes_{A_{\square}} M_{\lambda} \to A_{\square}^t[G] \otimes_{A_{\square}} M \to M$$
is contained in $M_{\mu}$ for some $\mu\geq \lambda$, where $A_{\square}^t[G]$ is a twisted group ring and 
$$A_{\square}^t[G] \otimes_{A_{\square}} M \to M$$
is the morphism associated to the continuous semilinear $\Gamma_K$-action (Lemma \ref{lem:semilinear action and twisted group ring}).
We set $$M_{\lambda}^{\prime}=\Im(A_{\square}^t[G] \otimes_{A_{\square}} M_{\lambda} \to M)\subset M_{\mu},$$ 
which is a $\Gamma_K$-stable submodule.
By Lemma \ref{lem:rd sub}, $M_{\lambda}^{\prime}$ is relatively discrete over $A$ and finitely presented.
\end{proof}

Finally, we prove a classical result of the extension of an automorphism to the action of $\Zbb_p$.
We assume $G=\Zbb_p$.
To avoid confusion, we write $\gamma\in G$ for $1\in \Zbb_p=G$.
Let $G_n$ denote the open subgroup topologically generated by $\gamma^{p^n}$ (i.e., $G_n=p^n\Zbb_p$).
Let $A$ be a Banach $\Qbb_p$-algebra with a continuous $G$-action.
Then, we can take a ring of definition $A_0\subset A$ which is stable under the $G$-action.
\begin{proposition}\label{prop:extension of automorphism to the group action}
    Let $n\geq 1$ be an integer.
    For $\alpha=1+(p\alpha_{ij})_{ij} \in 1+pM_n(A_0)$, it defines the $\gamma$-semilinear automorphism 
    $f_{\alpha}\colon A^n \to A^n$
    which is given by $f_{\alpha}(e_i)=\alpha_{1i}e_1+\cdots+\alpha_{ni}e_n$ where $e_1,\ldots, e_n$ is the standard basis of $A^n$.
    Then, it (uniquely) extends to a continuous semilinear $G$-action on $A^n$.
\end{proposition}
\begin{proof}
    It is enough to show that for every $m\geq 0$, the automorphism
    $f_{\alpha}\colon A_0^n/p^m \to A_0^n/p^m$
    extends to a smooth $G$-action. 
    There exists an open subgroup $G_r\subset G$ such that the images of $\alpha_{ij}$ in $A_0/p^m$ are fixed by $G_r$.
    Since $G$ is abelian, for every $g \in G$, the image of $g(\alpha_{i,j})$ in $A_0/p^m$ is also fixed by $G_r$.
    Let $B\subset A_0/p^m$ be the subalgebra generated by $g(\alpha_{i,j})$.
    Then the $G_r$-action on $B$ is trivial.
    Therefore, the morphism $f_{\alpha}^{p^r} \colon B^n \to B^n$ is $B$-linear. 
    Moreover, it corresponds to a matrix in $1+pM_n(B)$, so $f_{\alpha}^{p^{r+m-1}} \colon B^n \to B^n$ is the identity.
    Therefore, $f_{\alpha} \colon B^n \to B^n$ defines an action of $G/G_{r+m-1}$.
    In particular, it extends to a smooth action of $G$.
    Since the action of $G$ on $A_0/p^m$ is smooth, $f_{\alpha}\colon A_0^n/p^m \to A_0^n/p^m$ also extends to a smooth action of $G$.
\end{proof}

\begin{corollary}\label{cor:free module lift of action}
    Let $\Gamma$ be a compact open subgroup of $\Zbb_p^{\times}$, and $A$ be a noetherian Banach $\Qbb_p$-algebra with a continuous $\Gamma$-action.
    Let $M$ be a finitely presented $A$-module with a continuous semilinear $\Gamma$-action, where we endow $M$ with the natural topology.
    Then there is a finite free $A$-module $L$ with a continuous semilinear $\Gamma$-action and a $\Gamma$-equivariant $A$-linear surjection $L\to M$. 
\end{corollary}
\begin{proof}
    For $m$ sufficiently large, $1+p^m\Zbb_p\subset \Zbb_p^{\times}$ is contained in $\Gamma$, and let $\Gamma_m$ denote this compact open subgroup.
    We write $\gamma_m$ for $1+p^m\in \Gamma_m=1+p^m\Zbb_p$ to avoid confusion.
    We note that $\Gamma_m\cong \Zbb_p$ and $\gamma_m$ is a topological generator of $\Gamma_m$.
    We choose a ring of definition $A_0\subset A$ which is stable under the $\Gamma$-action, and choose a generator $x_1,\ldots,x_n\in M$ of $M$ as an $A$-module.
    Then for $m$ sufficiently large, there exists $\alpha=(\alpha_{ij})_{ij}\in 1+pM_n(A_0)$ such that $\gamma_m(x_i)=\alpha_{1i}x_1+\cdots+\alpha_{ni}x_n.$
    By Proposition \ref{prop:extension of automorphism to the group action}, $\alpha$ defines a continuous semilinear $\Gamma_m$-action on $A^n$ and the natural morphism $A^n \to M;\;e_i\mapsto x_i$ is a $\Gamma_m$-equivariant $A$-linear surjection.
    We set $L=A^{t}_{\square}[\Gamma]\otimes_{A^{t}_{\square}[\Gamma_m]}A^n$, where $A^{t}_{\square}[\Gamma]$ and $A^{t}_{\square}[\Gamma_m]$ are the twisted group rings.
    Then there is a $\Gamma$-equivariant $A$-linear surjection $L\to M$.
    Since $A^{t}_{\square}[\Gamma]$ is a finite free right $A^{t}_{\square}[\Gamma_m]$-module, $L$ is finite free as an $A$-module.
\end{proof}


\subsection{Locally analytic representations}
In this subsection, we prove some auxiliary results about locally analytic representations.
We use the theory of solid locally analytic representations introduced in \cite{RJRC22,RJRC23}.
For the foundation of solid locally analytic representations, see loc. cit.

In this subsection, let $G$ be a compact $p$-adic Lie group over $\Qbb_p$, and $G_0 \subset G$ be a compact open subgroup equipped with an integer-valued, saturated $p$-valuation.
Let $G_n$ denote the group $(G_0)^{p^n}$ for $n\geq 0$.
Then we can regard $G_n$ as a rigid analytic variety $\Gbb_n$ over $\Spa\Qbb_p$, and let $C^{\an}(\Gbb_n,\Qbb_p)$ denote the Banach $\Qbb_p$-algebra of analytic functions on $\Gbb_n$.
It has a natural continuous $G_n^2=G_n \times G_n$-action defined by left translations and right translations.
We define a rigid analytic variety $\Gbb^n=\coprod_{gG_n\in G/G_n} g \Gbb_n$.
Let $C^{n}(G,\Qbb_p)$ denote the Banach $\Qbb_p$-algebra $\Ocal(\Gbb^n)$ of rigid analytic functions on $\Gbb^n$. 
It also has a natural continuous $G^2=G \times G$-action defined by left translations and right translations.
For details, see \cite[4.1]{RJRC22}. 

\begin{notation}
    Let $n\in \Zbb_{>0}$ be a positive integer, and $N$ be an object of $\Dcal(\Qbb_{p,\square})$ with a continuous $G^n$-action.
    For a finite set $I\subset \{1,2,\ldots,n\}$, we have a group homomorphism $\iota_I\colon G\to G^n$ defined by $\iota_I(g)_j=g$ if $j\in I$ and $\iota_I(g)_j=1$ if $j\notin I$.
    By using $\iota_I$, we can define a continuous action of $N$, and we write it  as $N_{\star_{I}}$.
    For example, let $M$ be an object of $\Dcal(\Qbb_{p,\square})$ with a continuous $G^2$-action.
    Then $C(G,M)$ has a natural continuous $G^4$-action which is given by $$((g_1,g_2,g_3,g_4)\cdot f)(x)=(g_3,g_4)\cdot f(g_1^{-1}xg_2).$$
    In this case, the action on $C(G,M)_{\star_1}$ is given by $(g\cdot f)(x)=f(g^{-1}x),$ the action on $C(G,M)_{\star_{1,3}}$ is given by $(g\cdot f)(x)=(g,1)\cdot f(g^{-1}x),$ and the action on $C(G,M)_{\star_{1,3,4}}$ is given by $(g\cdot f)(x)=(g,g)\cdot f(g^{-1}x).$ 
    \end{notation}

\begin{definition}[{\cite[Definition 4.32, Definition 4.40]{RJRC22}}]
    Let $V$ be an object of $\Dcal(\Qbb_{p,\square})$ with a continuous $G$-action.
    \begin{enumerate}
        \item We write $C^{\bar{n}}(G,V)=C^{n}(G,\Qbb_p)\otimes_{\Qbb_{p,\square}}^{\Lbb}V$, which has the natural $G^3$-action.
        We have a natural $G^3$-equivariant morphism 
        \begin{align*}
        C^{\bar{n}}(G,V)&=C^{n}(G,\Qbb_p)\otimes_{\Qbb_{p,\square}}^{\Lbb}V\\
        &\to C(G,\Qbb_p)\otimes_{\Qbb_{p,\square}}^{\Lbb}V\\
        &\to C(G,V).
        \end{align*}
        \item 
        We define the space of \textit{derived $G_n$-analytic vectors} of $V$ as
        $$V^{G_n\mathchar`-R\an}=R\Gamma(G,C^{\bar{n}}(G,V)_{\star_{1,3}}),$$
        which has a $G$-action induced by the $\star_2$-action on $C^{\bar{n}}(G,V)$.
        We have a natural $G$-equivariant morphism 
        \begin{align*}
            V^{G_n\mathchar`-R\an}&=R\Gamma(G,C^{\bar{n}}(G,V)_{\star_{1,3}})\\
            &\to R\Gamma(G,C(G,V)_{\star_{1,3}})\simeq V,
        \end{align*}
        where the last equivalence follows from \cite[Proposition 4.25]{RJRC22}.
        We say $V$ is \textit{derived $G_n$-analytic} if the morphism $V^{G_n\mathchar`-R\an}\to V$ is an equivalence.
        If $V$ is static, then we define the space of \textit{(non-derived) $G_n$-analytic vectors} of $V$ as $$V^{G_n\mathchar`-\an}=H^0(V^{G_n\mathchar`-R\an})=\Gamma(G,C^{\bar{n}}(G,V)_{\star_{1,3}}).$$
        We say $V$ is \textit{non-derived $G_n$-analytic} if the morphism $V^{G_n\mathchar`-\an}\to V$ is an equivalence.
         \item We have a natural morphism $V^{G_n\mathchar`-\an}\to V^{G_{n+1}\mathchar`-\an}$.
         We define the space of \textit{derived $G$-locally analytic vectors} of $V$ as
         $$V^{G\mathchar`-R\la}=\varinjlim_n V^{G_n\mathchar`-R\an}\in \Rep_{\square}(G).$$
         We have a natural $G$-equivariant morphism $V^{G\mathchar`-R\la}\to V$.
         We say $V$ is \textit{$G$-locally analytic} if the morphism $V^{G\mathchar`-R\la}\to V$ is an equivalence.
         If $V$ is static, then we define the space of \textit{$G$-locally analytic vectors} of $V$ as $V^{G\mathchar`-\la}=H^0(V^{G\mathchar`-R\la})$.
    \end{enumerate}
\end{definition}

\begin{remark}
We can regard $C^{\bar{n}}(G,V)$ as a rigid analytic functions with values in $V$ on the canonical compactification $\bar{\Gbb^n}$ of $\Gbb^n$ in the sense of Huber.
\end{remark}

\begin{remark}
    We note that our definition of derived $G_n$-analytic vectors is slightly different from the definition in \cite[Definition 4.32]{RJRC22}.
    There, they use $$(C^{n}(G,\Qbb_p),C^{n}(G,\Qbb_p)^{\circ})_{\square}\otimes_{\Qbb_{p,\square}}^{\Lbb}V$$ instead of $C^{n}(G,\Qbb_p)\otimes_{\Qbb_{p,\square}}^{\Lbb}V$.
    If $V$ is nuclear as an object of $\Dcal(\Qbb_{p,\square})$ then these are equivalent, so our definition and the definition in loc. cit. coincide.

    By \cite[Lemma 2.1.9]{RJRC23}, our definition of derived locally analytic vectors and the definition of them in \cite{RJRC22} coincide.
\end{remark}

\begin{remark}
    Let $V$ be a static object of $\Dcal(\Qbb_{p,\square})$ with a continuous $G$-action.
    By \cite[Proposition 3.2.5]{RJRC23}, $V$ is $G$-locally analytic if and only if $V^{G\mathchar`-\la}\to V$ is an isomorphism.
    It is the reason why we do not use the term ``derived $G$-locally analytic''.
\end{remark}

\begin{remark}
    We have a natural $G$-equivariant equivalence $$R\intHom_{\Qbb_{p,\square}[G_n]}(\Qbb_{p,\square}[G],(C^{\an}(\Gbb_n,\Qbb_p)\otimes_{\Qbb_{p,\square}}^{\Lbb}V)_{\star_{1,3}})\simeq (C^{n}(G,\Qbb_p)\otimes_{\Qbb_{p,\square}}^{\Lbb}V)_{\star_{1,3}},$$
    where the left $\Qbb_{p,\square}[G_n]$-module structure on $\Qbb_{p,\square}[G]$ is given by $g\cdot x\coloneqq xg^{-1}$ for $g\in G_n$ and $x \in \Qbb_{p,\square}[G]$.
    Therefore, we have a natural $G_n$-equivalent equivalence $$V^{G_n\mathchar`-R\an}\simeq R\Gamma(G_n, (C^{\an}(\Gbb_n,\Qbb_p)\otimes_{\Qbb_{p,\square}}^{\Lbb}V)_{\star_{1,3}}).$$
\end{remark}

\begin{lemma}\label{lem:an is sub}
    Let $V$ be a flat nuclear $\Qbb_{p,\square}$-module with a continuous $G$-action.
    Then the natural morphism $V^{G_n\mathchar`-\an}\to V$ is injective.
    By using this, we regard $V^{G_n\mathchar`-\an}$ as a submodule of $V$.
\end{lemma}
\begin{proof}
    By definition, $V^{G_n\mathchar`-\an}$ is equal to $C^{\bar{n}}(G,V)^{G}$.
    On the other hand, we have an isomorphism $V\cong C(G,V)^{G}$ by \cite[Proposition 4.25]{RJRC22}.
    Therefore, it is enough to show that $f\colon C^{\bar{n}}(G,V)\to C(G,V)$ is injective.
    Since $V$ is a nuclear $\Qbb_{p,\square}$-module, the morphism $f$ is isomorphic to $C^{n}(G,\Qbb_{p})\otimes_{\Qbb_{p,\square}}V \to C(G,\Qbb_{p})\otimes_{\Qbb_{p,\square}}V$.
    It is injective since the morphism $C^{n}(G,\Qbb_{p})\to C(G,\Qbb_{p})$ is injective and $V$ is a flat $\Qbb_{p,\square}$-module.
\end{proof}

\begin{remark}
    We do not know whether the above lemma is true for a general static $\Qbb_{p,\square}$-module.
\end{remark}

\begin{proposition}\label{lem:orbit morphism locally analytic}
    Let $V$ be a flat $\Qbb_{p,\square}$-module with a continuous $G$-action.
    We assume that $V$ is non-derived $G_n$-analytic.
    Then we have a $G$-equivariant isomorphism $$C^{\bar{n}}(G,V)_{\star_1} \overset{\sim}{\to} C^{\bar{n}}(G,V)_{\star_{1,3}}.$$
\end{proposition}
\begin{proof}
    Since $V^{G_n\mathchar`-\an}\to V$ is an isomorphism, we have an orbit morphism
    $$\Ocal_V\colon V\simeq V^{G_n\mathchar`-\an}\to C^{n}(G,\Qbb_p)\otimes_{\Qbb_{p,\square}}V,$$
    which is $G^2$-equivariant where the action on the left-hand side is defined via $G^2\to G;\; (g,h)\mapsto h$ and the action on the right-hand side is given by the product of the $\star_{1,3}$-action and $\star_{2}$-action (we denote it by $(C^{n}(G,\Qbb_p)\otimes_{\Qbb_{p,\square}}V)_{\star_{1,3},\star_2}$. For other cases, we use similar notations).
    Therefore, we obtain a $G^2$-equivariant $C^{n}(G,\Qbb_p)$-linear morphism 
    $$\alpha\colon (C^{n}(G,\Qbb_p)\otimes_{\Qbb_{p,\square}}V)_{\star_1,\star_{2,3}} \to (C^{n}(G,\Qbb_p)\otimes_{\Qbb_{p,\square}}V)_{\star_{1,3},\star_2}.$$
    We construct the inverse of it.
    The morphism $G\to G;\; g \mapsto g^{-1}$ induces a $G^2$-equivariant morphism $\iota\colon C^{n}(G,\Qbb_p)_{\star_1,\star_2}\to C^{n}(G,\Qbb_p)_{\star_2,\star_1}$.
    Then we obtain a $G^2$-equivariant morphism 
    $$V \overset{\Ocal_V}{\longrightarrow} (C^{n}(G,\Qbb_p)\otimes_{\Qbb_{p,\square}}V)_{\star_2,\star_{1,3}} \overset{\iota\otimes\id}{\longrightarrow} (C^{n}(G,\Qbb_p)\otimes_{\Qbb_{p,\square}}V)_{\star_1,\star_{2,3}},$$
    where the action on the left-hand side is defined via $G^2\to G;\; (g,h)\mapsto g$.
    Therefore, it induces a $G^2$-equivariant $C^{n}(G,\Qbb_p)$-linear morphism 
    $$\beta\colon (C^{n}(G,\Qbb_p)\otimes_{\Qbb_{p,\square}}V)_{\star_{1,3},\star_{2}} \to (C^{n}(G,\Qbb_p)\otimes_{\Qbb_{p,\square}}V)_{\star_{1},\star_{2,3}}.$$
    We prove $\beta\circ\alpha=\id$.
    By taking $G$-invariant submodules for the $\star_1$-action and evaluating at $1\in G$, we obtain the following commutative diagram:
    $$
    \xymatrix{
        V\ar[r]\ar[d] &V\ar[d]\\
        (C^{n}(G,\Qbb_p)\otimes_{\Qbb_{p,\square}}V)_{\star_{1}}\ar[r]^-{\beta\circ\alpha}\ar[d] &(C^{n}(G,\Qbb_p)\otimes_{\Qbb_{p,\square}}V)_{\star_{1}}\ar[d]\\
        V\ar[r] &V,
    }
    $$
    where the compositions of vertical morphisms are the identity morphisms and the lower horizontal morphism is also the identity.
    Therefore, the upper horizontal morphism is also the identity.
    The equation $\beta\circ\alpha=\id$ follows from it.
    We can show $\alpha\circ\beta=\id$ by the same argument.
\end{proof}

\begin{proposition}\label{prop:analytic cohomology and tensor product}
    Let $V$ be a (static) $\Qbb_{p,\square}$-module with a continuous $G$-action, and $W$ be a (static) flat $\Qbb_{p,\square}$-module with a continuous $G$-action.
    We assume that the $G$-action on $W$ is non-derived $G_n$-analytic.
    Then we have a natural isomorphism 
    $$(V\otimes_{\Qbb_{p,\square}}W)^{G_n\mathchar`-\an} \simeq V^{G_n\mathchar`-\an}\otimes_{\Qbb_{p,\square}}W.$$
\end{proposition}
\begin{proof}
    It follows from the following isomorphisms:
    \begin{align*}
        (V\otimes_{\Qbb_{p,\square}}W)^{G_n\mathchar`-\an}
        &\simeq \Gamma(G,(C^{n}(G,\Qbb_p)\otimes_{\Qbb_{p,\square}}V\otimes_{\Qbb_{p,\square}}W)_{\star_{1,3,4}})\\
        &\simeq \Gamma(G,(C^{n}(G,\Qbb_p)\otimes_{\Qbb_{p,\square}}V\otimes_{\Qbb_{p,\square}}W)_{\star_{1,3}})\\
        &\simeq \Gamma(G,(C^{n}(G,\Qbb_p)\otimes_{\Qbb_{p,\square}}V_{\star_{1,3}}))\otimes_{\Qbb_{p,\square}}W\\
        &\simeq V^{G_n\mathchar`-\an}\otimes_{\Qbb_{p,\square}}W,
    \end{align*}
    where the second isomorphism follows from Proposition \ref{lem:orbit morphism locally analytic}, and the third isomorphism follows from Lemma \ref{lem:fixed part and tensor products}.
\end{proof}

\begin{proposition}\label{prop:an functor and tensor product commute}
    Let $A$ be a (static) flat $\Qbb_{p,\square}$-algebra with a continuous $G$-action, and $V$, $W$ be (static) $A_{\square}$-modules with continuous semilinear $G$-actions.
    We assume that the $G$-actions on $A$ and $W$ are non-derived $G_n$-analytic and $W$ is flat over $A_{\square}$.
    Then we have an isomorphism
    $$(V\otimes_{A_{\square}}W)^{G_n\mathchar`-\an} \cong V^{G_n\mathchar`-\an}\otimes_{A_{\square}}W.$$
\end{proposition}
\begin{proof}
    We have a $G$-equivariant isomorphism
    $$(C^{n}(G,\Qbb_p)\otimes_{\Qbb_{p,\square}}(V\otimes_{A_{\square}}W))_{\star_{1,3}}\cong C^{\bar{n}}(G,V)_{\star_{1,3}}\otimes_{C^{\bar{n}}(G,A)_{\square \star_{1,3}}}C^{\bar{n}}(G,W)_{\star_{1,3}}.$$ 
    By Lemma \ref{lem:orbit morphism locally analytic}, we have isomorphisms
    \begin{align*}
        C^{\bar{n}}(G,A)_{\star_{1,3}}&\cong C^{\bar{n}}(G,A)_{\star_{1}},\\
        C^{\bar{n}}(G,W)_{\star_{1,3}}&\cong C^{\bar{n}}(G,W)_{\star_{1}}.
    \end{align*}
    Moreover, we have an isomorphism
    $$C^{\bar{n}}(G,W)_{\star_{1}} \cong C^{\bar{n}}(G,A)_{\star_{1}} \otimes_{A^{\triv}_{\square}} W^{\triv},$$
    where $A^{\triv}$ (resp.\ $W^{\triv}$) is the $\Qbb_{p,\square}$-module $A$ (resp.\ $W$) with the trivial $G$-action.
    Therefore, we obtain an isomorphism
    $$C^{\bar{n}}(G,W)_{\star_{1,3}} \cong C^{\bar{n}}(G,A)_{\star_{1,3}} \otimes_{A^{\triv}_{\square}} W^{\triv},$$
    where $A^{\triv}\to C^{\bar{n}}(G,A)_{\star_{1,3}};\; a \mapsto (g\mapsto ga)$ is the orbit morphism.
    Hence, we obtain an isomorphism
    $$(C^{n}(G,\Qbb_p)\otimes_{\Qbb_{p,\square}}(V\otimes_{A_{\square}}W))_{\star_{1,3}}\cong C^{\bar{n}}(G,V)_{\star_{1,3}}\otimes_{A^{\triv}_{\square}} W^{\triv},$$
    where the $G$-action on the right-hand side is given by $g\otimes \id$ for $g\in G$.
    Therefore, the claim follows from Lemma \ref{lem:fixed part and tensor products}.
\end{proof}

\begin{proposition}\label{prop:la functor and tensor product commute}
    Let $A$ be an animated $\Qbb_{p,\square}$-algebra with a continuous $G$-action, and $V$, $W$ be objects of $\Dcal(A_{\square})$ with continuous semilinear $G$-actions.
    We assume that the $G$-actions on $A$ and $W$ are $G$-locally analytic.
    Then we have an equivalence
    $$(V\otimes_{A_{\square}}^{\Lbb}W)^{G\mathchar`-R\la} \simeq V^{G\mathchar`-R\la}\otimes_{A_{\square}}^{\Lbb}W.$$
\end{proposition}
\begin{proof}
    When $A=\Qbb_p$, the proposition is proved in \cite[Corollary 3.1.15 (3)]{RJRC23}.
    In the general case, the proposition follows from the Bar resolutions and the fact that the functor $(-)^{G\mathchar`-R\la}$ preserves all small colimits (\cite[Proposition 3.1.7]{RJRC23}).
\end{proof}


\section{Algebraic-affinoid analytic algebras}

In this section, we introduce a notion of algebraic-affinoid analytic algebras, and prove some basic properties of them.
Let $L$ be a complete non-archimedean field, and $\pi$ be a pseudo-uniformizer of $L$.

\begin{definition}\label{def:algebraic-affinoid algebra}
	A solid (static) $L_{\square}$-algebra $A$ is called an \textit{algebraic-affinoid $L_{\square}$-algebra} if there is an affinoid $L$-algebra $R$ and a morphism $R\to A$ such that $A$ is relatively discrete and finitely generated algebra over $R$.
\end{definition}

\begin{remark}\label{rem:aff of def sub}
    Let $R,A$ be as above.
    By replacing $R$ with its image in $A$, which is also an affinoid $L$-algebra, we can arrange $R\subset A$.
\end{remark}

\begin{remark}\label{rem:nuclear alg-aff}
    By \cite[Proposition 2.5]{Mikami23}, $R$ is a nuclear $L_{\square}$-module.
    Since $A$ is relatively discrete over $R_{\square}$, it is a nuclear $R_{\square}$-module by Remark \ref{rem:nuclear}. 
    Therefore, by \cite[Lemma 2.9]{Mikami23}, $A$ is a nuclear $L_{\square}$-module.
\end{remark}

\begin{definition}
	An \textit{algebraic-affinoid analytic $L_{\square}$-algebra} is a (static) analytic $L_{\square}$-algebra of the form $(A,A^+)_{\square}$ where $A$ is an algebraic-affinoid $L_{\square}$-algebra and $A^+$ is a subring of $A(\ast)$ such that $A$ is $(A^+,A^+)_{\square}$-complete.
	Let $\AlgAff_{L_{\square}}$ denote the category of algebraic-affinoid analytic $L_{\square}$-algebras.
\end{definition}

\begin{remark}
	By \cite[Proposition 3.34]{And21}, there is a fully faithful functor
	\begin{align*}
		&\{(A,A^+)\mid \text{$(A,A^+)$ is an affinoid pair over $L$ of weakly finite type}\}\\
        &\to \AlgAff_{\Qbb_p};\; (A,A^+)\mapsto (A,A^+)_{\square},
	\end{align*}
	where ``of weakly finite type'' means that $A$ is of topologically finite type over $L$ (there is no condition on $A^+$).
\end{remark}

\begin{remark}\label{rem:every morph steady}
    By the same argument as in \cite[Proposition 2.20]{Mikami23}, every morphism between algebraic-affinoid analytic $L_{\square}$-algebras is steady.
\end{remark}

\begin{definition}\label{def:power-bounded elements}
    Let $A$ be an algebraic-affinoid $L_{\square}$-algebra.
    We define $A^{\circ}(\ast)\subset A(\ast)$ as the subring of $f\in A(\ast)$ such that $A$ is $(\Zbb[T],\Zbb[T])_{\square}$-complete when viewed as a $\Zbb[T]$-module via $T\mapsto f$.
    We set $A^{b}(\ast)=A^{\circ}(\ast)[1/\pi]\subset A(\ast)$.
    Then, we define the condensed subring $A^{b}\subset A$ to be the largest one whose underlying discrete subring is $A^{b}(\ast)$.
\end{definition}

\begin{remark}
    For an extremally disconnected set $S$ and $f\in A(S)$, $f\in A^{b}(S)$ if and only if $f(s)\in A^{b}(\ast)$ for any $s\in S$.
\end{remark}

\begin{lemma}\label{lem:power-bounded elements}
    Let $A$ be an algebraic-affinoid $L_{\square}$-algebra.
    For $f\in A$, $f\in A^{\circ}$ if and only if $L[T]\to A;\; T\mapsto f$ extends to $L\langle T\rangle\to A$.
\end{lemma}
\begin{proof}
    If $f\in A^{\circ}$, then we get a morphism $$L\langle T\rangle=L[T]\otimes_{(\Zbb[T],\Zbb)_{\square}}(\Zbb[T],\Zbb[T])_{\square} \to A;\; T\mapsto f.$$
    On the other hand, we assume that $L[T]\to A;\; T\mapsto f$ extends to $L\langle T\rangle\to A$.
    Since $A$ and $\Ocal_L\langle T\rangle$ are nuclear $\Ocal_{L,\square}$-modules, $A$ is a nuclear $\Ocal_L\langle T\rangle_{\square}$-module by \cite[Lemma 2.9]{Mikami23}.
    Then by \cite[Proposition 2.5, Proposition 2.6]{Mikami23}, $A$ is $(\Zbb[T],\Zbb[T])_{\square}$-complete.
\end{proof}

\begin{corollary}\label{cor:A^b=A}
    Let $A$ be an affinoid $L$-algebra.
    Then $A^{\circ}(\ast)$ defined in Definition \ref{def:power-bounded elements} coincides with the subring of power-bounded elements in the usual sense.
    In particular, we have $A=A^b$.
\end{corollary}

\begin{lemma}\label{lem:A^b is integrally closed}
    Let $A$ be an algebraic-affinoid $L_{\square}$-algebra.
    Then $A^b(\ast)$ is integrally closed in $A(\ast)$.
\end{lemma}
\begin{proof}
    By Proposition \cite[Proposition 3.22]{And21}, $A^{\circ}(\ast)\subset A(\ast)$ is integrally closed.
    Therefore, $A^b(\ast)=A^{\circ}(\ast)[1/\pi]\subset A(\ast)$ is also integrally closed.
\end{proof}

\begin{lemma}\label{lem:bounded nilpotent}
    Let $A$ be an algebraic-affinoid $L_{\square}$-algebra, and $I\subset A$ be a nilpotent ideal.
    Then for $f\in A$, $f\in A^{\circ}(\ast)$ if and only if its image in $(A/I)(\ast)$ lies in $(A/I)^{\circ}(\ast)$.
\end{lemma}
\begin{proof}
    The only if direction follows from Lemma \ref{lem:power-bounded elements}.
    We prove the if direction.
    By induction, we may assume $I^2=0$.
    We regard $A$-modules as $\Zbb[T]$-modules via $T\mapsto f$.
    By hypothesis, $A/I$ is $(\Zbb[T],\Zbb[T])_{\square}$-complete.
    Since $I/I^2=I$ is relatively discrete over $A/I$, $I$ is also $(\Zbb[T],\Zbb[T])_{\square}$-complete.
    Therefore, $A$ is $(\Zbb[T],\Zbb[T])_{\square}$-complete.
\end{proof}

\begin{proposition}\label{prop:A^b integral}
    Let $R$ be an affinoid $L$-algebra, and $A$ be a relatively discrete and finitely generated $R_{\square}$-algebra.
    Then $A^{b}$ is relatively discrete and integral over $R_{\square}$.
    Moreover, the nilradical $\sqrt{(0)}$ of $A$ is contained in $A^b$ and $A^b/\sqrt{(0)}=(A/\sqrt{(0)})^b$ is an affinoid $L$-algebra.
\end{proposition}
\begin{proof}
    Since $A^{b}$ is an $R$-submodule of $A$, $A^{b}$ is relatively discrete over $R$ by Lemma \ref{lem:rd sub}.
    We write $A$ as a filtered colimit $A=\varinjlim_{\lambda} A_{\lambda}$ of finitely generated $R$-submodules of $A$.
    Let $f\in A^{\circ}(\ast)$. 
    Then, we get a morphism
    $$L\otimes_{\Zbb_{\square}}\Zbb[[S]]\to L\langle T\rangle \to A$$
    with $S\mapsto \pi T$ and $T\mapsto f$.
    Since $L\otimes_{\Zbb_{\square}}\Zbb[[S]]$ is a compact $L_{\square}$-module, the above morphism factors through a certain $A_{\lambda}$.
    Therefore, for any $n\geq 0$, $\pi^nf^n\in A_{\lambda}$, and hence $f^n \in A_{\lambda}$.
    Consequently, $f$ is integral over $R(\ast)$, and so $A^{b}(\ast)=A^{\circ}(\ast)[1/\pi]$ is integral over $R(\ast)$, which completes the proof of the first part.
    We prove the second part.
    By Lemma \ref{lem:bounded nilpotent}, we get $\sqrt{(0)}\subset A^b$ and $A^b/\sqrt{(0)}=(A/\sqrt{(0)})^b$.
    Moreover, by Lemma \ref{lem:A^b is integrally closed}, $(A/\sqrt{(0)})^b(\ast)$ is the integral closure of the image of $R(\ast)$ in $A/\sqrt{(0)}(\ast)$. 
    By \cite[Satz 3.3.3]{BKKN67}, $R(\ast)$ is an excellent ring, and hence a Nagata ring.
    Therefore, $(A/\sqrt{(0)})^{b}(\ast)$ is finite over $R(\ast)$.
\end{proof}

\begin{remark}
    We note that $A^b$ is not necessarily an affinoid $L$-algebra.
    For example, if $A=L[S,T]/(T^2)$, then $A^b=L\oplus L[S]T$ is not an affinoid $L$-algebra.
\end{remark}

\begin{definition}
    Let $A$ be an algebraic affinoid $L_{\square}$-algebra.
    An \textit{affinoid $L$-algebra of definition of $A$} is a subring $A'\subset A$ such that $A'$ is an affinoid $L$-algebra and $A$ is relatively discrete and finitely generated over $A'$.
    By Remark \ref{rem:aff of def sub}, there exists at least one affinoid $L$-algebra of definition of $A$.
\end{definition}

\begin{lemma}\label{lem:A^b=A}
    Let $A$ be an algebraic-affinoid $L_{\square}$-algebra.
    If $A=A^b$, then $A$ is an affinoid $L$-algebra.
\end{lemma}
\begin{proof}
    Let $A'\subset A$ be an affinoid $L$-algebra of definition of $A$.
    By Proposition \ref{prop:A^b integral}, $A$ is integral over $A'$.
    Moreover, $A$ is a finitely generated $A'$-algebra.
    Hence $A$ is finite over $A'$, and therefore an affinoid $L$-algebra.
\end{proof}

\begin{lemma}\label{lem:basic properties of aff of def}
    Let $A$ be an algebraic affinoid $L_{\square}$-algebra.
    \begin{enumerate}
        \item Every affinoid $L$-algebra of definition $A'\subset A$ is contained in $A^b$.
        \item Let $B$ be an affinoid $L$-algebra, and $f\colon B \to A$ is a morphism of $L_{\square}$-algebras.
        Then there exists an affinoid $L$-algebra of definition $A'\subset A$  such that $f\colon B\to A$ factors through $A'\subset A$.
        \item Let $A_1,A_2\subset A$ be affinoid $L$-algebras of definition of $A$.
        Then there exists an affinoid $L$-algebra of definition $A_3\subset A$ such that $A_1,A_2$ are contained in $A_3$.
        \item  There is an isomorphism
        $$\varinjlim_{A'\subset A}A'=A^b,$$
        where the colimit is taken over all affinoid $L$-algebras of definition of $A'\subset A$.
    \end{enumerate}
\end{lemma}
\begin{proof}
    For an affinoid $L$-algebra $A'\subset A$ of definition of $A$, by Corollary \ref{cor:A^b=A}, $A'=(A')^b\subset A^b$, which proves (1).
    We prove (2).
    We take an affinoid $L$-algebra of definition $A'\subset A$.
    Let $x_1,\ldots,x_n\in B$ be a set of topological generators of $B$ as an affinoid $L$-algebra, and we set $y_i=f(x_i)\in A^b$.
    By Proposition \ref{prop:A^b integral}, $y_1,\ldots, y_n$ are integral over $A'$.
    Therefore, $A'[y_1,\ldots,y_n]$ is finite over $A'$, and this is also an affinoid $L$-algebra of definition of $A$.
    By construction, $f\colon B\to A$ factors through $A'[y_1,\ldots,y_n]$.
    Parts (3) and (4) can be proved in the same way.
\end{proof}

\begin{lemma}\label{lem:tensor product}
    Let $B\leftarrow A\to C$ be a diagram of algebraic-affinoid $L_{\square}$-algebras.
    Then $B\otimes_{A_{\square}} C$ is also an algebraic-affinoid $L_{\square}$-algebra.
\end{lemma}
\begin{proof}
    By Lemma \ref{lem:basic properties of aff of def} (2), we can take affinoid $L$-algebras of definition $A'\subset A$, $B'\subset B$, and $C'\subset C$ such that the image of $A'$ in $B$ (resp.\ $C$) is contained in $B'$ (resp.\ $C'$).
    First, by the proof of \cite[Proposition 2.14]{Mikami23}, $B'\otimes_{A'_{\square}} C'$ is an affinoid $L$-algebra.
    Since $A$ (resp.\ $B$, $C$) is a relatively discrete and finite generated $A$-algebra (resp.\ $B$-algebra, $C$-algebra), $B\otimes_{A_{\square}} C$ is also a relatively discrete and finite generated $B'\otimes_{A'_{\square}} C'$-algebra.
\end{proof}

\begin{definition}
Let $\Acal$ be an algebraic-affinoid analytic $L_{\square}$-algebra.
We define $$\Acal^+=\Hom_{\AnRing_{\Zbb_{\square}}}(\Zbb[X]_{\square},\Acal),$$
which is a subring of $\underline{\Acal}^{\circ}(\ast)$.
\end{definition}

\begin{definition}
    An \textit{algebraic-affinoid pair} over $L_{\square}$ is a pair $(A,A^+)$, where $A$ is an algebraic-affinoid $L_{\square}$-algebra and $\sqrt{(0)}\subset A^+\subset A^b(\ast)$ such that $A^+/\sqrt{(0)}$ is a ring of integral elements of $A^b(\ast)/\sqrt{(0)}$ in Huber's sense.
    Let $\AlgAffPair_{L_{\square}}$ denote the category of algebraic-affinoid pairs over $L_{\square}$.
\end{definition}

\begin{proposition}\label{prop:ring of integral elements}
    The functor
    \begin{align*}
        \AlgAffPair_{L_{\square}}\to \AlgAff_{L_{\square}};\; (A,A^+)\mapsto (A,A^+)_{\square}
    \end{align*}
    is a categorical equivalence.
    A quasi-inverse is given by
    \begin{align*}
        \AlgAff_{L_{\square}}\to \AlgAffPair_{L_{\square}};\; \Acal\mapsto (\underline{\Acal},\Acal^+).
    \end{align*}
\end{proposition}
\begin{proof}
    Let $\Acal=(A,A^+)_{\square}\in \AlgAff_{L_{\square}}$, where $A$ is $(A^+,A^+)_{\square}$-complete.
    Then, $A^+\subset \Acal^+$, so $\Dcal((A,\Acal^+)_{\square})\subset \Dcal((A,A^+)_{\square})$.
    On the other hand, by the definition of $\Acal^+$, we have $\Dcal((A,A^+)_{\square})\subset \Dcal((A,\Acal^+)_{\square})$.
    Therefore, we get $(A,A^+)_{\square}=(A,\Acal^+)_{\square}$.
    By \cite[Proposition 3.22]{And21}, $\sqrt{(0)}\subset \Acal^+$, where we note that nilpotent elements are integral over $\Zbb$.
    We set $\Acal_{\mathrm{red}}=(A/\sqrt{(0)},\Acal^+/\sqrt{(0)})_{\square}$.
    By \cite[Proposition 12.23]{AG}, $\Acal_{\mathrm{red}}^{+}=\Acal^+/\sqrt{(0)}$ (cf.\ the proof of Lemma \ref{lem:A^+ nuclear}).
    By \cite[Proposition 3.22, Lemma 3.31]{And21}, $(\Acal_{\mathrm{red}})^{+}\subset (A^b/\sqrt{(0)})(\ast)$ is integrally closed and open.
    Therefore, it is a ring of integral elements of $(A^b/\sqrt{(0)})(\ast)$.

    Next, let $(A,A^+)$ be an algebraic-affinoid pair over $L_{\square}$, and we set $\Acal=(A,A^+)_{\square}$.
    Let us prove $(A,A^+)=(A,\Acal^+)$.
    In the same way as in the previous paragraph, we may assume that $A$ is reduced.
    By definition, we have $A^+\subset \Acal^+$. 
    Assume that $\Acal^+ \not\subset A^+$ and take $a\in \Acal^+\setminus A^+$.
    Then, there exists a morphism
    \begin{align*}
        f\colon \prod_{\Nbb} \Zbb[T] \to (A,A^+)_{\square}\otimes_{\Zbb_{\square}} \prod_{\Nbb} \Zbb\cong (A,A^+)_{\square}\otimes_{(\Zbb[T],\Zbb[T])_{\square}} \prod_{\Nbb} \Zbb[T]
    \end{align*}
    induced by $(\Zbb[T],\Zbb[T])_{\square}\to (A,\Acal^+)_{\square}=(A,A^+)_{\square};\; T\mapsto a$.
    We write $A$ as a filtered colimit $A=\varinjlim_{\lambda} A_{\lambda}$ of finitely generated $A^b$-submodules of $A$.
    Then, we have
    \begin{align*}
        (A,A^+)_{\square}\otimes_{\Zbb_{\square}} \prod_{\Nbb}\Zbb \cong \varinjlim_{\lambda}(A_{\lambda}\otimes_{(A^b,A^+)_{\square}}((A^b,A^+)_{\square}\otimes_{\Zbb_{\square}}\prod_{\Nbb}\Zbb)).
    \end{align*}
    Since $\prod_{\Nbb} \Zbb[T]$ is a compact $(\Zbb[T],\Zbb[T])_{\square}$-module, the morphism $f$ factors through $A_{\lambda}\otimes_{(A^b,A^+)_{\square}}((A^b,A^+)_{\square}\otimes_{\Zbb_{\square}}\prod_{\Nbb}\Zbb)$ for some $\lambda$.
    Let $A_{\lambda}^0\subset A_{\lambda}$ be a module of definition of $A_{\lambda}$, i.e., $\{\pi^nA_{\lambda}^0\}_{n\geq 0}$ is a fundamental system of neighborhoods of $0$ in $A_{\lambda}$.
    By the same argument as in \cite[Theorem 3.27]{And21}, we obtain an isomorphism
    \begin{align*}
        A_{\lambda}\otimes_{(A^b,A^+)_{\square}}((A^b,A^+)_{\square}\otimes_{\Zbb_{\square}}\prod_{\Nbb}\Zbb) \cong \varinjlim_{B\subset A^+, M}\prod_{\Nbb} M,
    \end{align*}
    where the colimit is taken over all finitely generated $\Zbb$-subalgebras $B\subset A^+$ and all closed $B$-submodules $M$ of $A_{\lambda}$ such that $M/(M\cap \pi^n A_{\lambda}^0)$ is finitely generated over $B$ for any $n>0$.
    Then, there exists $B$ and $M$ such that $f(1,T,T^2,\ldots)=(1,a,a^2,\ldots)\in \prod_{\Nbb} M$.
    Since $a\in A^b$, we have $(1,a,a^2,\ldots)\in \prod_{\Nbb} (M\cap A^b)$.
    Then, $M\cap A^b$ is quasi-finitely generated over $B$, so the same argument as in \cite[Proposition 3.34]{And21} leads to a contradiction.
\end{proof}

\begin{lemma}\label{lem:A^+ nuclear}
    Let $(A,A^+)$ be an algebraic-affinoid pair over $L_{\square}$.
    Then the inclusion $\Dcal((A,A^+)_{\square})\subset \Dcal(A_{\square})$ induces $\Dcal^{\nuc}((A,A^+)_{\square})= \Dcal^{\nuc}(A_{\square})$
\end{lemma}
\begin{proof}
    It suffices to show that for $M\in \Dcal(A_{\square})$, $M$ is nuclear in $\Dcal(A_{\square})$ if and only if $M$ is $(A,A^+)_{\square}$-complete and nuclear in $\Dcal((A,A^+)_{\square})$.
    By \cite[Lemma 4.5]{And23}, if $M$ is $(A,A^+)_{\square}$-complete and nuclear in $\Dcal((A,A^+)_{\square})$, then $M$ is nuclear in $\Dcal(A_{\square})$.
    On the other hand, assume that $M$ is nuclear in $\Dcal(A_{\square})$.
    If $M$ is $(A,A^+)_{\square}$-complete, then $M\cong M\otimes_{A_{\square}}^{\Lbb} (A,A^+)_{\square}$ is also nuclear in $\Dcal((A,A^+)_{\square})$ by \cite[Proposition 2.3.22 (iii)]{Mann22}.
    Therefore, it suffices to show that $M$ is $(A,A^+)_{\square}$-complete.
    We take $f\in A^+$.
    Then, we get a morphism $(\Ocal_L\langle T\rangle,\Zbb[T])_{\square}\to (A,A^+)_{\square};\; T\mapsto f$.
    By Remark \ref{rem:nuclear alg-aff} and \cite[Lemma 2.9]{Mikami23}, $M$ is nuclear in $\Dcal(\Ocal_L\langle T\rangle_{\square})$.
    Therefore, by \cite[Proposition 2.5, Proposition 2.6]{Mikami23}, $M$ is $(\Ocal_L\langle T\rangle,\Zbb[T])_{\square}$-complete.
    Since this holds for any $f\in A^+$, $M$ is $(A,A^+)_{\square}$-complete.
\end{proof}

\begin{corollary}\label{cor:dualizable indep of A+}
    The inclusion $\Dcal(\Acal)\subset \Dcal(A_{\square})$ induces a categorical equivalence between the categories of dualizable objects in $\Dcal(\Acal)$ and $\Dcal(A_{\square})$.
\end{corollary}
\begin{proof}
    By Lemma \ref{lem:A^+ nuclear}, the inclusion induces $\Dcal^{\nuc}(\Acal)=\Dcal^{\nuc}(A_{\square})$.
    Then, the claim follows from Corollary \ref{cor:dualizable nuclear}.
\end{proof}

\begin{lemma and definition}
    Let $(B,B^+)\leftarrow (A,A^+)\to (C,C^+)$ be a diagram of algebraic-affinoid pairs over $L_{\square}$, and we set $\Acal=(A,A^+)_{\square}$, $\Bcal=(B,B^+)_{\square}$, and $\Ccal=(C,C^+)_{\square}$.
    We define an algebraic-affinoid analytic $L_{\square}$-algebra
    $$\Bcal\otimes_{\Acal} \Ccal=(B\otimes_{A_{\square}}C, \mbox{the image of $B^+\otimes C^+$ in $(B\otimes_{A_{\square}}C)(\ast)$})_{\square}.$$
    Then, it is a pushout of $\Bcal\leftarrow \Acal\to \Ccal$ in the category $\AlgAff_{L_{\square}}.$
\end{lemma and definition}
\begin{proof}
    We write $D^+$ for the image of $B^+\otimes C^+$ in $(B\otimes_{A_{\square}}C)(\ast)$.
    Since $B$ and $C$ are nuclear as $L_{\square}$-modules, they are also nuclear as $A_{\square}$-modules by \cite[Lemma 2.9]{Mikami23}.
    Therefore, $B\otimes_{A_{\square}}^{\Lbb}C$ is also nuclear as an $A_{\square}$-module.
    By using \cite[Lemma 2.9]{Mikami23} again, $B\otimes_{A_{\square}}^{\Lbb}C$ is also nuclear as a $B_{\square}$-module (resp.\ a $C_{\square}$-module).
    By Lemma \ref{lem:A^+ nuclear}, $B\otimes_{A_{\square}}^{\Lbb}C$ is $\Bcal$-complete and $\Ccal$-complete.
    Therefore, $\Bcal\otimes_{\Acal}\Ccal$ defined in the above is isomorphic to $\pi_0(\Bcal\otimes_{\Acal}^{\Lbb}\Ccal)$, which proves the lemma.
\end{proof}

\begin{lemma and definition}
Let $A$ be an algebraic-affinoid $L_{\square}$-algebra, and $f\in A$.
We define $L_{\square}$-algebras
\begin{align*}
    &A\langle T\rangle\coloneqq A\otimes_{\Zbb_{\square}}^{\Lbb} (\Zbb[T],\Zbb[T])_{\square},\\
    &A\langle 1/f\rangle\coloneqq A\otimes_{(\Zbb[T],\Zbb)_{\square}}^{\Lbb}(\Zbb[T,T^{-1}],\Zbb[T^{-1}])_{\square},\\
    &A\langle f/1\rangle\coloneqq A\otimes_{(\Zbb[T],\Zbb)_{\square}}^{\Lbb}(\Zbb[T],\Zbb[T])_{\square},
\end{align*}
where $\Zbb[T]\to A$ is given by $T\mapsto f$.
Then, they are also algebraic-affinoid $L_{\square}$-algebras.
\end{lemma and definition}
\begin{proof}
    If $A$ is an affinoid $L$-algebra, then these rings coincide with the usual ones by \cite[Proposition 3.14, Proposition 4.11]{And21}.
    In general, let $A'\subset A$ be an affinoid $L$-algebra of definition of $A$.
    Then there is an isomorphism
    \begin{align*}
        A\otimes_{\Zbb_{\square}}^{\Lbb} (\Zbb[T],\Zbb[T])_{\square} 
        &\cong A\otimes_{A'_{\square}}^{\Lbb}(A^b\otimes_{\Zbb_{\square}}^{\Lbb} (\Zbb[T],\Zbb[T])_{\square})\\
        &\cong A\otimes_{A'_{\square}}^{\Lbb} A'\langle T\rangle,
    \end{align*}
    which is relatively discrete over $A'\langle T\rangle$.
    Since $A'\langle T\rangle(\ast)$ is flat over $A'(\ast)$, $A\langle T\rangle=A\otimes_{A'_{\square}}^{\Lbb} A'\langle T\rangle$ is static by Lemma \ref{lem:condensification functor} (1).
    Since $A$ is a finitely generated $A'$-algebra, $A\langle T\rangle$ is also a finitely generated $A'\langle T\rangle$-algebra.

    In the same way as in the proof of \cite[Proposition 4.11]{And21}, we get isomorphisms
    \begin{align*}
        &A\langle 1/f \rangle \cong A\langle T\rangle /^{\Lbb} (fT-1),\\
        &A\langle f/1 \rangle \cong A\langle T\rangle /^{\Lbb} (T-f),
    \end{align*}
    where $/^{\Lbb}$ means the derived quotient.
    It suffices to show that both $fT-1$ and $T-f$ are non-zero divisors in $A\langle T\rangle(\ast)$ by Lemma \ref{lem:condensification functor} (1).
    We write $A$ as a filtered colimit $A=\varinjlim_{\lambda} A_{\lambda}$ of finitely generated $A'$-submodules of $A$.
    Then, we have $A\langle T\rangle(\ast)=\varinjlim_{\lambda} A_{\lambda}\langle T\rangle (\ast)$.
    We take $g=\sum_{n=0}^{\infty} a_n T^n \in A\langle T\rangle(\ast)$ such that $(fT-1)g=0$ or $(T-f)g=0$.
    When $(fT-1)g=0$, then by comparing the coefficients starting from the constant term, we obtain $g=0$.
    When $(T-f)g=0$, then we get $a_0f=0$ and $a_i=a_{i+1}f$ for any $i\geq 0$.
    Since $A(\ast)$ is noetherian, there exists $N\geq 0$ such that all $f$-power torsion elements are $f^N$-torsion.
    From these, we obtain $g=0$.
\end{proof}

\begin{definition}
	Let $\Acal=(A,A^+)_{\square}$ be an algebraic-affinoid analytic $L_{\square}$-algebra.
\begin{enumerate}
\item 
We define an algebraic-affinoid analytic $L_{\square}$-algebra $\Acal\langle T\rangle$ by 
$$\Acal\langle T\rangle=(A\langle T\rangle ,A^+[T])_{\square}.$$
Then we have $$\Acal\langle T\rangle\cong \Acal\otimes_{\Zbb_{\square}}^{\Lbb} (\Zbb[T],\Zbb[T])_{\square}.$$ 
\item
For $f\in A$, we define an algebraic-affinoid analytic $L_{\square}$-algebras $\Acal\langle 1/f\rangle$ and $\Acal\langle f/1\rangle$ by 
\begin{align*}
    &\Acal\langle 1/f\rangle=(A\langle 1/f\rangle, A^+[1/f])_{\square},\\
    &\Acal\langle f/1\rangle=(A\langle f/1\rangle, A^+[f])_{\square}.
\end{align*}
Then we have
\begin{align*}
\Acal\langle 1/f\rangle&\cong \Acal\otimes_{(\Zbb[T],\Zbb)_{\square}}^{\Lbb}(\Zbb[T,T^{-1}],\Zbb[T^{-1}])_{\square},\\
\Acal\langle f/1\rangle&\cong \Acal\otimes_{(\Zbb[T],\Zbb)_{\square}}^{\Lbb}(\Zbb[T],\Zbb[T])_{\square},
\end{align*}
where $\Zbb[T]\to A$ is given by $T\mapsto f$.
A morphism of the form $\Acal \to \Acal \langle 1/f \rangle$ or $\Acal \to \Acal\langle f/1 \rangle$ is called a \textit{simple rational localization}.
\item A morphism $\Acal \to \Bcal$ in $\AlgAff_{L_{\square}}$ is called a \textit{rational localization} if it is a finite composition of simple rational localizations.
\item A finite family of rational localizations $\{\Acal\to \Bcal_i\}_{i=1}^n$ in $\AlgAff_{L_{\square}}$ is called a \textit{rational covering} if the family of functors 
$$\{-\otimes_{\Acal} \Bcal_i \colon \Dcal(\Acal)\to \Dcal(\Bcal_i)\}_{i=1}^n$$
is conservative.
\item A morphism $\Acal\to \Bcal$ in $\AlgAff_{L_{\square}}$ is called an \textit{affinoid localization} if 
\begin{itemize}
    \item it is a localization in the sense of \cite[Definition 12.16]{AG}, that is, the forgetful functor $\Dcal(\Bcal)\to \Dcal(\Acal)$ is fully faithful, and
    \item there exists a rational covering $\{\Bcal\to \Bcal_i\}_{i=1}^n$ such that, for each $i$, $\Acal\to \Bcal \to \Bcal_i$ is a rational localization.
\end{itemize}
A finite family of affinoid localizations $\{\Acal\to \Bcal_i\}_{i=1}^n$ in $\AlgAff_{L_{\square}}$ is called an \textit{affinoid covering} if the family of functors 
$$\{-\otimes_{\Acal} \Bcal_i \colon \Dcal(\Acal)\to \Dcal(\Bcal_i)\}_{i=1}^n$$
is conservative.
\end{enumerate}
\end{definition}

\begin{remark}
    Since both 
    $$(\Zbb[T],\Zbb)_{\square}\to (\Zbb[T,T^{-1}],\Zbb[T^{-1}])_{\square}$$ 
    and 
    $$(\Zbb[T],\Zbb)_{\square}\to (\Zbb[T],\Zbb[T])_{\square}$$ 
    are steady localizations, simple rational localizations are also steady localizations.
    Therefore, rational localizations are also steady localizations, and hence, affinoid localizations.
\end{remark}

\begin{example}
	Let $\Acal=(A,A^+)_{\square}$ be an algebraic-affinoid analytic $L_{\square}$-algebra.
	\begin{enumerate}
	\item For $f\in A$, $\{\Acal\to \Acal\langle 1/f\rangle,\Acal\to \Acal\langle f/1\rangle\}$ is a rational covering by the proof of \cite[Lemma 10.3 (ii)]{CM}.
	We call such a covering a \textit{simple Laurent covering}.
	\item For $f\in A$, $\{\Acal\to \Acal\langle 1/f\rangle,\Acal\to \Acal\langle 1/(1-f)\rangle\}$ is a rational covering, which can be proved in the same way as above.
	We call such a covering a \textit{simple balanced covering}.
	\end{enumerate}
\end{example}

\begin{example}
	Let $\Acal=(A,A^+)_{\square}$ be an algebraic-affinoid analytic $L_{\square}$-algebra.
	For $f\in A$, we define an algebraic-affinoid analytic $L_{\square}$-algebra 
	\begin{align*}
		\Acal_f=\Acal[1/f]\coloneqq (A[1/f],A^+)_{\square}\cong \Acal\otimes_{(\Zbb[X],\Zbb)_{\square}}^{\Lbb}(\Zbb[X,X^{-1}],\Zbb)_{\square},
	\end{align*}
	where $\Zbb[X]\to \Acal$ is given by $X\mapsto f$.
	Then $\Acal[1/f]$ is an algebraic-affinoid analytic $L_{\square}$-algebra, but $\Acal \to \Acal_f$ is \textbf{not} a rational localization in general.
\end{example}

\begin{example}
	Let $\Acal=(A,A^+)_{\square}$ be an algebraic-affinoid analytic $L_{\square}$-algebra.
	For $f_0,\ldots,f_n \in A$ such that $(f_0,\ldots,f_n)=A$, we define an algebraic-affinoid analytic $L_{\square}$-algebra 
	\begin{align*}
		&\phantom{{}\coloneqq{}}\Acal\left\langle \frac{f_1,\ldots,f_n}{f_0}\right\rangle \\
        &\coloneqq \left(A[1/f_0]\left\langle \frac{f_1/f_0}{1},\ldots, \frac{f_n/f_0}{1} \right\rangle, A[f_1/f_0,\ldots,f_n/f_0]\right)_{\square}\\
		&\cong \Acal[1/f_0] \otimes_{(\Zbb[X_0^{\pm 1},X_1,\ldots,X_n],\Zbb)_{\square}}^{\Lbb}(\Zbb[X_0^{\pm 1},X_1,\ldots,X_n],\Zbb[X_1/X_0,\ldots,X_n/X_0])_{\square}.
	\end{align*}
	Then $\Acal \to \Acal\left\langle \frac{f_1,\ldots,f_n}{f_0}\right\rangle$ is an affinoid localization, which can be proved as in \cite[Lemma 1.6.12, 1.6.13]{AWS17Ked}.
	Moreover, the family $\left\{\Acal \to \Acal\left\langle \frac{f_0,\ldots,f_n}{f_i}\right\rangle\right\}_{i=0}^n$ is an affinoid covering.
	We call it the \textit{standard affinoid covering defined by $f_0,\ldots,f_n$}.
\end{example}

\begin{example}\label{ex:classical affinoid cover}
	Let $(A,A^+)$ be an affinoid pair over $L$ of weakly finite type, and $\{(A,A^+)\to (B_i,B_i^+)\}_{i=1}^n$ be a family of affinoid localizations in the usual sense.
	Then $\{\Spa(B_i,B_i^+)\to \Spa(A,A^+)\}_{i=1}^n$ is an open covering if and only if $\{(A,A^+)_{\square}\to (B_i,B_i^+)_{\square}\}_{i=1}^n$ is an affinoid covering in $\AlgAff_{L_{\square}}$.
    Indeed, the only if direction follows from \cite[Proposition 4.12 (v)]{And21}, and the if direction follows from \cite[Proposition 2.7.8 (3)]{RC24}.
\end{example}

\begin{lemma}\label{lem:basic properties of affinoid localizations}
    \begin{enumerate}
        \item Let $\{\Acal\to \Bcal_i\}_{i=1}^n$ be an affinoid covering in $\AlgAff_{L_{\square}}$.
        Then $\Acal \to \prod_{i=1}^n \Bcal_i$ satisfies universal $\ast$-descent.
        \item Let $\Acal\to \Bcal$ be an affinoid localization in $\AlgAff_{L_{\square}}$. 
        Then for a morphism $\Acal\to \Ccal$ in $\AlgAff_{L_{\square}}$, $\Ccal\to \Ccal\otimes_{\Acal} \Bcal$ is also an affinoid localization.
        Moreover, the natural morphism $\Ccal\otimes_{\Acal}^{\Lbb} \Bcal\to \Ccal\otimes_{\Acal} \Bcal$ is an isomorphism.
    \end{enumerate}
\end{lemma}
\begin{proof}
    First, let us prove (1).
    Let $\Acal\to \Rcal$ be a morphism of analytic rings.
    We note that each $\Rcal \to \Rcal\otimes_{\Acal}^{\Lbb}\Bcal_i=\Rcal_i$ is a steady localization in the sense of \cite[Definition 12.16]{AG}.
    Therefore, by \cite[Proposition 12.18]{AG}, it suffices to show that 
    $$\{-\otimes_{\Rcal}^{\Lbb}\Rcal_i \colon \Dcal(\Rcal)\to \Dcal(\Rcal_i)\}_{i=1}^n$$
    is conservative.
    This follows from the fact that the forgetful functor $\Dcal(\Rcal)\to \Dcal(\Acal)$ is conservative and that each morphism $\Acal\to \Bcal_i$ is steady.

    Next, let us prove (2).
    If $\Acal\to \Bcal$ is a simple rational localization, then the claim follows from the definition.
    The case where $\Acal\to \Bcal$ is a rational localization follows from the above.
    In general, we take a rational covering $\{\Bcal\to \Bcal_i\}_{i=1}^n$ such that $\Acal\to \Bcal \to \Bcal_i$ is a rational localization for each $i$.
    By (1), $\{\Ccal\otimes_{\Acal}\Bcal\to \Ccal\otimes_{\Acal}\Bcal_i\}_{i=1}^n$ is also a rational covering, so $\Ccal\to \Ccal\otimes_{\Acal} \Bcal$ is also an affinoid localization.
    Moreover, $\Ccal\otimes_{\Acal}^{\Lbb} \Bcal\to \Ccal\otimes_{\Acal} \Bcal$ becomes an isomorphism after applying $-\otimes_{\Bcal}^{\Lbb}\Bcal_i$, and thus this is already an isomorphism by (1).
\end{proof}

\begin{lemma}\label{lem:adic space lcoal}
    Let $\Acal$ be an algebraic-affinoid analytic $L_{\square}$-algebra.
    If there is an affinoid covering $\{\Acal\to (B_i,B_i^+)_{\square}\}_{i=1}^n$ such that each $(B_i,B_i^+)$ is an affinoid pair over $L$ of weakly finite type, then there exists an affinoid pair $(A,A^+)$ over $L$ of weakly finite type such that $\Acal\cong (A,A^+)_{\square}$.
\end{lemma}
\begin{proof}
    Let $A=\underline{\Acal}$. 
    By Lemma \ref{lem:A^b=A}, it suffices to show $A=A^b$.
    We take $f\in A$. By replacing $f$ with $\pi^m f$ for $m$ sufficiently large, we may assume that the image of $f$ in $B_i$ is power-bounded.
    By gluing the morphisms $(L\langle T\rangle,\Zbb)_{\square} \to (B_i,B_i^+)_{\square};\; T\mapsto f$, we get a morphism
    $$(L\langle T\rangle,\Zbb)_{\square} \to \Acal;\; T\mapsto f.$$
    Therefore, we get $f\in A^{\circ}(\ast)\subset A^b(\ast).$
\end{proof}

\begin{lemma}\label{lem:aff op of adic space}
    Let $(A,A^+)$ be an affinoid pair over $L$ of weakly finite type.
    Then every affinoid localization $(A,A^+)_{\square}\to \Bcal$ is of the form $(A,A^+)_{\square}\to (B,B^+)_{\square}$ where $(A,A^+)\to (B,B^+)$ is an affinoid localization in the usual sense.
\end{lemma}
\begin{proof}
    By Lemma \ref{lem:adic space lcoal}, we may assume that $(A,A^+)_{\square}\to \Bcal$ is a simple rational localization.
    In this case, the claim is clear.
\end{proof}

\begin{definition}
    Let $(A,A^+)$ be an algebraic-affinoid pair over $L_{\square}$.
    An \textit{affinoid pair of definition of $(A,A^+)$} is an affinoid pair $(A',A'^+)$ over $L$ of weakly finite type, where $A'$ is an affinoid $L$-algebra of definition of $A$ and $A'^+\subset A^+$ such that $(A,A^+)_{\square}=(A,A'^+)_{\square}$.
\end{definition}

\begin{lemma}\label{lem:basic property of aff pair of def}
    Let $(A,A^+)$ be an algebraic-affinoid pair over $L_{\square}$.
\begin{enumerate}
    \item Let $A'$ be an affinoid $L$-algebra of definition of $A$.
    If the image of $A'$ in $A/\sqrt{(0)}$ coincides with $(A/\sqrt{(0)})^b$, then $(A',A'(\ast)\cap A^+)$ is an affinoid pair of definition of $(A,A^+)$.
    Moreover, there exists at least one affinoid $L$-algebra of definition $A'$ of $A$ satisfying the above condition.
    \item Let $(B,B^+)$ be an affinoid pair over $L$ of weakly finite type, and $f\colon B \to A$ be a morphism of $L_{\square}$-algebras.
    Then there exists an algebraic-affinoid pair of definition $(A',A'^+)$ of $(A,A^+)$ such that $f\colon (B,B^+)\to (A,A^+)$ factors through $(A',A'^+)$.
\end{enumerate}
\end{lemma}
\begin{proof}
    First, we prove (1).
    If the image of $A'$ in $A/\sqrt{(0)}$ coincides with $(A/\sqrt{(0)})^b$, then $A'(\ast)+\sqrt{(0)}=A^b(\ast)$ by Proposition \ref{prop:A^b integral}.
    Therefore, we get $(A'(\ast)\cap A^+)+\sqrt{(0)}=A^+$.
    In particular, $A^+$ is integral over $A'(\ast)\cap A^+$, so $(A,A'(\ast)\cap A^+)_{\square}=(A,A^+)_{\square}$ by \cite[Proposition 3.22]{And21}.
    Next, we prove the existence of $A'$.
    Since $(A/\sqrt{(0)})^b$ is an affinoid $L$-algebra, there exists a surjection 
    $$\bar{f}\colon L\langle T_1,\ldots, T_n\rangle \to (A/\sqrt{(0)})^b.$$ 
    We write $\bar{f}(T_i)=\bar{t}_i\in (A/\sqrt{(0)})^{\circ}(\ast)$.
    By Lemma \ref{lem:bounded nilpotent}, we take a lift $t_i\in A^{\circ}(\ast)$ of $\bar{t}_i$ for each $i$.
    As in the proof of Lemma \ref{lem:power-bounded elements}, we get a morphism 
    $$f\colon L\langle T_1,\ldots, T_n\rangle \to A;\; T_i\mapsto t_i.$$
    Let $A'$ (resp.\ $I$) be the image (resp.\ kernel) of $f$.
    By Lemma \ref{lem:rd sub}, $I$ is relatively discrete over $L\langle T_1,\ldots, T_n\rangle$.
    Therefore, $A'=L\langle T_1,\ldots, T_n\rangle/I$ is relatively discrete and finite over $L\langle T_1,\ldots, T_n\rangle$.
    Hence, $A'$ is an affinoid $L$-algebra.
    By construction, the image of $A'$ in $A/\sqrt{(0)}$ coincides with $(A/\sqrt{(0)})^b$, and $A/\sqrt{(0)}$ is relatively discrete over $A'$.
    Therefore, for any $n\geq 0$, $\sqrt{(0)}^n/\sqrt{(0)}^{n+1}$ is also relatively discrete over $A'$.
    Since the class of relatively discrete objects is closed under extensions by Remark \ref{rem:rd extension}, $A$ is also relatively discrete over $A'$.
    Hence, $A'$ is an affinoid $L$-algebra of definition of $A$, which completes the proof of (1).
    The claim (2) follows from Lemma \ref{lem:basic properties of aff of def} (2).
\end{proof}

\begin{proposition}\label{prop:Zariski open}
Let $\Acal=(A,A^+)_{\square}$ be an algebraic-affinoid analytic $L_{\square}$-algebra.
Then, there is an affinoid covering $\{\Acal\to \Acal_i\}_{i=1}^n$ such that for each $i$, $\Acal_i$ is of the form $(B_i,B_i^+)_{\square}[1/f_i]$ for some affinoid pair $(B_i,B_i^+)$ over $L$ of weakly finite type and $f_i\in B_i$.
\end{proposition}
\begin{proof}
    Let $(A',A'^+)$ be an affinoid pair of definition of $(A,A^+)$, and $f_1,\ldots,f_n\in A$ be a set of generators of $A$ as an $A'$-algebra.
	Then the standard affinoid covering defined by $1,f_1\ldots,f_n$ satisfies the required properties.
\end{proof}

\begin{remark}
The choice of $f_1,\ldots,f_n$ defines an immersion $\Spec A\to \Pbb_{A'}^n$.
We can extend this immersion to $\AnSpec \Acal \to \Pbb^{n,\an}_{(A',A'^+)_{\square}}$, and the affinoid covering in the above proof is the pullback of the standard open covering of $\Pbb^{n,\an}_{(A',A'^+)_{\square}}$. 
\end{remark}

\begin{definition}
\begin{enumerate}
\item By Lemma \ref{lem:basic properties of affinoid localizations}, the family of affinoid coverings define a Grothendieck topology on $\AlgAff_{L_{\square}}^{\op}$.
We call this topology the \textit{analytic topology}.
\item For $\Acal\in \AlgAff_{L_{\square}}$, we write $\AnSpec \Acal$ for the analytic sheaf $\AlgAff_{L_{\square}}\to \Set$ represented by $\Acal$, where we note that the presheaf represented by $\Acal$ is an analytic sheaf by Lemma \ref{lem:basic properties of affinoid localizations} (1) (cf.\ \cite[Lemma 2.4.7 (i)]{Mann22}).
\item Let $\Acal\in \AlgAff_{L_{\square}}$, and let $F$ be an analytic sheaf on $\AlgAff_{L_{\square}}^{\op}$.
We call an injection $F\to \AnSpec \Acal$ an \textit{open immersion} if the morphism
$$\varinjlim_{\AnSpec \Bcal\to F}\AnSpec\Bcal \to F$$
is an isomorphism, where the colimit is taken over all affinoid localizations $\AnSpec\Bcal \to \AnSpec \Acal$ that factors through $F$.
\item Let $F,G$ be analytic sheaves on $\AlgAff_{L_{\square}}^{\op}$.
We call an injection $F\to G$ an \textit{open immersion} if, for any $\Acal\in \AlgAff_{L_{\square}}$ and $\AnSpec\Acal \to G$, the injection $F\times_G \AnSpec \Acal \to \AnSpec \Acal$ is an open immersion.
\item Let $F$ be an analytic sheaf on $\AlgAff_{L_{\square}}^{\op}$.
We call $F$ is an \textit{algebraic-analytic space} over $L_{\square}$ if the morphism
$$\varinjlim_{\AnSpec \Acal\to F}\AnSpec\Acal \to F$$
is an isomorphism, where the colimit is taken over all open immersions $\AnSpec\Acal \to F$.
Let $\AlgAnSp_{L_{\square}}$ denote the category of algebraic-analytic spaces over $L_{\square}$.
\end{enumerate}
\end{definition}

\begin{proposition}\label{prop:adic space to alg-aff space}
	There is a natural fully faithful functor 
	\begin{align*}
	\{X \mid \text{an adic space over $L$ locally of weakly finite type}\}\to \AlgAnSp_{L_{\square}};\; X\mapsto X_{\square}
	\end{align*}
	which is an extension of the functor $\Spa(A,A^+)\mapsto \AnSpec (A,A^+)_{\square}$. 
\end{proposition}
\begin{proof}
    It easily follows from Example \ref{ex:classical affinoid cover} and Lemma \ref{lem:aff op of adic space}.
\end{proof}


\section{$(\varphi,\Gamma)$-modules over perfect period rings and imperfect period rings}

\subsection{$(\varphi,\Gamma)$-modules over perfect period rings}\label{subsection:perfect}
Let $\pi^{\flat}\in \Ocal_{\hat{K}_{\infty}^{\flat}}$ be a pseudo-uniformizer such that $\pi^{\flat\sharp}/p\in \Ocal_{\hat{K}_{\infty}}^{\times}$.
We set 
\begin{align*}
&Y_{\bar{K}}=\Spa(W(\Ocal_{C^{\flat}}))\setminus \{p[\pi^{\flat}]=0\},\\
&Y_{K_{\infty}}=\Spa(W(\Ocal_{\hat{K}_{\infty}^{\flat}}))\setminus \{p[\pi^{\flat}]=0\},
\end{align*}
which are analytic adic spaces over $\Qbb_p$ equipped with a Frobenius automorphism $\varphi$ and a continuous action of $G_K$, respectively $\Gamma_K$.
We also define the Fargues-Fontaine curves $X_{\bar{K}}=Y_{\bar{K}}/\varphi^{\Zbb}$ and $X_{K_{\infty}}=Y_{K_{\infty}}/\varphi^{\Zbb}$, which have a continuous action of $G_K$ and $\Gamma_K$, respectively.

There exist surjective continuous morphisms 
$$\kappa \colon Y_{\bar{K}} \to (0,\infty), \kappa \colon Y_{K_{\infty}} \to (0,\infty)$$ 
defined by 
$$\kappa(x)=\frac{\log \lvert[\pi^{\flat}](\tilde{x})\rvert}{\log \lvert p(\tilde{x})\rvert},$$
where $\tilde{x}$ is the maximal generalization of $x$ (cf.\ \cite[12.2]{SW20}).
For a closed interval $I=[r,s]\subset [0,\infty)$ (such that $r,s\in \Qbb$, $I\neq [0,0]$), let $Y_{\bar{K}}^{I}, Y_{K_{\infty}}^{I}$ denote the interior of the preimages of $I$ under $\kappa$.
We set
\begin{align*}
&\tilde{B}_{\bar{K}}^{I}=\Ocal_{Y_{\bar{K}}}(Y_{\bar{K}}^{I}),\\
&\tilde{B}_{\bar{K}}^{I,+}=\Ocal^{+}_{Y_{\bar{K}}}(Y_{\bar{K}}^{I}),\\
&\tilde{B}_{K_{\infty}}^{I}=\Ocal_{Y_{K_{\infty}}}(Y_{K_{\infty}}^{I}),\\
&\tilde{B}_{K_{\infty}}^{I,+}=\Ocal^{+}_{Y_{K_{\infty}}}(Y_{K_{\infty}}^{I}).
\end{align*}

It is easy to check that the continuous action of $G_K$ (resp.\ $\Gamma_K$) on $Y_{\bar{K}}$ (resp.\ $Y_{K_{\infty}}$) 
induces the action on $\tilde{B}_{\bar{K}}^{I}$
(resp.\ $\tilde{B}_{K_{\infty}}^{I}$),
and the Frobenius automorphism $\varphi \colon Y_{\bar{K}} \to Y_{\bar{K}}$ (resp.\ $\varphi \colon Y_{K_{\infty}} \to Y_{K_{\infty}}$) induces the Frobenius automorphism
$\varphi \colon \tilde{B}_{\bar{K}}^{I} \overset{\sim}{\to} \tilde{B}_{\bar{K}}^{I/p}$  (resp.\ $\varphi \colon \tilde{B}_{K_{\infty}}^{I} \overset{\sim}{\to} \tilde{B}_{K_{\infty}}^{I/p}$), where $I/p=[r/p,s/p]$.

\begin{remark}\label{rem:nullstellensatz FF curve}
    The Banach $\Qbb_p$-algebras $\tilde{B}_{K_{\infty}}^{I}$ and $\tilde{B}_{\bar{K}}^{I}$ defined above are Jacobson-Tate rings (Definition \ref{def:Jacobson-Tate}), which follows from \cite[Theorem 3.7]{KM25}.
    A Banach $\tilde{B}_{K_{\infty}}^{I}$-algebra (resp.\ $\tilde{B}_{\bar{K}}^{I}$-algebra) topologically of finite type is also a Jacobson-Tate ring.
\end{remark}
Let $I\subset (0,\infty)$ be a closed interval, and let $(A,A^+)$ be an algebraic-affinoid pair over $\Qbb_p$ or a complete Tate affinoid pair over $\Qbb_p$.
We set $\Acal=(A,A^+)_{\square}$.
We define analytic $\Acal$-algebras
\begin{align*}
    &\tilde{\Bcal}_{\bar{K},\Acal}^{I}=(\tilde{B}_{\bar{K}}^{I},\tilde{B}_{\bar{K}}^{I,+})_{\square}\otimes_{\Qbb_{p,\square}}^{\Lbb}\Acal,\\
    &\tilde{\Bcal}_{K_{\infty},\Acal}^{I}=(\tilde{B}_{K_{\infty}}^{I},\tilde{B}_{K_{\infty}}^{I,+})_{\square}\otimes_{\Qbb_{p,\square}}^{\Lbb}\Acal.
\end{align*}
We also define $A$-algebras
\begin{align*}
&\tilde{B}_{\bar{K},A}^{I}=\tilde{B}_{\bar{K}}^{I}\otimes_{\Qbb_{p,\square}}A,\\
&\tilde{B}_{K_{\infty},A}^{I}=\tilde{B}_{K_{\infty}}^{I}\otimes_{\Qbb_{p,\square}}A.
\end{align*}

\begin{lemma}\label{lem:underlying algebra}
    The underlying (condensed) ring of $\tilde{\Bcal}_{\bar{K},\Acal}^{I}$ (resp.\ $\tilde{\Bcal}_{K_{\infty},\Acal}^{I}$) is $\tilde{B}_{\bar{K},A}^{I}$ (resp.\ $\tilde{B}_{K_{\infty},A}^{I}$).
    In particular, it only depends on $A$ and does not depend on $A^+$.
\end{lemma}

\begin{proof}
Since both can be proved in the same way, we only prove for $\tilde{\Bcal}_{\bar{K},\Acal}^{I}$.
Since $\Qbb_{p,\square}\to \Acal$ is steady by Remark \ref{rem:every morph steady}, the underlying $\Qbb_{p,\square}$-algebra of $\tilde{\Bcal}_{\bar{K},\Acal}^{I}$ is isomorphic to $\tilde{B}_{\bar{K}}^{I}\otimes_{\Qbb_{p,\square}}^{\Lbb}\Acal$.
Since $\tilde{B}_{\bar{K}}^{I}$ is a nuclear $\Qbb_{p,\square}$-module by \cite[Corollary 3.16]{RJRC22}, $\tilde{B}_{\bar{K}}^{I}\otimes_{\Qbb_{p,\square}}^{\Lbb}A$ is also a nuclear object of $\Dcal(A_{\square})$.
By Lemma \ref{lem:A^+ nuclear}, $\tilde{B}_{\bar{K},A}^{I}$ is already $\Acal$-complete.
In other words, the natural morphism 
$$\tilde{B}_{\bar{K}}^{I}\otimes_{\Qbb_{p,\square}}^{\Lbb}A\to \tilde{B}_{\bar{K}}^{I}\otimes_{\Qbb_{p,\square}}^{\Lbb}\Acal$$
is an isomorphism.
Since $A$ is a flat $\Qbb_{p,\square}$-module (Example \ref{ex:flat module}), we get the claim.
\end{proof}

\begin{remark}
If $A$ is a Banach $\Qbb_p$-algebra, then $\tilde{B}_{\bar{K},A}^{I}$ (resp.\ $\tilde{B}_{K_{\infty},A}^{I}$) is equal to the Banach $\Qbb_p$-algebra $\tilde{B}_{\bar{K}}^{I}\hotimes_{\Qbb_p}A$ (resp.\ $\tilde{B}_{K_{\infty}}^{I}\hotimes_{\Qbb_p}A$).
\end{remark}

\begin{definition}\label{def:perfect module}
\begin{enumerate}
    \item A \textit{$\varphi$-module} $M$ over $\tilde{\Bcal}_{\bar{K},\Acal}$ is a family of finite projective $\tilde{\Bcal}_{\bar{K},\Acal}^{I}$-modules $\{M^{I}\}_{I}$ for each closed interval $I\subset (0,\infty)$ and isomorphisms 
    \begin{align*}
        &\tau_{I,I'}\colon M^{I}\otimes_{\tilde{\Bcal}_{\bar{K},\Acal}^{I}} \tilde{\Bcal}_{\bar{K},\Acal}^{I'} \overset{\sim}{\to} M^{I'},\\
    &\Phi_{I}\colon M^{I}\otimes_{\tilde{\Bcal}_{\bar{K},\Acal}^{I},\varphi}\tilde{\Bcal}_{\bar{K},\Acal}^{I/p}\cong M^{I/p}
    \end{align*}
    for each closed intervals $I' \subset I$, satisfying usual cocycle conditions and $$\Phi_{I'}\circ\varphi^{\ast}\tau_{I,I'}=\tau_{I/p,I'/p} \circ \Phi_{I}.$$
    In the same way, we define a notion of $\varphi$-modules over $\tilde{\Bcal}_{K_{\infty},\Acal}$.
    Let $\VB_{\tilde{\Bcal}_{\bar{K},\Acal}}^{\varphi}$ (resp.\ $\VB_{\tilde{\Bcal}_{K_{\infty},\Acal}}^{\varphi}$) denote the category of $\varphi$-modules over $\tilde{\Bcal}_{\bar{K},\Acal}$ (resp.\ $\tilde{\Bcal}_{K_{\infty},\Acal}$).
    \item A \textit{$(\varphi,G_K)$-module} $M$ over $\tilde{\Bcal}_{\bar{K},\Acal}$ is a family of finite projective $\tilde{\Bcal}_{\bar{K},\Acal}^{I}$-modules $\{M^{I}\}_{I}$ for each closed interval $I\subset (0,\infty)$ with a continuous semilinear $G_K$-action on $M^{I}$ and $G_K$-equivariant isomorphisms 
     \begin{align*}
        &\tau_{I,I'}\colon M^{I}\otimes_{\tilde{\Bcal}_{\bar{K},\Acal}^{I}} \tilde{\Bcal}_{\bar{K},\Acal}^{I'} \overset{\sim}{\to} M^{I'},\\
    &\Phi_{I}\colon M^{I}\otimes_{\tilde{\Bcal}_{\bar{K},\Acal}^{I},\varphi}\tilde{\Bcal}_{\bar{K},\Acal}^{I/p}\cong M^{I/p}
    \end{align*}
    for each closed intervals $I' \subset I$, satisfying the same condition as above.
    In the same way, we define a notion of $(\varphi,\Gamma_K)$-modules over $\tilde{\Bcal}_{K_{\infty},\Acal}$.
    Let $\VB_{\tilde{\Bcal}_{\bar{K},\Acal}}^{\varphi,G_K}$ (resp.\ $\VB_{\tilde{\Bcal}_{K_{\infty},\Acal}}^{\varphi,\Gamma_K}$) denote the category of $(\varphi,G_K)$-modules over $\tilde{\Bcal}_{\bar{K},\Acal}$ (resp.\ $(\varphi,\Gamma_K)$-modules over $\tilde{\Bcal}_{K_{\infty},\Acal}$).
\end{enumerate}
\end{definition}

We present several remarks.
While the following refers only to $(\varphi,G_K)$-modules over $\tilde{B}_{\bar{K},A}$, the same applies to other cases as well.

\begin{remark}\label{rem:A^+ independent}
    The notion of finite projective $\tilde{\Bcal}_{\bar{K},\Acal}^{I}$-modules only depends on $A$ and does not depend on $A^+$. 
    Moreover, for a finite projective $\tilde{\Bcal}_{\bar{K},\Acal}^{I}$-module $M$, there are isomorphisms
    \begin{align*}
        &M^{I}\otimes_{\tilde{\Bcal}_{\bar{K},\Acal}^{I}} \tilde{\Bcal}_{\bar{K},\Acal}^{I'}\cong M^{I}\otimes_{(\tilde{B}_{\bar{K},A}^{I})_{\square}} \tilde{B}_{\bar{K},A}^{I'}
        ,\\
        &M^{I}\otimes_{\tilde{\Bcal}_{\bar{K},\Acal}^{I}} \tilde{\Bcal}_{\bar{K},\Acal}^{I'}\cong M^{I}\otimes_{(\tilde{B}_{\bar{K},A}^{I})_{\square}} \tilde{B}_{\bar{K},A}^{I'}
        .
    \end{align*}
    Therefore, the notion of $\varphi$-modules (resp.\ $(\varphi,G_K)$-modules) over $\tilde{\Bcal}_{\bar{K},\Acal}$ depends only on $A$ and not on $A^+$.
    Thus, a $\varphi$-module (resp.\ $(\varphi,G_K)$-module) defined over $\tilde{\Bcal}_{\bar{K},\Acal}$ will also be referred to as a $\varphi$-module (resp.\ $(\varphi,G_K)$-module) over $\tilde{B}_{\bar{K},A}$.
    In practice, it is often more convenient to use the formulation over $\tilde{\Bcal}_{\bar{K},\Acal}$ when taking covers of $\Acal$ and considering descent, while in other situations the viewpoint over $\tilde{B}_{\bar{K},A}$ is preferable.
\end{remark}

\begin{remark}
    We will often omit $\tau$ from the notation, and call the morphism 
    $$\tau_{I,I'}\colon M^{I}\to M^{I}\otimes_{\tilde{\Bcal}_{\bar{K},\Acal}^{I}} \tilde{\Bcal}_{\bar{K},\Acal}^{I'} \overset{\sim}{\to} M^{I'}$$
    the restriction morphism.
\end{remark}

\begin{remark}
    Let $M=\{M^{I}\}_I $ be a $(\varphi,G_K)$-module over $\tilde{\Bcal}_{\bar{K},\Acal}$.
    Then $$\Phi_{I}\colon M^{I}\otimes_{\tilde{\Bcal}_{\bar{K},\Acal}^{I},\varphi}\tilde{\Bcal}_{\bar{K},\Acal}^{I/p}\cong M^{I/p}$$ defines a $\varphi$-semilinear isomorphism $M^{I}\to M^{I/p}$.
    We also write it as $\Phi_{I}$.
    We often omit $I$ from the notation.
    For $n\in \Zbb$ and $I'\subset I/p^n$, we also denote the composition of $\Phi^n \colon M^{I}\to M^{I/p^n}$ and the restriction morphism $M^{I/p^n} \to M^{I'}$ by $\Phi^n\colon M^{I}\to M^{I'}$.
\end{remark}

\begin{remark}
    We write $X_{\bar{K},\Acal}=X_{\bar{K}}\times_{\AnSpec\Qbb_{p,\square}}\AnSpec \Acal$, where we regard $X$ as an analytic space in the sense of Clausen-Scholze (\cite[Definition 13.2]{AG}).
    Let $M=\{M^{I}\}_{I}$ be a $(\varphi,G_K)$-module over $\tilde{\Bcal}_{\bar{K},\Acal}$.
    Then it defines a $G_K$-equivariant object of $\Dcal(X_{\bar{K},\Acal})$, so we can regard $M$ as a $G_K$-equivariant ``vector bundle'' on $X_{\bar{K},\Acal}$.
    We note that the notion of ``vector bundle'' is very subtle, see the remark below.
\end{remark}

\begin{remark}
We assume that $\Acal$ is an algebraic-affinoid analytic $\Qbb_{p,\square}$-algebra.
Let $M$ be an object of $\Dcal(Y_{\bar{K},\Acal})$, where $Y_{\bar{K},\Acal}=Y_{\bar{K}}\times_{\AnSpec\Qbb_{p,\square}}\AnSpec \Acal$.
In this case, even if the restriction $M^{I_i} \in \Dcal(Y_{\bar{K},\Acal}^{I_i})$ of $M$ to the open subspace $Y_{\bar{K},\Acal}^{I_i} \subset Y_{\bar{K},\Acal}$ is a finite projective $\tilde{\Bcal}_{\bar{K},\Acal}^{I_i}$-module for closed intervals $I_i\subset (0,\infty)$ such that $\cup I_i=(0,\infty)$, we do not know that $M^{I} \in \Dcal(Y_{\bar{K},\Acal}^{I})$ is a finite projective $\tilde{\Bcal}_{\bar{K},\Acal}^{I}$-module for every closed interval $I\subset (0,\infty)$.
A problem is that we do not know whether there is a descent theorem for finite projective modules over non-Fredholm analytic animated rings.
We note that we cannot apply Theorem \ref{thm:descent of finite projective modules}, since the restriction morphism $(\tilde{B}_{\bar{K},A}^{I})_{\square}\to (\tilde{B}_{\bar{K},A}^{I'})_{\square}$ is not flat, that is, the base change functor along it is not $t$-exact.
There is no such a problem if $A$ is a Banach $\Qbb_p$-algebra.
\end{remark}

\begin{remark}\label{rem:restriction of closed intervals}
    \begin{enumerate}
    \item
    The set of closed intervals used in Definition \ref{def:perfect module} can be restricted to a set of closed intervals of the form $[p^{m},p^{n}]$ for $m,n \in \Zbb$,
    since for any closed interval $I\subset (0,\infty)$ there exist $m,n\in\Zbb$ such that $I\subset[p^m,p^n]$.
    \item Let $l$ be any integer.
    Then the set of closed intervals used in Definition \ref{def:perfect module} can be restricted to a set of closed intervals of the form $[p^{m},p^{n}]$ for $m,n \in \Zbb$ such that $m\leq n \leq l$.
    It follows from the same argument as in \cite[Proposition 12.3.5]{SW20}.
    \end{enumerate}
\end{remark}

\begin{remark}\label{rem:tensor and int hom}
    Let $M=\{M^{I}\}_I$ and $N=\{N^{I}\}_I$ be $(\varphi,G_K)$-modules over $\tilde{B}_{\bar{K},A}$.
    Then $$M\otimes N=\{M^{I}\otimes_{\tilde{\Bcal}_{\bar{K},\Acal}^{I}}N^{I}\}_{I}$$ and $$\intHom(M,N)=\{\intHom_{\tilde{\Bcal}_{\bar{K},\Acal}^{I}}(M^{I},N^{I})\}_{I}$$ become $(\varphi,G_K)$-modules over $\tilde{\Bcal}_{\bar{K},\Acal}$ in the natural way. 
    We call it the tensor product (resp.\ internal hom) of $(\varphi,G_K)$-modules over $\tilde{\Bcal}_{\bar{K},\Acal}$.
    It is easy to show that the set $\Hom_{\varphi,G_K}(M,N)$ of morphisms of $(\varphi,G_K)$-modules is isomorphic to $$\Hom_{\varphi,G_K}(\tilde{B}_{\bar{K},A},\intHom(M,N)),$$ where $\tilde{B}_{\bar{K},A}=\{\tilde{B}_{\bar{K},A}^{I}\}_{I}$.
\end{remark}

The main theorem in this subsection is the following.

\begin{theorem}\label{thm:FF curve descent}
    The natural functor
    $$\VB_{\tilde{\Bcal}_{K_{\infty},\Acal}}^{\varphi,\Gamma_K} \to \VB_{\tilde{\Bcal}_{\bar{K},\Acal}}^{\varphi,G_K};\; M=\{M^{I}\}_{I} \mapsto \{M^{I}\otimes_{\tilde{\Bcal}_{K_{\infty},\Acal}^{I}} \tilde{\Bcal}_{\bar{K},\Acal}^{I}\}_{I}$$
    is a categorical equivalence.
    The quasi-inverse functor is given by
    $$\VB_{\tilde{\Bcal}_{\bar{K},\Acal}}^{\varphi,G_K}\to \VB_{\tilde{\Bcal}_{K_{\infty},\Acal}}^{\varphi,\Gamma_K} ;\; N=\{N^{I}\}_{I}\mapsto \{\Gamma(H_K,N^{I})\}_{I}.$$
\end{theorem}

We note that if $A$ is an affinoid $\Qbb_p$-algebra, then Theorem \ref{thm:FF curve descent} is well-known (cf.\ \cite[5.1]{EGH23}). 
To generalize the result, we translate it in terms of condensed mathematics.
Let $m\leq n$ be integers and we write $I=[p^m,p^n]$. 
It is enough to show that there is a natural categorical equivalence $\FP_{\tilde{B}_{K_{\infty},A}^{I}}^{\Gamma_K} \simeq \FP_{\tilde{B}_{\bar{K},A}^{I}}^{G_K}$, where $\FP_{\tilde{B}_{\bar{K},A}^{I}}^{G_K}$ (resp.\ $\FP_{\tilde{B}_{K_{\infty},A}^{I}}^{\Gamma_K}$) is the category of finite projective $\tilde{B}_{\bar{K},A}^{I}$-modules (resp.\ $\tilde{B}_{K_{\infty},A}^{I}$-modules) with a continuous semilinear action of $G_K$ (resp.\ $\Gamma_K$).
To prove this, we prove the following propositions.

\begin{proposition}\label{prop:almost Galois descent}
\begin{enumerate}
\item There is an isomorphism 
$\tilde{B}_{\bar{K}}^{I}\cong \hat{\bigoplus_{J}}\tilde{B}_{K_{\infty}}^{I}$ 
of $\tilde{B}_{K_{\infty}}^{I}$-modules for some set $J$ (for the definition of $\hat{\bigoplus_{J}}\tilde{B}_{K_{\infty}}^{I}$, see Example \ref{ex:flat module}).
In particular, $\tilde{B}_{\bar{K},A}^{I}$ is faithfully flat over $(\tilde{B}_{K_{\infty},A}^{I})_{\square}$.
\item The morphism $(\tilde{B}_{K_{\infty},A}^{I})_{\square}\to (\tilde{B}_{\bar{K},A}^{I})_{\square}$ of analytic $\Qbb_{p,\square}$-algebras satisfies the $\ast$-descent.
\end{enumerate}
\end{proposition}
\begin{proof}
The existence of the isomorphism can be proved in the same way as in \cite[Proposition 11.2.15]{FF18}.
The faithfully flatness of $\tilde{B}_{\bar{K},A}^{I}=\tilde{B}_{\bar{K}}^{I}\otimes_{\Qbb_{p,\square}}^{\Lbb}A$ over $(\tilde{B}_{K_{\infty},A}^{I})_{\square}=(\tilde{B}_{K_{\infty}}^{I})_{\square}\otimes_{\Qbb_{p,\square}}^{\Lbb}A_{\square}$ follows from Example \ref{ex:flat module}.
Part (2) follows from Proposition \ref{prop:ff descent}.
\end{proof}

There is a natural continuous $\tilde{B}_{K_{\infty}}^{I}$-linear action of $H_K$ on $\tilde{B}_{\bar{K}}^{I}$.
From this we obtain an augmented cosimplicial object $\{C(H_K^n,\tilde{B}_{\bar{K}}^{I})\}_{[n] \in \Delta} \to \tilde{B}_{\bar{K}}^{I}$.

\begin{proposition}\label{prop: Cech nerve}
The augmented cosimplicial object $$\{C(H_K^n,\tilde{B}_{\bar{K}}^{I})_{\square}\}_{[n] \in \Delta} \to (\tilde{B}_{\bar{K}}^{I})_{\square}$$ is the \v{C}ech nerve of the morphism $(\tilde{B}_{K_{\infty}}^{I})_{\square}\to (\tilde{B}_{\bar{K}}^{I})_{\square}$
\end{proposition}
\begin{proof}
By \cite[Proposition 6.1.2.11]{HTT}, it is enough to show that the natural morphism
$$\tilde{B}_{\bar{K}}^{I} \otimes_{(\tilde{B}_{K_{\infty}}^{I})_{\square}}^{\Lbb} \tilde{B}_{\bar{K}}^{I} \to C(H_K, \tilde{B}_{\bar{K}}^{I})$$
is an isomorphism.
We show that the morphism 
$$\tilde{B}_{\bar{K}}^{I,+} \otimes_{(\tilde{B}_{K_{\infty}}^{I,+})_{\square}}^{\Lbb} \tilde{B}_{\bar{K}}^{I,+} \to C(H_K, \tilde{B}_{\bar{K}}^{I,+})$$
is an almost isomorphism, that is, both the kernel and the cokernel of this morphism are annihilated by $[\pi^{\flat}]^{1/p^n}$ for every $n$.
Let $L$ be a finite extension of $K_{\infty}$.
We define $\tilde{B}_L^{I,+}$ as before.
Then $\tilde{B}_{\bar{K}}^{I,+}$ is the $[\pi^{\flat}]$-adic completion of $\varinjlim_{L}\tilde{B}_L^{I,+}$.
By \cite[Proposition 2.12.10]{Mann22}, $\tilde{B}_{\bar{K}}^{I,+} \otimes_{(\tilde{B}_{K_{\infty}}^{I,+})_{\square}}^{\Lbb} \tilde{B}_{\bar{K}}^{I,+}$ is the $[\pi^{\flat}]$-adic completion of $\varinjlim_{L}\tilde{B}_L^{I,+} \otimes_{(\tilde{B}_{K_{\infty}}^{I,+})_{\square}}^{\Lbb} \tilde{B}_L^{I,+}$.
On the other hand, $C(H_K, \tilde{B}_{\bar{K}}^{I,+})$ is the $[\pi^{\flat}]$-adic completion of $\varinjlim_{L} C(\Gal(L/K_{\infty}), \tilde{B}_L^{I,+})$, since both are $[\pi^{\flat}]$-adically complete and both coincide modulo $[\pi^{\flat}]$.
Therefore, it is enough to show that $$\tilde{B}_{L}^{I,+} \otimes_{(\tilde{B}_{K_{\infty}}^{I,+})_{\square}}^{\Lbb} \tilde{B}_{L}^{I,+} \to C(\Gal(L/K_{\infty}), \tilde{B}_L^{I,+})$$ is an almost isomorphism.
It follows from \cite[Proposition 11.2.15]{FF18}.
\end{proof}

\begin{proof}[Proof of Theorem \ref{thm:FF curve descent}]
It is enough to show that there is a natural categorical equivalence $\FP_{\tilde{B}_{K_{\infty},A}^{I}}^{\Gamma_K} \simeq \FP_{\tilde{B}_{\bar{K},A}^{I}}^{G_K}$ for $I=[p^m,p^n]$ with $m,n\in \Zbb$.
We prove the functors 
\begin{align*}
    &\alpha \colon \FP_{\tilde{B}_{K_{\infty},A}^{I}}^{\Gamma_K} \to \FP_{\tilde{B}_{\bar{K},A}^{I}}^{G_K} ;\;
    M \mapsto M\otimes_{(\tilde{B}_{K_{\infty},A}^{I})_{\square}} (\tilde{B}_{\bar{K},A}^{I})_{\square}, \\
    &\beta \colon \FP_{\tilde{B}_{\bar{K},A}^{I}}^{G_K} \to \FP_{\tilde{B}_{K_{\infty},A}^{I}}^{\Gamma_K} ;\; 
    N \mapsto \Gamma(H_K,N)
\end{align*}
are well-defined and quasi-inverse to each other.
The well-definedness of $\alpha$ is clear.
Next, we prove $\beta\circ\alpha\simeq\id$.
We have a cosimplicial object $$\{X_n\}_{[n]\in \Delta}=\{C(H_K^n, M\otimes_{(\tilde{B}_{K_{\infty},A}^{I})_{\square}} (\tilde{B}_{\bar{K},A}^{I})_{\square})\}_{[n]\in \Delta}$$ corresponding to the action of $H_K$ on $M\otimes_{(\tilde{B}_{K_{\infty},A}^{I})_{\square}} (\tilde{B}_{\bar{K},A}^{I})_{\square}$ by Lemma \ref{lem:stacky definition}, and we have an equivalence $$R\Gamma(H_K,M\otimes_{(\tilde{B}_{K_{\infty},A}^{I})_{\square}} (\tilde{B}_{\bar{K},A}^{I})_{\square})=\varprojlim_{[n]\in \Delta}X_n$$
in $\Dcal(A_{\square})$ by Corollary \ref{cor:stacky cohomology comparison}.
From the definition, the $0$-th cohomology of $$R\Gamma(H_K,M\otimes_{(\tilde{B}_{K_{\infty},A}^{I})_{\square}} (\tilde{B}_{\bar{K},A}^{I})_{\square})$$ is equal to $\Gamma(H_K,M\otimes_{(\tilde{B}_{K_{\infty},A}^{I})_{\square}} (\tilde{B}_{\bar{K},A}^{I})_{\square})$.
On the other hand, by Proposition \ref{prop:almost Galois descent} (2) and Proposition \ref{prop: Cech nerve}, we have an isomorphism $M\cong R\Gamma(H_K,M\otimes_{(\tilde{B}_{K_{\infty},A}^{I})_{\square}} (\tilde{B}_{\bar{K},A}^{I})_{\square})$.
Therefore, we obtain a natural equivalence $\beta\circ\alpha\simeq\id$.

Next, we prove that $\beta$ is well-defined and $\alpha\circ\beta\simeq\id$.
For any object $N\in \FP_{\tilde{B}_{\bar{K},A}^{I}}^{G_K}$, we want to show that $\Gamma(H_K,N)$ is a finite projective $\tilde{B}_{K_{\infty},A}^{I}$-module, and that the natural morphism 
$$\Gamma(H_K,N)\otimes_{(\tilde{B}_{K_{\infty},A}^{I})_{\square}} (\tilde{B}_{\bar{K},A}^{I})_{\square}\to N$$
is an isomorphism.
There is no harm in regarding $N$ as an object $\FP_{\tilde{B}_{\bar{K},A}^{I}}^{H_K}$.
We set $M=R\Gamma(H_K,N)\in \Dcal((\tilde{B}_{K_{\infty},A}^{I})_{\square})$.
Then, by Proposition \ref{prop:almost Galois descent} (2), we have $$M\otimes_{(\tilde{B}_{K_{\infty},A}^{I})_{\square}}^{\Lbb} (\tilde{B}_{\bar{K},A}^{I})_{\square}\simeq N.$$
Moreover, by Proposition \ref{prop:descent finite projective Fredholm} (when $A$ is a Banach $\Qbb_p$-algebra), Theorem \ref{thm:descent of finite projective modules}, Remark \ref{rem:nullstellensatz FF curve}, and Proposition \ref{prop:almost Galois descent} (when $A$ is an algebraic-affinoid $\Qbb_{p,\square}$-algebra), we find that $M$ is a finite projective $\tilde{B}_{K_{\infty},A}^{I}$-module, which proves the claim.
\end{proof}

Finally, we prove that the notion of $(\varphi,G_K)$-modules over $\tilde{B}_{\bar{K},A}$ (or equivalently $(\varphi,\Gamma_K)$-modules over $\tilde{B}_{K_{\infty},A}$) is a generalization of the notion of continuous $G_K$-representations on finite projective $A$-modules.
Let $\Rep_A(G_K)$ denote the category of continuous $G_K$-representations on finite projective $A$-modules.

\begin{theorem}\label{thm:Galois representation fully faithful functor}
    There exists a fully faithful functor 
    $$\Rep_A(G_K) \to \VB_{\tilde{\Bcal}_{\bar{K},\Acal}}^{\varphi,\Gamma_K} ;\; V \mapsto V\otimes_{\Acal}\tilde{\Bcal}_{\bar{K},\Acal}=\{V\otimes_{\Acal}\tilde{\Bcal}_{\bar{K},\Acal}^{I}\}_{I},$$
    where $G_K$ acts on $V\otimes_{\Acal}\tilde{\Bcal}_{\bar{K},\Acal}^{I}$ by the diagonal action.
\end{theorem}
\begin{proof}
It is enough to show that the functor
$$\FP_A \to \VB_{\tilde{B}_{\bar{K},A}}^{\varphi} ;\; V \mapsto \{V\otimes_{A_{\square}}\tilde{B}_{\bar{K},\Acal}^{I}\}_{I}$$
is fully faithful, where $\FP_A$ is the category of finite projective $A$-modules.
We note that for $V,W \in \FP_A$, there is a natural isomorphism 
$$\intHom_A(V,W)\otimes_{A_{\square}}\tilde{B}_{\bar{K},A}\cong \intHom(V\otimes_{A_{\square}}\tilde{B}_{\bar{K},A}, W\otimes_{A_{\square}}\tilde{B}_{\bar{K},A})$$ 
as $\varphi$-modules over $\tilde{B}_{\bar{K},A}$, so it is enough to show that for $V\in \FP_A$, the natural morphism $V\to \Hom_{\varphi}(\tilde{B}_{\bar{K},A}, V\otimes_{A_{\square}}\tilde{B}_{\bar{K},A})$ is an isomorphism (Remark \ref{rem:tensor and int hom}).
A morphism $f\in \Hom_{\varphi}(\tilde{B}_{\bar{K},A}, V\otimes_{A_{\square}}\tilde{B}_{\bar{K},A})$ corresponds bijectively to $(f(1))_{I}=(x_{I})\in \prod_{I}\left(V\otimes_{A_{\square}}\tilde{B}_{\bar{K},A}^{I}\right)(\ast)$ satisfying the following condition:
\begin{itemize}
    \item For $n\in \Zbb$ and $I'\subset I/p^n$, we have $\Phi^n(x_{I})=x_{I'}$.
\end{itemize}
Then $(x_{I})_{I}$ as above corresponds bijectively to $x=x_{[1,p]}\in \left(V\otimes_{A_{\square}}\tilde{B}_{\bar{K},A}^{[1,p]}\right)(\ast)$ such that $\Phi(x)=x$ in $\left(V\otimes_{A_{\square}}\tilde{B}_{\bar{K},A}^{[1,1]}\right)(\ast)$.
In fact, we can recover $x_{[p^n,p^{n+1}]}$ from $x=x_{[1,p]}$ by $x_{[p^n,p^{n+1}]}=\Phi^{-n}(x)$, and then we can recover all $x_{I}$ by gluing them and restricting.
Therefore, it is enough to show that 
$$0\to V \to V\otimes_{A_{\square}}\tilde{B}_{\bar{K},A}^{[1,p]} \overset{\varphi-1}{\longrightarrow} V\otimes_{A_{\square}}\tilde{B}_{\bar{K},A}^{[1,1]} $$ 
is an exact sequence.
We can reduce to the case $V=A$ since $V$ is a finite projective $A$-module.
Moreover, we can reduce to the case $A=\Qbb_p$ since $A$ is flat over $\Qbb_p$.
In this case, the claim follows from the same argument as in the proof of \cite[Proposition 4.1.1]{FF18}.
\end{proof}

\subsection{$(\varphi,\Gamma)$-modules over imperfect period rings}\label{subsection:imperfect}

Let $\Acal=(A,A^+)_{\square}$ be as in the previous subsection.
In this subsection, we discuss a ``deperfection'' of $(\varphi,\Gamma_K)$-modules over $\tilde{\Bcal}_{K_{\infty},\Acal}$ based on the ideas of \cite{Por24}.
First, we recall coefficient rings used in the theory of $(\varphi,\Gamma_K)$-modules. 

\begin{construction}\label{construction:coefficient rings}
Let $K_0$ (resp.\ $K_0^{\prime}$) be the maximal unramified extension of $\Qbb_p$ contained in $K$ (resp.\ $K_{\infty}$), and let $e$ be the ramification index of the extension $K_{\infty}/K_{0}(\zeta_{p^{\infty}})$.
For a closed interval $[r,s]\subset (0,\infty)$, let $\Bbb_{K_0^{\prime}}^{[r,s]}$ denote the rational localization of $\Spa({K_0^{\prime}}\langle T\rangle, \Ocal_{K_0^{\prime}}\langle T\rangle)$ defined by $\lvert p^s\rvert \leq \lvert T \rvert \leq \lvert p^r\rvert$, and let $\Bbb_{K_0^{\prime}}^{(0,s]}$ denote the union of $\Bbb_{K_0^{\prime}}^{[r,s]}$ for all $0<r\leq s$.
Then as the construction in \cite[Chapitre 1]{Ber08B-pair}, for some rational number $0<s^{\prime}<1$, we can construct the following:
\begin{itemize}
\item A continuous action of $\Gamma_K$ on $\Bbb_{K_0^{\prime}}^{(0,\frac{ps^{\prime}}{(p-1)e}]}$.
\item A Frobenius endomorphism $\varphi\colon \Bbb_{K_0^{\prime}}^{(0,\frac{s^{\prime}}{(p-1)e}]} \to \Bbb_{K_0^{\prime}}^{(0,\frac{ps^{\prime}}{(p-1)e}]}$.
\item A $\Gamma_K$-equivariant morphism $Y_{K_{\infty}}^{(0,s^{\prime}]} \to \Bbb_{K_0^{\prime}}^{(0,\frac{ps^{\prime}}{(p-1)e}]}$ which is compatible with the Frobenius endomorphisms.
Moreover it satisfies that for every closed interval $[r,s]\subset (0,s^{\prime}]$, the preimage of $\Bbb_{K_0^{\prime}}^{[\frac{pr}{(p-1)e},\frac{ps}{(p-1)e}]}$ in $Y_{K_{\infty}}^{(0,s^{\prime}]}$ is $Y_{K_{\infty}}^{[r,s]}$, and the induced morphism $\Ocal(\Bbb_{K_0^{\prime}}^{[\frac{pr}{(p-1)e},\frac{ps}{(p-1)e}]})\to \tilde{B}_{K_{\infty}}^{[r,s]}$ is a closed embedding.
\end{itemize}
For a closed interval $I=[r,s]\subset (0,\infty)$, let $B_K^{I}$ (resp.\ $B_K^{I+}$) denote the image of the closed embedding $\Ocal(\Bbb_{K_0^{\prime}}^{\frac{pI}{(p-1)e}})\to \tilde{B}_{K_{\infty}}^{I}$ (resp.\ $\Ocal^+(\Bbb_{K_0^{\prime}}^{\frac{pI}{(p-1)e}})\to \tilde{B}_{K_{\infty}}^{I}$).
From the above, $B_K^{I}, B_K^{I,+}\subset \tilde{B}_{K_{\infty}}^{I}$ are stable under the action of $\Gamma_K$, and $\varphi \colon \tilde{B}_{K_{\infty}}^{I} \to \tilde{B}_{K_{\infty}}^{I/p}$ induces the Frobenius endomorphisms $\varphi \colon B_K^{I} \to B_K^{I/p}$ and $\varphi \colon B_K^{I,+} \to B_K^{I/p,+}$.
Let $B_{K,n}^{I}$ (resp.\ $B_{K,n}^{I,+}$) be the preimage of $B_{K}^{I/p^n}$ (resp.\ $B_{K}^{I/p^n,+}$) under the morphism $\varphi^n \colon \tilde{B}_{K_{\infty}}^{I}\to \tilde{B}_{K_{\infty}}^{I/p^n}.$
We define $B_{K,\infty}^{I}$ (resp.\ $B_{K,\infty}^{I,+}$) as $\varinjlim_n B_{K,n}^{I}\subset \tilde{B}_{K_{\infty}}^{I}$ (resp.\ $\varinjlim_n B_{K,n}^{I,+}$), where we regard it as a colimit of $\Qbb_{p,\square}$-algebras (resp.\ $\Zbb_{p,\square}$-algebras).
\end{construction}

\begin{definition}\label{def:coefficient ring Robba}
For $n\geq 0$ or $n=\infty$ and a closed interval $I\subset (0,s']$, we define analytic $\Acal$-algebra
\begin{align*}
    \Bcal_{K,n,\Acal}^I=(B_{K,n}^{I},B_{K,n}^{I,+})_{\square}\otimes_{\Qbb_{p,\square}}^{\Lbb} \Acal.
\end{align*}
We also define $A$-algebra
\begin{align*}
B_{K,n,A}^I=B_{K,n}^{I}\otimes_{\Qbb_{p,\square}} A.
\end{align*}
When $n=0$, then we omit $n$ from the notation.
\end{definition}

In the same way as in Lemma \ref{lem:underlying algebra}, we get the following.

\begin{lemma}\label{lem:underlying algebra Robba}
    For $n\geq 0$ or $n=\infty$ and a closed interval $I\subset (0,s']$, the underlying (condensed) ring of $\Bcal_{K,n,\Acal}^I$ is $B_{K,n,A}^I$.
    In particular, it only depends on $A$ and does not depend on $A^+$.
\end{lemma}

\begin{example}
    If $K=\Qbb_p$ (or $K$ is an unramified extension of $\Qbb_p$), we can construct the embeddings explicitly.
    We explain the construction.
    First, we fix a system $\varepsilon=(\zeta_{p^n})_{n\geq0}$ such that $\zeta_{p^n}\in\bar{\Qbb_p}$ is a primitive $p^n$-th root of unity and that $\zeta_{p^{n+1}}^p=\zeta_{p^n}$.
    Then we can regard $\varepsilon$ as an element of $\Ocal_{\hat{K}_{\infty}^{\flat}}=\varprojlim_{x\mapsto x^p} \Ocal_{\hat{K}_{\infty}}$.
    We write $\varpi=[\varepsilon]-1 \in W(\Ocal_{\hat{K}_{\infty}^{\flat}})$, where $[\varepsilon]$ is the Teichm\"{u}ller lift of $\varepsilon$.
    Then for a closed interval $I\subset (0,1)$, a morphism $\Ocal_{\Bbb_{\Qbb_p}}(\Bbb_{\Qbb_p}^{\frac{pI}{p-1}}) \to \tilde{B}_{\Qbb_{p}(\zeta_{p^{\infty}})}^{I} ;\; T\mapsto \varpi$ is well-defined and it gives a desired embedding.
    The actions of $\varphi$ and $\gamma\in\Gamma_{\Qbb_p}$ are given by $\varphi(\varpi)=(1+\varpi)^p-1$ and $\gamma(\varpi)=(1+\varpi)^{\chi(\gamma)}-1$ , where $\chi$ is the $p$-adic cyclotomic character.
\end{example}

\begin{remark}\label{rem:deperfection}
	By the construction in \cite{Col08}, $B_{\Qbb_p}^{I}$ is contained in $B_K^{I}$, where we regard $\tilde{B}_{\Qbb_p(\zeta_{p^{\infty}})}^{I}$ as a subalgebra of $\tilde{B}_{K_{\infty}}^{I}$ in the natural way, and $B_K^{I}$ is a finite free $B_{\Qbb_p}^{I}$-module (of rank $[K_{\infty}\colon \Qbb_p(\zeta_{p^{\infty}})]$).
	Moreover, the natural morphism 
	$$B_K^{I} \otimes_{(B_{\Qbb_p}^{I})_{\square}}^{\Lbb}B_{\Qbb_p,n}^{I}\to B_{K,n}^{I}$$ 
	is an isomorphism for every $n\geq 0$ or $n=\infty$.
\end{remark}

\begin{remark}\label{rem:open subspace perfect}
	For closed intervals $I_1 \subset I_2 \subset (0,s^{\prime}]$, the natural morphism
	$$B_K^{I_1} \otimes_{(B_K^{I_2})_{\square}}^{\Lbb}\tilde{B}_{K_{\infty}}^{I_2}\to \tilde{B}_{K_{\infty}}^{I_1}$$
	is an isomorphism.
    In fact, since the preimage of $\Bbb_{K_0^{\prime}}^{\frac{pI_1}{(p-1)e}}$ in $Y_{K_{\infty}}^{\frac{pI_2}{(p-1)e}}$ is $Y_{K_{\infty}}^{I_1}$, the claim follows from \cite[Proposition 4.12]{And21}.
\end{remark}

\begin{remark}\label{rem:open subspace imperfect}
    For closed intervals $I_1\subset I_2\subset (0,s^{\prime}]$, the natural morphism 
    $$B_K^{I_1} \otimes_{(B_K^{I_2})_{\square}}^{\Lbb}B_{K,n}^{I_2}\to B_{K,n}^{I_1}$$
	is an isomorphism for $n\geq 1$ or $n=\infty$ by the construction in \cite{Col08}
\end{remark}

For a positive integer $m$, let $\Gamma_{K,m}$ denote the preimage of $1+p^m\Zbb_p$ under the $p$-adic cyclotomic character $\chi \colon \Gamma_K \to \Zbb_p^{\times}$.
If $m$ is sufficiently large, then $\Gamma_{K,m}$ is isomorphic to $1+p^m\Zbb_p$, and it admits a natural rigid analytic group structure.
For a static $\Qbb_{p,\square}$-module $M$ with a continuous $\Gamma_K$-action, let $M^{m\mathchar`-\an}\subset M$ denote the submodule of (non-derived) $\Gamma_{K,m}$-analytic vectors for $m$ sufficiently large.

\begin{proposition}[{\cite[Theorem 4.4]{Ber16}, \cite[Corollary 5.6]{Por24}}]\label{prop:la vector of perfect witt ring}
    We set $I=[p^{-k},p^{-l}]\subset (0,s^{\prime}]$ for positive integers $k$ and $l$.
    \begin{enumerate}
    \item The subring $\tilde{B}_{K_{\infty},A}^{I,\la}$ of $\Gamma_K$-locally analytic vectors of $\tilde{B}_{K_{\infty},A}^{I}$ is equal to $B_{K,\infty,A}^{I}$.
    Moreover, $\tilde{B}_{K_{\infty},A}^{I}$ has no derived locally analytic vectors, that is, the natural morphism $\tilde{B}_{K_{\infty},A}^{I,\la}\to \tilde{B}_{K_{\infty},A}^{I,R\la}$ is an isomorphism.
    \item 
    If $l$ is sufficiently large, then the subring $\tilde{B}_{K_{\infty},A}^{[p^{-k},p^{-l}],l\mathchar`-\an}$ of $\Gamma_{K,l}$-analytic vectors of $\tilde{B}_{K_{\infty},A}^{[p^{-k},p^{-l}]}$ is contained in $B_{K,A}^{[p^{-k},p^{-l}]}$ for $k=l,l+1$.
    \end{enumerate}
\end{proposition}
\begin{proof}
    We have an isomorphism
    \begin{align*}
    \tilde{B}_{K_{\infty},A}^{I,R\la} &\simeq (\tilde{B}_{K_{\infty}}^{I}\otimes^{\Lbb}_{\Qbb_{p,\square}} A)^{R\la}\\
    &\simeq \tilde{B}_{K_{\infty}}^{I,R\la}\otimes^{\Lbb}_{\Qbb_{p,\square}} A\\
    &\simeq B_{K,\infty}^{I}\otimes^{\Lbb}_{\Qbb_{p,\square}} A\\
    &\simeq B_{K,\infty,A}^{I},
    \end{align*}
    where the second equivalence follows from Proposition \ref{prop:la functor and tensor product commute}, and the third equivalence follows from \cite[Theorem 4.4 (2)]{Ber16} and \cite[Corollary 5.6]{Por24}.
    It proves (1).
    We prove (2).
    By Proposition \ref{prop:an functor and tensor product commute}, we may assume $A=\Qbb_p$.
    We fix $l\geq 0$ such that $[p^{-l-1},p^{-l}]\subset (0,s^{\prime}]$.
Then for $m$ sufficiently large, $B_{K}^{[p^{-l-1},p^{-l}]}$ is $\Gamma_{K,m}$-analytic.
By \cite[Theorem 4.4 (1)]{Ber16}, for $m$ sufficiently large, we have
\begin{align*}
	\tilde{B}_{\Qbb_{p,\infty}}^{[p^{-l-1},p^{-l}],m+l\mathchar`-\an}\subset \tilde{B}_{\Qbb_{p,\infty}}^{[p^{-l-1},p^{-l}],m+l+1\mathchar`-\an}\subset B_{\Qbb_p,m}^{[p^{-l-1},p^{-l}]}.
\end{align*}
Therefore, for $m$ sufficiently large, we get
\begin{align*}
	\tilde{B}_{K_{\infty}}^{[p^{-l-1},p^{-l}],m+l\mathchar`-\an}
	&=(\tilde{B}_{\Qbb_{p,\infty}}^{[p^{-l-1},p^{-l}]}\otimes_{(B_{\Qbb_p}^{[p^{-l-1},p^{-l}]})_{\square}}B_{K}^{[p^{-l-1},p^{-l}]})^{m+l\mathchar`-\an}\\
	&=\tilde{B}_{\Qbb_{p,\infty}}^{[p^{-l-1},p^{-l}],m+l\mathchar`-\an}\otimes_{(B_{\Qbb_p}^{[p^{-l-1},p^{-l}]})_{\square}}B_{K}^{[p^{-l-1},p^{-l}]} \\
	&\subset B_{\Qbb_p,m}^{[p^{-l-1},p^{-l}]}\otimes_{(B_{\Qbb_p}^{[p^{-l-1},p^{-l}]})_{\square}}B_{K}^{[p^{-l-1},p^{-l}]}\\
	&= B_{K,m}^{[p^{-l-1},p^{-l}]},
\end{align*}
where the second equality follows from Proposition \ref{prop:an functor and tensor product commute}.
By applying $\varphi^{-m}$, we get
\begin{align*}
	\tilde{B}_{K_{\infty}}^{[p^{-m-l-1},p^{-m-l}],m+l\mathchar`-\an}\subset B_{K}^{[p^{-m-l-1},p^{-m-l}]}
\end{align*}
for $m$ sufficiently large, which proves (2).
By the same argument, we can show the case $k=l$.
\end{proof}

\begin{corollary}\label{cor:la vector of perfect witt ring}
    Let $I\subset (0,s^{\prime}]$ be a closed interval.
    Then the subring $\tilde{B}_{K_{\infty},A}^{I,\la}$ of $\Gamma_K$-locally analytic vectors of $\tilde{B}_{K_{\infty},A}^{I}$ is equal to $B_{K,\infty,A}^{I}$.
    Moreover, $\tilde{B}_{K_{\infty},A}^{I}$ has no derived locally analytic vectors, that is, the natural morphism $\tilde{B}_{K_{\infty},A}^{I,\la}\to \tilde{B}_{K_{\infty},A}^{I,R\la}$ is an isomorphism.
\end{corollary}
\begin{proof}
    Since there are $\Gamma_K$-equivariant isomorphisms
    \begin{align*}
        \varphi\colon \tilde{B}_{K_{\infty},A}^{I} \to \tilde{B}_{K_{\infty},A}^{I/p}, \; \varphi\colon B_{K,\infty,A}^{I} \to B_{K,\infty,A}^{I/p},
    \end{align*}
    we may assume $s\leq p^{-1}s^{\prime}$.
    Then we can take integers $k,l$ such that $I\subset [p^{-k},p^{-l}]\subset (0,s^{\prime}]$.
    Then the claim follows from Proposition \ref{prop:la functor and tensor product commute} and isomorphisms
    \begin{align*}
      &B_{K,A}^{I} \otimes_{(B_{K,A}^{[p^{-k},p^{-l}]})_{\square}}^{\Lbb}\tilde{B}_{K_{\infty},A}^{[p^{-k},p^{-l}]}\cong \tilde{B}_{K_{\infty},A}^{I},\\
      &B_{K,A}^{I} \otimes_{(B_{K,A}^{[p^{-k},p^{-l}]})_{\square}}^{\Lbb}B_{K,\infty,A}^{[p^{-k},p^{-l}]}\cong B_{K,\infty,A}^{I}.
    \end{align*}
\end{proof}

\begin{definition}
    For a closed interval $I\subset (0,\infty)$, we define
    $$B_{K,\infty,A}^{I}=\widetilde{B}_{K_{\infty},A}^{I,\la}\cong \widetilde{B}_{K_{\infty},A}^{I,R\la}$$
    as the subring of $\Gamma_K$-locally analytic vectors of $\tilde{B}_{K_{\infty},A}^{I}$.
    By Corollary \ref{cor:la vector of perfect witt ring}, this definition coincides with the previous definition when $I\subset (0,s^{\prime}]$.
    There is the $\Gamma_K$-action on $B_{K,\infty,A}^{I}$ and the Frobenius automorphism 
    $$\varphi\colon B_{K,\infty,A}^{I}\overset{\sim}{\to} B_{K,\infty,A}^{I/p}.$$
    We take an integer $n>0$ such that $I/p^n\subset (0,s']$, and we define an analytic $\Acal$-algebra $\Bcal_{K,\infty,\Acal}^I$ whose underlying algebra is $B_{K,\infty,A}^{I}$ so that 
    $$\varphi^n\colon \Bcal_{K,\infty,\Acal}^{I}\overset{\sim}{\to} \Bcal_{K,\infty,\Acal}^{I/p^n}$$
    becomes an isomorphism of analytic rings, where the analytic ring $\Bcal_{K,\infty,\Acal}^{I/p^n}$ is already defined in Definition \ref{def:coefficient ring Robba}.
    We note that it does not depend on the choice of $n$.
\end{definition}
\begin{remark}
    The natural morphism $B_{K,\infty}^{I}\otimes_{\Qbb_{p,\square}}^{\Lbb}A\to B_{K,\infty,A}^{I}$ is an isomorphism by Proposition \ref{prop:la functor and tensor product commute}.
    Moreover, we have also $\Bcal_{K,\infty}^{I}\otimes_{\Qbb_{p,\square}}^{\Lbb}\Acal\cong \Bcal_{K,\infty,\Acal}^{I}$.
\end{remark}

\begin{definition}\label{def:family of phi Gamma module perfect la}
    We define $\varphi$-modules and $(\varphi,\Gamma_K)$-modules over $\Bcal_{K,\infty,\Acal}$ by replacing $\tilde{\Bcal}_{K_{\infty},A}^{I}$ with $\Bcal_{K,\infty,A}^{I}$ in Definition \ref{def:perfect module}.
    Let $\VB_{B_{K,\infty,\Acal}}^{\varphi}$ (resp.\ $\VB_{B_{K,\infty,\Acal}}^{\varphi,\Gamma_K}$) denote the category of $\varphi$-modules (resp.\ $(\varphi,\Gamma_K)$-modules) over $\Bcal_{K,\infty,\Acal}$.
\end{definition}

\begin{remark}
    When $A=\Qbb_p$, then the notion of $(\varphi,\Gamma_K)$-modules over $\Bcal_{K,\infty}$ coincides with the notion of locally analytic vector bundles defined in \cite[Definition 4.4]{Por24}.
\end{remark}

\begin{definition}\label{def:family of phi Gamma module}
\begin{enumerate}
\item
A \textit{$(\varphi,\Gamma_K)$-module} $M$ over $\Bcal_{K,\Acal}$ is a family of finite projective $\Bcal_{K,\Acal}^{I}$-modules $\{M^{I}\}_{I}$ for some $t$ and $I\subset (0,t]$ with a continuous semilinear $\Gamma_K$-action on $M^{I}$ and $\Gamma_K$-equivariant isomorphisms 
\begin{align*}
    &\tau_{I,I'}\colon M^{I}\otimes_{\Bcal_{K,\Acal}^{I}} \Bcal_{K,\Acal}^{I'} \overset{\sim}{\to} M^{I'}\\
&\Phi_{I}\colon M^{I}\otimes_{\Bcal_{K,\Acal}^{I},\varphi} \Bcal_{K,\Acal}^{I/p}\cong M^{I/p}
\end{align*}
for $I'\subset I$, satisfying the same condition as in Definition \ref{def:perfect module}.
In this case, $M$ is said to be defined over $(0,t]$.

\item
Let $M=\{M^{I}\}_{I}$ (resp.\ $N=\{N^{I}\}_{I}$) be a $(\varphi,\Gamma_K)$-module over $\Bcal_{K,\Acal}$ defined over $(0,t]$ (resp.\ $(0,t^{\prime}]$).
Then a morphism $f\colon M \to N$ of $(\varphi,\Gamma_K)$-modules over $\Bcal_{K,\Acal}$ is a family of $\Gamma_K$-equivariant morphisms $f^{I}\colon M^{I}\to N^{I}$ of $\Bcal_{K,\Acal}^{I}$-modules,
for some $t^{\prime\prime}\leq t,t^{\prime}$ and any $I\subset(0,t^{\prime\prime}]$, compatible with $\tau_{I,I'}$ and $\Phi_{I}$.
Let $\VB_{\Bcal_{K,\Acal}}^{\varphi,\Gamma_K}$ denote the category of $(\varphi,\Gamma_K)$-modules over $\Bcal_{K,\Acal}$.
\end{enumerate}
\end{definition}

\begin{remark}
    Let $M$ be a $(\varphi,\Gamma_K)$-module over $\Bcal_{K,\Acal}$ defined over $(0,t]$.
    Then $M$ defines an object of $\Dcal(\Bbb_{K_0^{\prime},A}^{(0,\frac{pt}{(p-1)e}]})$.
\end{remark}

\begin{remark}
    For the same reason as in Remark \ref{rem:A^+ independent}, the notions of $\varphi$-modules and $(\varphi,G_K)$-modules over $\Bcal_{K,\infty,\Acal}$ (resp.\ $(\varphi,G_K)$-modules over $\Bcal_{K,\Acal}$) only depends on $A$ and does not depend on $A^+$. 
\end{remark}

\begin{remark}
    For the same reason as in Remark \ref{rem:restriction of closed intervals}, the set of closed intervals used above can be restricted to a set of closed intervals of the form $[p^{m},p^{n}]$ for $m,n \in \Zbb$.
\end{remark}

\begin{remark}
    As in Remark \ref{rem:tensor and int hom}, we can define a tensor product and an internal hom of $(\varphi,\Gamma_K)$-modules over $\Bcal_{K,\infty,\Acal}$ (resp.\ $\Bcal_{K,\Acal}$).
\end{remark}

\begin{remark}\label{rem:the reason why family}
If $A$ is an affinoid $\Qbb_p$-algebra, then the category of $(\varphi,\Gamma_K)$-modules over $\Bcal_{K,\Acal}$ is equivalent to the category of $(\varphi,\Gamma_K)$-modules over the Robba ring $\Rcal_A=\varinjlim_{0<s} \varprojlim_{0<r\leq s} B_{K,A}^{[r,s]}$ (\cite[Proposition 2.2.7]{KPX14}).
This categorical equivalence relies on the fact that the natural morphism $\varprojlim_{0<r\leq s} B_{K,A}^{[r,s]}\to \Rvarprojlim_{0<r\leq s} B_{K,A}^{[r,s]}$ is an isomorphism.
However, if $A$ is not a Banach $\Qbb_p$-algebra, then the above morphism is not necessarily an isomorphism.
Therefore, we use families of modules instead of $(\varphi,\Gamma_K)$-modules over the Robba ring.
\end{remark}

The goal of this subsection is to prove the following theorem.
\begin{theorem}\label{thm:deperfection}
    The categories $\VB_{\tilde{\Bcal}_{K_{\infty},\Acal}}^{\varphi,\Gamma_K}$, $\VB_{\Bcal_{K,\infty,\Acal}}^{\varphi,\Gamma_K}$, and $\VB_{\Bcal_{K,\Acal}}^{\varphi,\Gamma_K}$ are canonically equivalent each other.
\end{theorem}
\begin{remark}
    When $A=\Qbb_p$, then the equivalence between $\VB_{\tilde{\Bcal}_{K_{\infty}}}^{\varphi,\Gamma_K}$ and $\VB_{\Bcal_{K,\infty}}^{\varphi,\Gamma_K}$ is proved by Porat in \cite[Theorem 6.1]{Por24}.
\end{remark}

To prove this, we construct a ``deperfection'' of a finite projective $\tilde{B}_{K_{\infty},A}^{I}$-module $M^{I}$ with a continuous semilinear $\Gamma_K$-action.
In the following, we assume that $I\subset (0,\infty)$ is of the form $[p^{-k},p^{-l}]$ for integers $k$ and $l$ sufficiently large.

For the proof of Theorem \ref{thm:deperfection}, the following theorem is essential.

\begin{theorem}\label{thm:locally analytic descent}
\begin{enumerate}
    \item Let $M^{I}$ be a finite projective $\tilde{B}_{K_{\infty},A}^{I}$-module with a continuous semilinear $\Gamma_K$-action.
    Then $M^{I,\la}$ is a finite projective $B_{K,\infty,A}^I$-module, and $M^{I}$ has no derived analytic vectors (that is, $H^i(M^{I,R\la})=0$ for any $i>0$).
    Moreover, the natural morphism
    $$M^{I,\la}\otimes_{(B_{K,\infty,A}^{I})_{\square}} \tilde{B}_{K_{\infty},A}^{I} \to M^{I}$$
    is an isomorphism.
    In fact, for $m$ sufficiently large, the natural morphism
    $$M^{I,m\mathchar`-\an}\otimes_{(B_{K,A}^{I,m\mathchar`-\an})_{\square}} \tilde{B}_{K_{\infty},A}^{I} \to M^{I}$$
    is an isomorphism.
    \item Let $N^{I}$ be a finite projective $B_{K,\infty,A}^{I}$-module with a continuous semilinear $\Gamma_K$-action.
    Then the $\Gamma_K$-action on $N^{I}$ is locally analytic, and the natural morphism
    $$N^{I} \to (N^{I}\otimes_{(B_{K,\infty,A}^{I})_{\square}} \tilde{B}_{K_{\infty},A}^{I})^{\la}$$
    is an isomorphism.
    Moreover, let $N_n^{I}$ be a finite projective $B_{K,n,A}^{I}$-module with a continuous semilinear $\Gamma_K$-action for some integer $n\geq 0$.
    Then $N_n^{I}$ is a $\Gamma_{K,m}$-analytic representation for $m$ sufficiently large.
\end{enumerate}
\end{theorem}

Let us prove Theorem \ref{thm:locally analytic descent}.
First, we assume that $A$ is a Banach $\Qbb_p$-algebra.
In this case, we can apply the Tate-Sen methods (\cite{BC08}, \cite{Por24}).
We regard $\tilde{B}_{K_{\infty},A}^{I}$, $B_{K,m,A}^{I}$ as (classical) Banach $\Qbb_p$-algebras.
We have normalized trace morphisms $R_m \colon \tilde{B}_{K_{\infty}}^{I} \to B_{K,m}^{I}$ for $m \geq0$ which satisfy the Tate-Sen axioms (TS1)-(TS4) in \cite[Section 5]{Por24} by \cite[Proposition 1.1.12]{Ber08B-pair}, \cite[Corollary 9.5]{Col08}, and \cite[Example 5.5 (2)]{Por24}.
Note that in this setting, $H=\Ker(\chi\colon G=\Gamma_K \to \Zbb_{p}^{\times})$ is trivial, so the Tate-Sen axioms become much simpler. 
For example, (TS1) is trivial.
By taking the completed tensor product $-\hat{\otimes}_{\Qbb_p} A$, we obtain $A$-linear continuous morphisms $R_{m,A} \colon \tilde{B}_{K_{\infty},A}^{I} \to B_{K,m,A}^{I}$ for $m \geq0$.
It is easy to show that they also satisfy the Tate-Sen axioms (TS1)-(TS4).
\begin{example}
    We explain the construction of $R_m$ in the case $K=\Qbb_p$.
    We put $S_{m}=\{r/p^m \mid 0 \leq r \leq p^m-1, \}$ and $S_{\infty}=\bigcup_{m} S_m$.
    Then we have $B_{\Qbb_p,m}^{I}=\bigoplus_{r\in S_m} B_{\Qbb_p}^{I}[\varepsilon^r]$ and $\tilde{B}_{\Qbb_{p}(\zeta_{p^{\infty}})}^{I}=\hat{\bigoplus_{r\in S_{\infty}}} B_{\Qbb_p}^{I}[\varepsilon^r]$.
    Therefore, we can define $R_m \colon \tilde{B}_{\Qbb_{p}(\zeta_{p^{\infty}})}^{I} \to B_{\Qbb_p,m}^{I}$ as the projection corresponding to the inclusion $S_m\subset S_{\infty}$.
    By computing the action of $\gamma\in \Gamma_{\Qbb_{p}}$ on $[\varepsilon]$ explicitly, we can check that these morphisms satisfy the Tate-Sen axioms.
\end{example}

By using the Tate-Sen methods, we can prove Theorem \ref{thm:locally analytic descent} when $A$ is a Banach $\Qbb_p$-algebra.

\begin{proof}[Proof of Theorem \ref{thm:locally analytic descent} when $A$ is a Banach $\Qbb_p$-algebra.]
    First, we prove (1).
    Let $M$ be a finite projective $\tilde{B}_{K_{\infty},A}^{I}$-module with a continuous semilinear $\Gamma_K$-action.
    Then there exists a finite extension $L/K$ in $K_{\infty}$ and a finite free $\tilde{B}_{K_{\infty},A}^{I}$-module with a continuous semilinear $\Gamma_L$-action which contains $M$ as a direct summand as $\Gamma_L$-representations by \cite[Lemma 5.4]{Por22}\footnote{In \cite[Lemma 5.4]{Por22}, it is claimed that one can take $L=K$, but there is a gap in the argument. However, the statement we are using here can be deduced from the discussion loc.\ cit.; see also \cite{Poraterrata}.}.
    Since we have $(-)^{\Gamma_K\mathchar`-R\la}=(-)^{\Gamma_L\mathchar`-R\la}$ and $(-)^{\Gamma_{K,m}\mathchar`-\an}=(-)^{\Gamma_{L,m}\mathchar`-\an}$ for $m$ sufficiently large, we may assume that $L=K$ and $M$ is a finite free $\tilde{B}_{K_{\infty},A}^{I}$-module.
    Then by \cite[Proposition 5.3]{Por24}, $M$ has no derived locally analytic vectors.
    We take a finite free $\tilde{B}_{K_{\infty},A}^{I,+}$-submodule $M^+$ of $M$ which generates $M$ as a $\tilde{B}_{K_{\infty},A}^{I}$-module.
    Then there is a compact open subgroup $H\subset \Gamma_K$ such that $M^+$ is $H$-stable.
    Moreover, we may assume that $H$ and $M^+$ satisfy the condition of \cite[Proposition 3.3.1]{BC08}.
    Then there exist $n\geq 0$ and a finite free $B_{K,n,A}^{I,+}$-submodule $D^+_n(M^+)$ of $M^+$ which is $H$-stable and satisfies the condition that the natural morphism 
    $$D^+_n(M^+) \otimes_{(B_{K,n,A}^{I,+})_{\square}} \tilde{B}_{K_{\infty},A}^{I,+} \to M^+$$
    is an isomorphism.
    We write $D_n(M)=D^+_n(M^+)[1/p]$, which is a $H$-stable finite free $B_{K,n,A}^{I}$-submodule of $M$.
    By \cite[Proposition 3.3.3]{RJRC23} and (TS4), we deduce that $D_n(M)$ and $B_{K,n,A}^{I}$ are $\Gamma_{K,m}$-analytic representations for $m$ sufficiently large.
    Therefore, by Proposition \ref{prop:an functor and tensor product commute}, we obtain isomorphisms 
    $$M^{m\mathchar`-\an}\cong(D_n(M) \otimes_{(B_{K,n,A}^{I})_{\square}} \tilde{B}_{K_{\infty},A}^{I})^{m\mathchar`-\an}\cong D_n(M) \otimes_{(B_{K,n,A}^{I})_{\square}}\tilde{B}_{K_{\infty},A}^{I,m\mathchar`-\an},$$
    which proves (1).

    Next, we prove (2). 
    Let $N$ be a finite projective $B_{K,\infty,A}^{I}$-module with a continuous semilinear $\Gamma_K$-action.
    By \cite[Lemma 2.7.4]{Mann22}, there exist a finite projective $B_{K,n,A}^{I}$-module $N_n$ with a continuous semilinear $\Gamma_K$-action and an isomorphism 
    $$N_n\otimes_{(B_{K,n,A}^{I})_{\square}}B_{K,\infty,A}^{I}\cong N$$ 
    for some $n\geq 0$, where we note $B_{K,\infty,A}^{I}=\varinjlim_{n} B_{K,n,A}^{I}$.
    It is enough to show that the $\Gamma_K$-action on $N_n$ is $\Gamma_{K,m}$-analytic for $m$ sufficiently large.
    By the same reason as above, we may assume that $N_n$ is a finite free $B_{K,n,A}^{I}$-module.
    We take a finite free $B_{K,n,A}^{I,+}$-submodule $N_n^+$ of $N_n$ which generates $N_n$ as a $B_{K,n,A}^{I}$-module.
    Then there is a compact open subgroup $H\subset \Gamma_K$ such that $N_n^+$ is $H$-stable and the $H$-action on $N_n^+/p$ is trivial.
    Therefore, the $\Gamma_K$-action on $N$ is $\Gamma_{K,m}$-analytic for $m$ sufficiently large by \cite[Proposition 3.4.3]{RJRC23}.
    By Proposition \ref{prop:la functor and tensor product commute} and Proposition \ref{prop:la vector of perfect witt ring}, we find that
    $$N^{I} \to (N^{I}\otimes_{(B_{K,\infty,A}^{I})_{\square}} \tilde{B}_{K_{\infty},A}^{I})^{\la}$$
    is an isomorphism.
\end{proof}

Next, we prove Theorem \ref{thm:locally analytic descent} for a general $A$.
We note the following lemma.
\begin{lemma}\label{lem:flatness}
    Let $A$ be a Banach $\Qbb_p$-algebra or an algebraic-affinoid $\Qbb_{p,\square}$-algebra.
\begin{enumerate}
\item
Let $m\geq 0$ be a non-negative integer.
Then there is an isomorphism $$\tilde{B}_{K_{\infty}}^{I}\cong\hat{\bigoplus_{J_m}} B_{K,m}^{I}$$ of $B_{K,m}^{I}$-modules for some set $J_m$.
In particular, $\tilde{B}_{K_{\infty},A}^{I}$ is a faithfully flat $(B_{K,m,A}^{I})_{\square}$-algebra.
\item The $(B_{K,\infty,A}^{I})_{\square}$-algebra $\tilde{B}_{K_{\infty},A}^{I}$ is faithfully flat.
\item For $m\in \Zbb_{\geq 0}\cup \{\infty\}$, the morphism $(B_{K,m,A}^{I})_{\square} \to (\tilde{B}_{K_{\infty},A}^{I})_{\square}$ of analytic rings satisfies the $\ast$-descent.
\end{enumerate}
\end{lemma}
\begin{proof}
    The existence of an isomorphism $\tilde{B}_{K_{\infty}}^{I}\cong\hat{\bigoplus_{J_m}} B_{K,m}^{I}$ follows from \cite[Proposition 8.5(i)(b)]{Col08}.
    The flatness of $\tilde{B}_{K_{\infty},A}^{I}=\tilde{B}_{K_{\infty}}^{I}\otimes_{\Qbb_{p,\square}}^{\Lbb}A_{\square}$ over $(B_{K,m}^{I})_{\square}\otimes_{\Qbb_{p,\square}}^{\Lbb}A_{\square}$ follows from Example \ref{ex:flat module}.
    Part (2) follows from Lemma \ref{lem:faithfully flat colimit}, and Part (3) follows from Proposition \ref{prop:ff descent}.
\end{proof}

Let $A$ be an algebraic-affinoid $\Qbb_{p,\square}$-algebra, and $B$ be an affinoid $\Qbb_p$-algebra of definition of $A$.
First, we assume that the underlying discrete ring $A(\ast)$ of $A$ is an integral domain.

\begin{lemma}\label{lem:injection to Banach}
    There exists a Banach $B$-algebra $A^{\prime}$ and an injection $A \to A^{\prime}$ of $B$-algebras.
\end{lemma}
\begin{proof}
    Let $\mfrak$ be a maximal ideal of $A(\ast)$.
    Since $B(\ast)$ is a Jacobson ring, the preimage $\nfrak$ of $\mfrak$ in $B(\ast)$ is also a maximal ideal.
    We write $F=A(\ast)/\mfrak$, which is a finite extension of $\Qbb_p$.
    We take a surjection $f\colon B(\ast)[X_1,\ldots,X_n]\to A(\ast)$.
    We can arrange that the images of $f(X_i)$ in $F$ lie in $\Ocal_F$.
    Let $B_0(\ast)$ be a ring of definition of $B(\ast)$, and $A_0(\ast)$ be the image of the morphism $$B_0(\ast)[X_1,\ldots,X_n] \subset B(\ast)[X_1,\ldots,X_n] \overset{f}{\to} A(\ast).$$
    Then the image of $B_0(\ast)[X_1,\ldots,X_n] \subset B(\ast)[X_1,\ldots,X_n] \overset{f}{\to} A(\ast) \to F$ is contained in $\Ocal_F$, so $p$ is not invertible in $A_0(\ast)$.
    Since $A_0(\ast)$ is a noetherian integral domain, the $p$-adic completion morphism  $A_0(\ast)\to A_0(\ast)^{\wedge}$ is injective.
    We write $A'=A_0(\ast)^{\wedge}[1/p]$, which is an affinoid $B$-algebra.
    We regard it as a condensed ring.
    Then, the injection $A(\ast)\to A'$ induces an injection $A=\Cond_{B}(A(\ast))\to A'$ by Lemma \ref{lem:condensification functor} (2).
\end{proof}

\begin{proposition}\label{prop:la descent inj}
    Let $M$ be a finite projective $\tilde{B}_{K_{\infty},A}^{I}$-module with a continuous semilinear $\Gamma_K$-action.
    Then the natural morphism
    $$M^{\la}\otimes_{(B_{K,\infty,A}^{I})_{\square}} \tilde{B}_{K_{\infty},A}^{I} \to M$$
    is injective.
\end{proposition}
\begin{proof}
    We take an injection $A\to A^{\prime}$ as in Lemma \ref{lem:injection to Banach}.
    We write $$M^{\prime}=M\otimes_{\tilde{B}_{K_{\infty},A}^{I}}\tilde{B}_{K_{\infty},A^{\prime}}^{I},$$ which is a finite projective $\tilde{B}_{K_{\infty},A^{\prime}}^{I}$-module with a continuous semilinear $\Gamma_K$-action.
    Then we have the following commutative diagram:
    $$
    \xymatrix{
        M^{\la}\otimes_{(B_{K,\infty,A}^{I})_{\square}} \tilde{B}_{K_{\infty},A}^{I}\ar[r]^-{\alpha}\ar[d]^-{\beta} & M\ar[d] \\
        M^{\prime,\la}\otimes_{(B_{K,\infty,A^{\prime}}^{I})_{\square}} \tilde{B}_{K_{\infty},A^{\prime}}^{I}\ar[r]^-{\gamma} & M^{\prime},
    }
    $$
    where $\gamma$ is an isomorphism.
    We prove that $\beta$ is injective, which implies $\alpha$ is also injective.
    Since $\tilde{B}_{K_{\infty},A}^{I}$ (resp.\ $\tilde{B}_{K_{\infty},A^{\prime}}^{I}$) is flat over $(B_{K,\infty,A}^{I})_{\square}$ (resp.\ $(B_{K,\infty,A^{\prime}}^{I})_{\square}$) by Lemma \ref{lem:flatness}, we have the following diagram:
    $$
    \xymatrix{
        M^{\la}\otimes_{(B_{K,\infty,A}^{I})_{\square}} \tilde{B}_{K_{\infty},A}^{I}\ar[r]\ar[d]^-{\beta} & M\otimes_{(B_{K,\infty,A}^{I})_{\square}} \tilde{B}_{K_{\infty},A}^{I}\ar[d]^-{\delta} \\
        M^{\prime,\la}\otimes_{(B_{K,\infty,A^{\prime}}^{I})_{\square}} \tilde{B}_{K_{\infty},A^{\prime}}^{I}\ar[r] & M^{\prime}\otimes_{(B_{K,\infty,A^{\prime}}^{I})_{\square}} \tilde{B}_{K_{\infty},A^{\prime}}^{I},
    }
    $$
    where the horizontal morphisms are injective.
    The morphism $\delta$ is equal to $$M\otimes_{(B_{K,\infty,A}^{I})_{\square}} \tilde{B}_{K_{\infty},A}^{I} \to M\otimes_{(B_{K,\infty,A}^{I})_{\square}} \tilde{B}_{K_{\infty},A^{\prime}}^{I}.$$
    Since $M$ is flat over $(\tilde{B}_{K_{\infty},A}^{I})_{\square}$ and $\tilde{B}_{K_{\infty},A}^{I}$ is flat over $(B_{K,\infty,A}^{I})_{\square}$, $M$ is flat over $(B_{K,\infty,A}^{I})_{\square}$.
    Moreover, the morphism
    $$\tilde{B}_{K_{\infty},A}^{I}=A\otimes_{\Qbb_{p,\square}}\tilde{B}_{K_{\infty}}^{I} \to \tilde{B}_{K_{\infty},A^{\prime}}^{I}=A^{\prime}\otimes_{\Qbb_{p,\square}}\tilde{B}_{K_{\infty}}^{I}$$
    is injective since $\tilde{B}_{K_{\infty}}^{I}$ is flat over $\Qbb_{p,\square}$ (Example \ref{ex:flat module}).
    Therefore, $\delta$ is injective, which shows that $\beta$ is also injective.
\end{proof}

Next, we prove the morphism in Proposition \ref{prop:la descent inj} is surjective.
\begin{proposition}\label{prop:la descent surj}
    Let $M$ be a finite projective $\tilde{B}_{K_{\infty},A}^{I}$-module with a continuous semilinear $\Gamma_K$-action.
    Then the natural morphism
    $$M^{\la}\otimes_{(B_{K,\infty,A}^{I})_{\square}} \tilde{B}_{K_{\infty},A}^{I} \to M$$
    is surjective.
\end{proposition}
\begin{proof}
By Proposition \ref{prop:relatively discrete action filtered colimit}, there exists a finitely presented $\tilde{B}_{K_{\infty},B}^{I}$-submodule $N\subset M$ which is stable under the $\Gamma_K$-action and contains a generator of $M$ as a $\tilde{B}_{K_{\infty},A}^{I}$-module.
By Corollary \ref{cor:free module lift of action}, there exists a finite free $\tilde{B}_{K_{\infty},B}^{I}$-module $L$ with a continuous semilinear $\Gamma_K$-action and a $\Gamma_K$-equivariant $\tilde{B}_{K_{\infty},B}^{I}$-linear surjection $L\to M$.
From these, we get a $\Gamma_K$-equivariant surjection
$$\beta\colon L\otimes_{(\tilde{B}_{K_{\infty},B}^{I})_{\square}}\tilde{B}_{K_{\infty},A}^{I}\to M.$$
Therefore, we obtain the following commutative diagram:
$$
\xymatrix{
    (L^{\la}\otimes_{(B_{K,\infty,B}^{I})_{\square}} \tilde{B}_{K_{\infty},B}^{I})\otimes_{(\tilde{B}_{K_{\infty},B}^{I})_{\square}}\tilde{B}_{K_{\infty},A}^{I}\ar[r]^-{\alpha}\ar[d] & L\otimes_{(\tilde{B}_{K_{\infty},B}^{I})_{\square}}\tilde{B}_{K_{\infty},A}^{I}\ar[d]^-{\beta} \\
    M^{\la}\otimes_{(B_{K,\infty,A}^{I})_{\square}} \tilde{B}_{K_{\infty},A}^{I}\ar[r]^-{\gamma} & M.
}
$$
Since both $\alpha$ and $\beta$ are surjective, $\gamma$ is also surjective.
\end{proof}

We prove Theorem \ref{thm:locally analytic descent} under the assumption that $A$ is an integral domain.
In the proof below, this assumption is used only to apply Proposition \ref{prop:la descent inj} and Proposition \ref{prop:la descent surj}. Conversely, the same argument applies to any $A$ for which both propositions hold.

\begin{proof}[Proof of Theorem \ref{thm:locally analytic descent} when $A$ is an integral domain]
    In order to simplify notation, we denote $M^{I}$ simply as $M$.
    By Proposition \ref{prop:la descent inj} and Proposition \ref{prop:la descent surj}, the natural morphism
    $$M^{\la}\otimes_{(B_{K,\infty,A}^{I})_{\square}} \tilde{B}_{K_{\infty},A}^{I} \to M$$
    is an isomorphism.
    By Proposition \ref{prop:la functor and tensor product commute} and Proposition \ref{prop:la vector of perfect witt ring}, we find that $M$ has no derived locally analytic vectors.
    By Lemma \ref{lem:flatness} and Theorem \ref{thm:descent of finite projective modules}, $M^{\la}$ is a finite projective $B_{K,\infty,A}^{I}$-module.
    By \cite[Lemma 2.7.4]{Mann22}, there exist a finite projective $B_{K,m,A}^{I}$-module $M_m$ with a continuous semilinear $\Gamma_K$-action and an isomorphism $$M_m\otimes_{(B_{K,m,A}^{I})_{\square}}B_{K,\infty,A}^{I}\cong M^{\la}$$ for some $m\geq 0$, where we note $B_{K,\infty,A}^{I}=\varinjlim_{m} B_{K,m,A}^{I}$.
    We prove that $M_m$ is $\Gamma_{K,n}$-analytic for $n$ sufficiently large, where we note that it leads to (2) by the same reason as in the proof when $A$ is a Banach $\Qbb_p$-algebra.
    We take $l\geq m$ such that $B_{K,m,B}^{I}$ and $B_{K,m,A}^{I}$ are $\Gamma_{K,l}$-analytic.
    By the same argument as in the proof of Proposition \ref{prop:la descent surj}, we can take a finite free $B_{K,m,B}^{I}$-module $L_m$ with a continuous semilinear $\Gamma_K$-action and a $\Gamma_K$-equivariant surjection 
    $$L_m\otimes_{(B_{K,m,B}^{I})_{\square}}B_{K,m,A}^{I}\to M_m.$$
    For $n\geq l$ sufficiently large, $L_m$ is $\Gamma_{K,n}$-analytic.
    Therefore, $L_m\otimes_{(B_{K,m,B}^{I})_{\square}}B_{K,m,A}^{I}$ is also $\Gamma_{K,n}$-analytic by Proposition \ref{prop:an functor and tensor product commute}, and we get a morphism 
    $$L_m\otimes_{(B_{K,m,B}^{I})_{\square}}B_{K,m,A}^{I}\to M_m^{n\mathchar`-\an} \subset M_m,$$
    where the inclusion $M_m^{n\mathchar`-\an} \subset M_m$ follows from Lemma \ref{lem:an is sub}.
    Since $$L_m\otimes_{(B_{K,m,B}^{I})_{\square}}B_{K,m,A}^{I}\to M_m$$ 
    is surjective, we get $M_m^{n\mathchar`-\an}=M_m$.
    From the above and Proposition \ref{prop:an functor and tensor product commute}, we obtain 
    $$M^{n\mathchar`-\an}\cong (M_m\otimes_{(B_{K,m,A}^{I})_{\square}}\tilde{B}_{K_{\infty},A}^{I})^{n\mathchar`-\an}\cong M_m\otimes_{(B_{K,m,A}^{I})_{\square}}\tilde{B}_{K_{\infty},A}^{I,n\mathchar`-\an},$$
    which completes the proof of (1).
\end{proof}

Next, we consider general general algebraic-affinoid $\Qbb_{p,\square}$-algebras.
\begin{lemma}\label{lem:reduction to red}
    Both Proposition \ref{prop:la descent inj} and Proposition \ref{prop:la descent surj} are true for a general algebraic-affinoid $\Qbb_{p,\square}$-algebra $A$.
\end{lemma}
\begin{proof}
    There exists a filtration of ideals of $A$ 
    $$0=I_0\subset I_1 \subset \cdots \subset I_n=A$$
    such that each quotient $I_i/I_{i-1}$ is isomorphic to $A/\pfrak_i$ for some prime ideal $\pfrak_i$ of $A$ by \cite[Theorem 10]{Mat80} or \cite[00L0]{stacks-project}.
    We proceed by induction on $n$.
    The case $n=1$ is already proved.
    Now suppose $n>1$.
    We write $\bar{A}=A/I_1$ and $A/\pfrak_1=A_1$.
    From the isomorphism $I_1\cong A_1$ of $A_{\square}$-modules, we get an exact sequence
    $$0\to \tilde{B}_{K_{\infty},A_1}^{I}\to \tilde{B}_{K_{\infty},A}^{I}\to \tilde{B}_{K_{\infty},\bar{A}}^{I}\to 0.$$
    Let $M$ be a finite projective $\tilde{B}_{K_{\infty},A}^{I}$-module with a continuous semilinear $\Gamma_K$-action.
    We write $\bar{M}=M\otimes_{(\tilde{B}_{K_{\infty},A}^{I})_{\square}}\tilde{B}_{K_{\infty},\bar{A}}^{I}$ and $M_1=M\otimes_{(\tilde{B}_{K_{\infty},A}^{I})_{\square}}\tilde{B}_{K_{\infty},A_1}^{I}$.
    Then we have an exact sequence
    $$0\to M_1 \to M \to \bar{M} \to 0.$$
    By Theorem \ref{thm:locally analytic descent} for integral domains, $M_1$ has no derived locally analytic vectors.
    Therefore, there exists an exact sequence
    $$0\to M_1^{\la} \to M^{\la} \to \bar{M}^{\la} \to 0.$$
    By applying $-\otimes_{(B_{K,\infty}^{I})_{\square}} \tilde{B}_{K_{\infty}}^{I}$, which is exact by Lemma \ref{lem:flatness} (2), to the above exact sequence, we obtain an exact sequence
    $$0\to M_1^{\la}\otimes_{(B_{K,\infty,A_1}^{I})_{\square}} \tilde{B}_{K_{\infty},A_1}^{I} \to M^{\la}\otimes_{(B_{K,\infty,A}^{I})_{\square}} \tilde{B}_{K_{\infty},A}^{I} \to \bar{M}^{\la}\otimes_{(B_{K,\infty,\bar{A}}^{I})_{\square}} \tilde{B}_{K_{\infty},\bar{A}}^{I} \to 0.$$
    Since the natural morphisms
    \begin{align*}
        M_1^{\la}\otimes_{(B_{K,\infty,A_1}^{I})_{\square}} \tilde{B}_{K_{\infty},A_1}^{I} &\to M_1,\\
        \bar{M}^{\la}\otimes_{(B_{K,\infty,\bar{A}}^{I})_{\square}} \tilde{B}_{K_{\infty},\bar{A}}^{I} &\to \bar{M}
    \end{align*}
    are isomorphisms by the induction hypothesis, we find that the natural morphism
    $$M^{\la}\otimes_{(B_{K,\infty,A}^{I})_{\square}} \tilde{B}_{K_{\infty},A}^{I} \to M$$
    is also an isomorphism.
\end{proof}

\begin{proof}[Proof of Theorem \ref{thm:locally analytic descent}]
The same argument as in the proof of Theorem \ref{thm:locally analytic descent} for integral domains, together with Lemma \ref{lem:reduction to red}, shows the theorem for a general algebraic-affinoid $\Qbb_{p,\square}$-algebra $A$.
\end{proof}

\begin{corollary}
    The functors
    $$\VB_{\tilde{\Bcal}_{K_{\infty},\Acal}}^{\varphi,\Gamma_K} \to \VB_{\Bcal_{K,\infty,\Acal}}^{\varphi,\Gamma_K} ;\; \{M^{I}\}_{I}\mapsto \{M^{I,\la}\}_{I}$$
    and 
    \begin{align*}
          \VB_{\Bcal_{K,\infty,\Acal}}^{\varphi,\Gamma_K}\to \VB_{\tilde{\Bcal}_{K_{\infty},\Acal}}^{\varphi,\Gamma_K} ;\; \{N^{I}\}_{I}&\mapsto \{N^{I}\otimes_{\Bcal_{K,\infty,\Acal}^{I}}\tilde{\Bcal}_{K_{\infty},\Acal}^{I}\}_{I}\\
          &=\{N^{I}\otimes_{(B_{K,\infty,A}^{I})_{\square}}\tilde{B}_{K_{\infty},A}^{I}\}_{I}
    \end{align*}
    are quasi-inverse to each other.
\end{corollary}
\begin{proof}
    It immediately follows from Theorem \ref{thm:locally analytic descent}.
\end{proof}

\begin{definition}
     We write the above categorical equivalences as
     $$\VB_{\Bcal_{K,\infty,\Acal}}^{\varphi,\Gamma_K}\ni N\mapsto N\otimes_{\Bcal_{K,\infty,A}}\tilde{\Bcal}_{K_{\infty},\Acal}\in\VB_{\tilde{\Bcal}_{K_{\infty},\Acal}}^{\varphi,\Gamma_K}$$
    and 
    $$\VB_{\tilde{\Bcal}_{K_{\infty},\Acal}}^{\varphi,\Gamma_K}\ni M=\{M^{I}\}\mapsto M^{\la}=\{M^{I,\la}\}\in \VB_{\Bcal_{K,\infty,\Acal}}^{\varphi,\Gamma_K}.$$
\end{definition}

It remains to show that the categories $\VB_{\tilde{\Bcal}_{K_{\infty},\Acal}}^{\varphi,\Gamma_K}$ and $\VB_{\Bcal_{K,\Acal}}^{\varphi,\Gamma_K}$ are canonically equivalent.

\begin{proof}[Proof of Theorem \ref{thm:deperfection}]
    There is a natural functor 
    \begin{align*}
    \alpha \colon \VB_{\Bcal_{K,\Acal}}^{\varphi,\Gamma_K} \to \VB_{\tilde{\Bcal}_{K_{\infty,\Acal}}}^{\varphi,\Gamma_K};\; \{N^{I}\}_{I} &\mapsto \{N^{I}\otimes_{\Bcal_{K,\Acal}^{I}} \tilde{\Bcal}_{K_{\infty},\Acal}^{I}\}_{I}\\
    &=\{N^{I}\otimes_{(B_{K,A}^{I})_{\square}} \tilde{B}_{K_{\infty},A}^{I}\}_{I}.
    \end{align*}
    We construct a quasi-inverse functor of $\alpha$.
    We take an object $M=\{M^{I}\}_{I}$ of $\VB_{\tilde{\Bcal}_{K_{\infty,\Acal}}}^{\varphi,\Gamma_K}$.
    We fix a sufficiently large integer $l\in \Zbb$.
    By Theorem \ref{thm:locally analytic descent}, there exists $m\in \Zbb$ such that the natural morphisms
    \begin{align*}
        &M^{[p^{-l-1},p^{-l}],m\mathchar`-\an}\otimes_{(\tilde{B}_{K_{\infty},A}^{[p^{-l-1},p^{-l}],m\mathchar`-\an})_{\square}} \tilde{B}_{K_{\infty},A}^{[p^{-l-1},p^{-l}]} \to M^{[p^{-l-1},p^{-l}]},\\
        &M^{[p^{-l},p^{-l}],m\mathchar`-\an}\otimes_{(\tilde{B}_{K_{\infty},A}^{[p^{-l},p^{-l}],m\mathchar`-\an})_{\square}} \tilde{B}_{K_{\infty},A}^{[p^{-l},p^{-l}]} \to M^{[p^{-l},p^{-l}]}
    \end{align*}
    are isomorphisms.
    Since we have a $\Gamma_K$-equivariant isomorphism 
    $$M^{[p^{-l-1},p^{-l}]}\otimes_{(\tilde{B}_{K_{\infty},A}^{[p^{-l-1},p^{-l}]})_{\square},\varphi} \tilde{B}_{K_{\infty},A}^{[p^{-l-2},p^{-l-1}]}\cong M^{[p^{-l-2},p^{-l-1}]}$$
    and $\varphi\colon \tilde{B}_{K_{\infty},A}^{[p^{-l-1},p^{-l}]}\to \tilde{B}_{K_{\infty},A}^{[p^{-l-2},p^{-l-1}]}$ is a $\Gamma_K$-equivariant isomorphism,
    we obtain an $\Gamma_K$-equivariant isomorphism
    $$M^{[p^{-l-1},p^{-l}],m\mathchar`-\an}\otimes_{(\tilde{B}_{K_{\infty},A}^{[p^{-l-1},p^{-l}],m\mathchar`-\an})_{\square,\varphi}} \tilde{B}_{K_{\infty},A}^{[p^{-l-2},p^{-l-1}],m\mathchar`-\an} \cong M^{[p^{-l-2},p^{-l-1}],m\mathchar`-\an}.$$
    Therefore, we find that the natural morphism
    \begin{align}
        M^{[p^{-l-2},p^{-l-1}],m\mathchar`-\an}\otimes_{(\tilde{B}_{K_{\infty},A}^{[p^{-l-2},p^{-l-1}],m\mathchar`-\an})_{\square}} \tilde{B}_{K_{\infty},A}^{[p^{-l-2},p^{-l-1}]} \to M^{[p^{-l-2},p^{-l-1}]} \label{eq1}
    \end{align} 
    is an isomorphism.
    By repeating this argument, for $k\geq l$, the natural morphism
    \begin{align}
    M^{[p^{-k-1},p^{-k}],m\mathchar`-\an}\otimes_{(\tilde{B}_{K_{\infty},A}^{[p^{-k-1},p^{-k}],m\mathchar`-\an})_{\square}} \tilde{B}_{K_{\infty},A}^{[p^{-k-1},p^{-k}]} \to M^{[p^{-k-1},p^{-k}]}\label{eq2}
    \end{align}
    is also an isomorphism.
    By the same argument, we find that for $k\geq l$, the natural morphism
    $$M^{[p^{-k},p^{-k}],m\mathchar`-\an}\otimes_{(\tilde{B}_{K_{\infty},A}^{[p^{-k},p^{-k}],m\mathchar`-\an})_{\square}} \tilde{B}_{K_{\infty},A}^{[p^{-k},p^{-k}]} \to M^{[p^{-k},p^{-k}]}$$
    is an isomorphism.
    Since, for $k\geq l,m$, we have 
    \begin{align*}
        &\tilde{B}_{K_{\infty},A}^{[p^{-k-1},p^{-k}],m\mathchar`-\an}\subset\tilde{B}_{K_{\infty},A}^{[p^{-k-1},p^{-k}],k\mathchar`-\an}\subset B_{K,A}^{[p^{-k-1},p^{-k}]}, \\
        &\tilde{B}_{K_{\infty},A}^{[p^{-k},p^{-k}],m\mathchar`-\an}\subset\tilde{B}_{K_{\infty},A}^{[p^{-k},p^{-k}],k\mathchar`-\an}\subset B_{K,A}^{[p^{-k},p^{-k}]}
    \end{align*}
    by Proposition \ref{prop:la vector of perfect witt ring}, 
    we can define 
    \begin{align*}
        &N^{[p^{-k-1},p^{-k}]}=(M^{[p^{-k-1},p^{-k}]})^{m\mathchar`-\an}\otimes_{(\tilde{B}_{K_{\infty},A}^{[p^{-k-1},p^{-k}],m\mathchar`-\an})_{\square}}B_{K,A}^{[p^{-k-1},p^{-k}]},\\
        &N^{[p^{-k},p^{-k}]}=(M^{[p^{-k},p^{-k}]})^{m\mathchar`-\an}\otimes_{(\tilde{B}_{K_{\infty},A}^{[p^{-k},p^{-k}],m\mathchar`-\an})_{\square}}B_{K,A}^{[p^{-k},p^{-k}]},
    \end{align*}
    which is a finite projective $B_{K,A}^{[p^{-k-1},p^{-k}]}$-module (resp.\ $B_{K,A}^{[p^{-k},p^{-k}]}$-module) with a continuous semilinear $\Gamma_K$-action by the following isomorphisms \eqref{eq6}, Proposition \ref{prop:descent finite projective Fredholm}, and Theorem \ref{thm:descent of finite projective modules}.
    By (\ref{eq1}) and (\ref{eq2}), we have isomorphisms
    \begin{equation}\label{eq6}
    \begin{aligned}
        &N^{[p^{-k-1},p^{-k}]}\otimes_{(B_{K,A}^{[p^{-k-1},p^{-k}]})_{\square}} \tilde{B}_{K_{\infty},A}^{[p^{-k-1},p^{-k}]} \cong M^{[p^{-k-1},p^{-k}]},\\
        &N^{[p^{-k},p^{-k}]}\otimes_{(B_{K,A}^{[p^{-k},p^{-k}]})_{\square}} \tilde{B}_{K_{\infty},A}^{[p^{-k},p^{-k}]} \cong M^{[p^{-k},p^{-k}]}
    \end{aligned}
    \end{equation}
    for $k \geq l,m$.
    We note that for $i=k,k+1$, the natural morphism
    $$B_{K,A}^{[p^{-i},p^{-i}]}\otimes_{(B_{K,A}^{[p^{-k-1},p^{-k}]})_{\square}} \tilde{B}_{K_{\infty},A}^{[p^{-k-1},p^{-k}]} \to \tilde{B}_{K_{\infty},A}^{[p^{-i},p^{-i}]}$$
    is an isomorphism by Remark \ref{rem:open subspace perfect}, so we find that the natural morphism
    $$N^{[p^{-k-1},p^{-k}]} \otimes_{(B_{K,A}^{[p^{-k-1},p^{-k}]})_{\square}} B_{K,A}^{[p^{-i},p^{-i}]} \to N^{[p^{-i},p^{-i}]}$$
    becomes an isomorphism after taking the tensor product $-\otimes_{(B_{K,A}^{[p^{-i},p^{-i}]})_{\square}}\tilde{B}_{K_{\infty},A}^{[p^{-i},p^{-i}]}$.
By Proposition \ref{prop:ff descent} and Lemma \ref{lem:flatness}, we find that $$N^{[p^{-k-1},p^{-k}]} \otimes_{\Bcal_{K,\Acal}^{[p^{-k-1},p^{-k}]}}^{\Lbb} \Bcal_{K,\Acal}^{[p^{-i},p^{-i}]}\cong N^{[p^{-k-1},p^{-k}]} \otimes_{(B_{K,A}^{[p^{-k-1},p^{-k}]})_{\square}} B_{K,A}^{[p^{-i},p^{-i}]} \to N^{[p^{-i},p^{-i}]}$$
is also an isomorphism for $i=k,k+1$.
By the same argument, we obtain isomorphisms
\begin{align*}
\Phi_{[p^{-k-1},p^{-k}]}\colon N^{[p^{-k-1},p^{-k}]}\otimes_{\Bcal_{K,\Acal}^{[p^{-k-1},p^{-k}]},\varphi}^{\Lbb} \Bcal_{K,\Acal}^{[p^{-k-2},p^{-k-1}]}\cong N^{[p^{-k-2},p^{-k-1}]},\\
\Phi_{[p^{-k},p^{-k}]}\colon N^{[p^{-k},p^{-k}]}\otimes_{\Bcal_{K,\Acal}^{[p^{-k},p^{-k}]},\varphi}^{\Lbb} \Bcal_{K,\Acal}^{[p^{-k-1},p^{-k-1}]}\cong N^{[p^{-k-1},p^{-k-1}]}.
\end{align*}
By gluing them, we obtain a family $\{N^{[p^{-k^{\prime}},p^{-k}]}\}_{k,k^{\prime}}$ of objects of $\Dcal(\Bcal_{K,\Acal}^{[p^{-k^{\prime}},p^{-k}]})$ for $k^{\prime}\geq k\geq l,m$ with a continuous semilinear $\Gamma_K$-action on $N^{[p^{-k^{\prime}},p^{-k}]}$ and $\Gamma_K$-equivariant isomorphisms 
\begin{align*}
\Phi_{p^{-k^{\prime}},p^{-k}}\colon N^{[p^{-k^{\prime}},p^{-k}]}\otimes_{\Bcal_{K,\Acal}^{[p^{-k^{\prime}},p^{-k}]},\varphi}^{\Lbb} \Bcal_{K,\Acal}^{[p^{-k^{\prime}-1},p^{-k-1}]}\cong N^{[p^{-k^{\prime}-1},p^{-k-1}]}
\end{align*}
which satisfy the same condition as in Definition \ref{def:perfect module}. 
By Proposition \ref{prop:descent finite projective Fredholm} (1), $N^{[p^{-k^{\prime}},p^{-k}]}$ is dualizable over $\Bcal_{K,\Acal}^{[p^{-k^{\prime}},p^{-k}]}$.
In particular, we find that $N^{[p^{-k^{\prime}},p^{-k}]}$ is nuclear in $\Dcal((B_{K,A}^{[p^{-k^{\prime}},p^{-k}]})_{\square})$ by Lemma \ref{lem:dualizable compact nuclear} and Lemma \ref{lem:A^+ nuclear}.
By construction, we have an isomorphism
$$N^{[p^{-k^{\prime}},p^{-k}]}\otimes_{\Bcal_{K,\Acal}^{[p^{-k^{\prime}},p^{-k}]}}^{\Lbb} \tilde{B}_{K_{\infty},A}^{[p^{-k^{\prime}},p^{-k}]} \to M^{[p^{-k^{\prime}},p^{-k}]}.$$
Since $N^{[p^{-k^{\prime}},p^{-k}]}\otimes_{(B_{K,A}^{[p^{-k^{\prime}},p^{-k}]})_{\square}}^{\Lbb} \tilde{B}_{K_{\infty},A}^{[p^{-k^{\prime}},p^{-k}]}$ is a nuclear object in $\Dcal((B_{K,A}^{[p^{-k^{\prime}},p^{-k}]})_{\square})$, it is $\Bcal_{K,\Acal}^{[p^{-k^{\prime}},p^{-k}]}$-complete by Lemma \ref{lem:A^+ nuclear}.
Therefore, the morphism 
$$N^{[p^{-k^{\prime}},p^{-k}]}\otimes_{(B_{K,A}^{[p^{-k^{\prime}},p^{-k}]})_{\square}}^{\Lbb} \tilde{B}_{K_{\infty},A}^{[p^{-k^{\prime}},p^{-k}]}\to N^{[p^{-k^{\prime}},p^{-k}]}\otimes_{\Bcal_{K,\Acal}^{[p^{-k^{\prime}},p^{-k}]}}^{\Lbb} \tilde{B}_{K_{\infty},A}^{[p^{-k^{\prime}},p^{-k}]}$$
is an isomorphism, and 
$$N^{[p^{-k^{\prime}},p^{-k}]}\otimes_{(B_{K,A}^{[p^{-k^{\prime}},p^{-k}]})_{\square}}^{\Lbb} \tilde{B}_{K_{\infty},A}^{[p^{-k^{\prime}},p^{-k}]} \to M^{[p^{-k^{\prime}},p^{-k}]}$$ 
is also an isomorphism.
By Proposition \ref{prop:descent finite projective Fredholm} and Theorem \ref{thm:descent of finite projective modules}, we find that $N^{[p^{-k^{\prime}},p^{-k}]}$ is a finite projective $\Bcal_{K,\Acal}^{[p^{-k^{\prime}},p^{-k}]}$-module, which shows that $N=\{N^{[p^{-k^{\prime}},p^{-k}]}\}$ is a $(\varphi,\Gamma_K)$-module over $B_{K,A}$.
We note that this construction does not depend on the choice of $m$.
In fact, it is enough to show that for integers $m^{\prime}\geq m\geq 0$, the natural morphism
\begin{align*}
&\phantom{{}\to{}}(M^{[p^{-k-i},p^{-k}]})^{m\mathchar`-\an}\otimes_{(\tilde{B}_{K_{\infty},A}^{[p^{-k-i},p^{-k}],m\mathchar`-\an})_{\square}}B_{K,A}^{[p^{-k-i},p^{-k}]}\\
&\to (M^{[p^{-k-i},p^{-k}]})^{m^{\prime}\mathchar`-\an}\otimes_{(\tilde{B}_{K_{\infty},A}^{[p^{-k-i},p^{-k}],m^{\prime}\mathchar`-\an})_{\square}}B_{K,A}^{[p^{-k-i},p^{-k}]}
\end{align*}
is an isomorphism for $k$ sufficiently large and $i=0,1$.
It becomes an isomorphism after taking the tensor product $-\otimes_{(B_{K,A}^{[p^{-k-i},p^{-k}]})_{\square}}\tilde{B}_{K_{\infty},A}^{[p^{-k-i},p^{-k}]}$, and by Proposition \ref{prop:ff descent} and Lemma \ref{lem:flatness}, it is already an isomorphism.
Therefore, we obtain the functor 
$$\beta \colon \VB_{\tilde{B}_{K_{\infty},A}}^{\varphi,\Gamma_K}\to \VB_{B_{K,A}}^{\varphi,\Gamma_K}.$$
By construction, we have $\alpha\circ \beta\simeq \id$. 

Let us prove $\beta\circ \alpha\simeq \id$.
We take a $(\varphi,\Gamma_K)$-module $N=\{N^{I}\}_{I}\in \VB_{B_{K,A}}^{\varphi,\Gamma_K}$ over $B_{K,A}$.
We set $\alpha(N)=\{N^{I}\otimes_{(B_{K,A}^{I})_{\square}} \tilde{B}_{K_{\infty},A}^{I}\}_{I}=\{M^{I}\}_I\in \VB_{\tilde{B}_{K_{\infty},A}}^{\varphi,\Gamma_K}$.
It is enough to show that there exists an integer $m\geq 0$ such that for $k$ sufficiently large and $i=0,1$, there exists a natural isomorphism
$$(M^{[p^{-k-i},p^{-k}]})^{m\mathchar`-\an}\otimes_{(\tilde{B}_{K_{\infty},A}^{[p^{-k-i},p^{-k}],m\mathchar`-\an})_{\square}}B_{K,A}^{[p^{-k-i},p^{-k}]}\to N^{[p^{-k-i},p^{-k}]}.$$
We fix an integer $l$.
Then, by Theorem \ref{thm:locally analytic descent} (2), there exists an integer $m\geq 0$ such that $N^{[p^{-l-i},p^{-l}]}$ and $B_{K,A}^{[p^{-l-i},p^{-l}]}$ are  $\Gamma_{K,m}$-analytic for $i=0,1$.
Then for $k\geq 0$, we have an isomorphism
\begin{align*}
    &(M^{[p^{-l-k-i},p^{-l-k}]})^{m\mathchar`-\an}\\
    \cong \phantom{}&(N^{[p^{-l-i},p^{-l}]}\otimes_{(B_{K,A}^{[p^{-l-i},p^{-l}]})_{\square},\varphi^k}\tilde{B}_{K_{\infty},A}^{[p^{-l-k-i},p^{-l-k}]})^{m\mathchar`-\an}\\
    \cong \phantom{}&N^{[p^{-l-i},p^{-l}]}\otimes_{(B_{K,A}^{[p^{-l-i},p^{-l}]})_{\square},\varphi^k}\tilde{B}_{K_{\infty},A}^{[p^{-l-k-i},p^{-l-k}],m\mathchar`-\an}.
\end{align*}
If $k$ is sufficiently large, then $\tilde{B}_{K_{\infty},A}^{[p^{-l-k-i},p^{-l-k}],m\mathchar`-\an}$ is contained in $B_{K,A}^{[p^{-l-k-i},p^{-l-k}]}$ by Proposition \ref{prop:la vector of perfect witt ring} (2).
Therefore, for $k$ sufficiently large, we get an isomorphism 
\begin{align*}
    &(M^{[p^{-l-k-i},p^{-l-k}]})^{m\mathchar`-\an}\otimes_{(\tilde{B}_{K_{\infty},A}^{[p^{-l-k-i},p^{-l-k}],m\mathchar`-\an})_{\square}}B_{K,A}^{[p^{-l-k-i},p^{-l-k}]}\\
    \cong \phantom{}&N^{[p^{-l-i},p^{-l}]}\otimes_{(B_{K,A}^{[p^{-l-i},p^{-l}]})_{\square},\varphi^k}B_{K,A}^{[p^{-l-k-i},p^{-l-k}]}\\
    \cong \phantom{}&N^{[p^{-l-k-i},p^{-l-k}]},
\end{align*}
which proves the claim.
\end{proof}

\begin{definition}\label{def:perfection and deperfection}
    We write the functors $$\alpha \colon \VB_{\Bcal_{K,\Acal}}^{\varphi,\Gamma_K} \to \VB_{\tilde{\Bcal}_{K_{\infty},\Acal}}^{\varphi,\Gamma_K}$$ and $$\beta \colon \VB_{\tilde{\Bcal}_{K_{\infty},\Acal}}^{\varphi,\Gamma_K}\to \VB_{\Bcal_{K,\Acal}}^{\varphi,\Gamma_K}$$ defined in the proof of Theorem \ref{thm:deperfection} as $$\VB_{\Bcal_{K,\Acal}}^{\varphi,\Gamma_K}\ni N\mapsto N\otimes_{\Bcal_{K,\Acal}}\tilde{\Bcal}_{K_{\infty},\Acal}\in\VB_{\tilde{\Bcal}_{K_{\infty},\Acal}}^{\varphi,\Gamma_K}$$
    and $$\VB_{\tilde{\Bcal}_{K_{\infty},\Acal}}^{\varphi,\Gamma_K}\ni M=\{M^{I}\}_I\mapsto M^{\dep}=\{M^{\dep,I}\}_I\in\VB_{\Bcal_{K,\Acal}}^{\varphi,\Gamma_K},$$
    where $\dep$ stands for deperfection.
\end{definition}

\section{Dualizability of $(\varphi,\Gamma)$-cohomology}
In this section, we prove the dualizability of the cohomology of $(\varphi,\Gamma_K)$-modules.
Our proof is similar to the proof in \cite{Bel24}.
In \cite{Bel24}, the dualizability is proved by constructing a quasi-isomorphism, which is completely continuous in every term of the complexes, between complexes which compute the cohomology of $(\varphi,\Gamma_K)$-modules.
A key point is that for closed intervals $I'\subsetneq I\subset (0,\infty)$, the restriction morphism $B_K^{I}\to B_K^{I'}$ is completely continuous.
In this paper, we interpret the above argument in terms of condensed mathematics, which makes it possible to treat the case of general coefficient rings which are not necessarily Banach.

Let $\Acal=(A,A^+)_{\square}$ be an algebraic-affinoid analytic $\Qbb_{p,\square}$-algebra or an analytic $\Qbb_{p,\square}$-algebra associated to a complete Tate affinoid pair $(A,A^+)$ over $\Qbb_p$.
First, we define a cohomology theory for $(\varphi,G_K)$-modules over $\tilde{\Bcal}_{\bar{K},\Acal}$ (resp.\ $(\varphi,\Gamma_K)$-modules over $\Bcal_{K,\infty,\Acal}$, $(\varphi,\Gamma_K)$-modules over $\tilde{\Bcal}_{K_{\infty},\Acal}$).

\begin{definition}\label{def:phi cohomology perfect}
    \begin{enumerate}
    \item Let $M=\{M^{I}\}_{I}$ be a $\varphi$-module over $\tilde{\Bcal}_{\bar{K},\Acal}$.
    Then we define the $\varphi$-cohomology of $M$ as 
    $$R\Gamma_{\varphi}(M)=\fib(\Phi-1\colon\varinjlim_{0<s}\Rvarprojlim_{0<r\leq s}M^{[r,s]}\to \varinjlim_{0<s}\Rvarprojlim_{0<r\leq s}M^{[r,s]})\in \Dcal(\Acal),$$
    where $\Phi$ is induced by the $\varphi$-semilinear morphisms
    $$\Phi\colon M^{[r,s]}\to M^{[r/p,s/p]}.$$
    \item Let $M=\{M^{I}\}_{I}$ be a $(\varphi,G_K)$-module over $\tilde{\Bcal}_{\bar{K},\Acal}$.
    Then the continuous $G_K$-action on $M$ induces a continuous $G_K$-action on $R\Gamma_{\varphi}(M)$.
    We define the $(\varphi,G_K)$-cohomology of $M$ as
    $$R\Gamma_{\varphi,G_K} (M)=R\Gamma(G_K,R\Gamma_{\varphi}(M)) \in \Dcal(\Acal).$$
    \end{enumerate}
    In the same way, we define $R\Gamma_{\varphi}(M)$ and $R\Gamma_{\varphi,\Gamma_K}(M)$ for any $\varphi$-module or $(\varphi,\Gamma_K)$-module $M$ over either $\tilde{\Bcal}_{K_{\infty},\Acal}$ or $\Bcal_{K,\infty,\Acal}$.
\end{definition}

The definition of $R\Gamma_{\varphi}(M)$ may seem to be complicated.
However, there is a simple presentation of $R\Gamma_{\varphi}(M)$ as follows:
\begin{lemma}\label{lem:perfect level cohomology simple presentation}
    Let $M=\{M^{I}\}_{I}$ be a $\varphi$-module over $\tilde{\Bcal}_{\bar{K},\Acal}$ (resp.\ $\tilde{\Bcal}_{K_{\infty},\Acal}$, $\Bcal_{K,\infty,\Acal}$), and $[r,s]\subset(0,\infty)$ be a closed interval such that $pr\leq s$.
    Let $1 \colon M^{[r,s]}\to M^{[r,s/p]}$ denote the restriction morphism.
    Then the natural morphism
    $$\fib(\Phi-1\colon M^{[r,s]}\to M^{[r,s/p]}) \to R\Gamma_{\varphi}(M)$$
    is an equivalence.
\end{lemma}
\begin{proof}
   Let $\Phi \colon \Rvarprojlim_{0<r\leq s}M^{[r,s]}\to \Rvarprojlim_{0<r\leq s/p}M^{[r,s/p]}$ denote the morphism induced by the $\varphi$-semilinear morphisms
    $\Phi \colon M^{[r,s]}\to M^{[r/p,s/p]},$
    and $$1 \colon \Rvarprojlim_{0<r\leq s}M^{[r,s]}\to \Rvarprojlim_{0<r\leq s/p}M^{[r,s/p]}$$ denote the morphism induced by the restriction morphisms.
    Then we have an equivalence $$R\Gamma_{\varphi}(M)\simeq\varinjlim_{0<s}\fib(\Phi-1\colon \Rvarprojlim_{0<r\leq s}M^{[r,s]}\to \Rvarprojlim_{0<r\leq s/p}M^{[r,s/p]}).$$
    Moreover, we have an equivalence
    \begin{align*}
    &\fib(\Phi-1\colon \Rvarprojlim_{0<r\leq s}M^{[r,s]}\to \Rvarprojlim_{0<r\leq s/p}M^{[r,s/p]}) \\
    \simeq &\Rvarprojlim_{0<r\leq s}\fib(\Phi-1\colon M^{[r,s]}\to M^{[r,s/p]}).
    \end{align*}
    Therefore, it is enough to show that the morphisms
    \begin{align*}
        &\fib(\Phi-1\colon M^{[r^{\prime},s]}\to M^{[r^{\prime},s/p]}) \to\fib(\Phi-1\colon M^{[r,s]}\to M^{[r,s/p]}),\\
        &\fib(\Phi-1\colon M^{[r,s]}\to M^{[r,s/p]}) \to\fib(\Phi-1\colon M^{[r,s^{\prime}]}\to M^{[r,s^{\prime}/p]})
    \end{align*}
    are equivalences for $r^{\prime}\leq r \leq pr\leq s^{\prime} \leq s$.
    We prove the first morphism is an equivalence.
   We may assume $r/p\leq r^{\prime}$.
   Then by the gluing condition, there is a cartesian and cocartesian diagram
    $$
    \xymatrix{
        M^{[r^{\prime},s]}\ar[r]\ar[d] & M^{[r^{\prime},s/p]}\ar[d]\\
        M^{[r,s]}\ar[r] & M^{[r,s/p]},\\
    }
    $$
    where the morphisms are the restriction morphisms.
    Moreover, we have the following commutative diagram from the assumption $r/p\leq r^{\prime}$:
    $$
    \xymatrix{
        M^{[r^{\prime},s]}\ar[r]^-{\Phi}\ar[d] & M^{[r^{\prime},s/p]}\ar[d]\\
        M^{[r,s]}\ar[r]^-{\Phi}\ar[ru]^-{\Phi} & M^{[r,s/p]},\\
    }
    $$
    where the vertical morphisms are the restriction morphisms.
    From these diagrams, we can show that 
    $$\fib(\Phi-1\colon M^{[r^{\prime},s]}\to M^{[r^{\prime},s/p]}) \to\fib(\Phi-1\colon M^{[r,s]}\to M^{[r,s/p]})$$
    is an equivalence by diagram chasing.
    Next, we show that 
    $$\fib(\Phi-1\colon M^{[r,s]}\to M^{[r,s/p]}) \to\fib(\Phi-1\colon M^{[r,s^{\prime}]}\to M^{[r,s^{\prime}/p]})$$ 
    is an equivalence.
    It is equivalent to show that 
    $$\fib(M^{[r,s]}\overset{\Phi-1}{\longrightarrow} M^{[r,s/p]}\overset{\Phi^{-1}}{\longrightarrow}M^{[pr,s]}) \to\fib(M^{[r,s^{\prime}]}\overset{\Phi-1}{\longrightarrow} M^{[r,s^{\prime}/p]}\overset{\Phi^{-1}}{\longrightarrow}M^{[pr,s^{\prime}]})$$ 
    is an equivalence, which can be proved in the same way as above.
\end{proof}
\begin{remark}
     Let $M=\{M^{I}\}_{I}$ be a $\varphi$-module over $\tilde{\Bcal}_{\bar{K},\Acal}$.
     Then we can regard $M$ as an object of $\Dcal(X_{\bar{K},\Acal})$, and we can identify $\fib(\Phi-1\colon M^{[r,s]}\to M^{[r,s/p]})$ with the cohomology $R\Gamma(X_{\bar{K},\Acal},M)$.
     By using this, we can simplify the proof of the above lemma.
\end{remark}

\begin{remark}\label{rem:variant of phi cohomology}
    Let $\Ccal$ be the category whose objects are closed intervals $[r,s]\subset (0,\infty)$ such that $r,s\in \Qbb$ and whose morphisms are defined by $$\Hom_{\Ccal}([r,s],[r',s'])=\{n\in \Zbb_{\geq 0}\mid [p^{n}r,p^{n}s]\subset [r',s']\},$$
    where the composition is defined by the usual addition of $\Zbb_{\geq 0}$.
    Then for a $\varphi$-module $M=\{M^{I}\}_{I}$ over $\tilde{\Bcal}_{\bar{K},\Acal}$, we can define a functor $\Ccal^{\op}\to \Dcal(\Acal);\; I\to M^{I}$ where we associates a morphism $n\colon I\to I'$ in $\Ccal$ to $\Phi^n \colon M^{I'}\to M^{I}$.
    It is also natural to define a $\varphi$-cohomology of $M$ as a limit $R\varprojlim_{\Ccal^{\op}} M^{I}$ in $\Dcal(\Acal)$.
    In fact, it is also equivalent to $R\Gamma(X_{\bar{K},A},M)$, where we regard $M$ as an object of $\Dcal(X_{\bar{K},A})$.
    Therefore, there is a natural equivalence $R\Gamma_{\varphi}(M)\simeq R\varprojlim_{\Ccal^{\op}} M^{I}$ in $\Dcal(\Acal)$.
\end{remark}

\begin{proposition}\label{prop:comparison of phi gamma cohomology}
    Let $M=\{M^{I}\}_{I}$ be a $(\varphi,G_K)$-module over $\tilde{\Bcal}_{\bar{K},\Acal}$, and $M^{H_K}$ (resp.\ $M^{H_K,\la}$) be the $(\varphi,\Gamma_K)$-module over $\tilde{\Bcal}_{K_{\infty},\Acal}$ (resp.\ $\Bcal_{K,\infty,\Acal}$) associated to $M$.
    Then the natural morphisms
    \begin{align*}
        R\Gamma_{\varphi,\Gamma_K}(M^{H_K})&\to R\Gamma_{\varphi,G_K}(M),\\
        R\Gamma_{\varphi,\Gamma_K}(M^{H_K,\la})&\to R\Gamma_{\varphi,\Gamma_K}(M^{H_K})
    \end{align*}
    are equivalences in $\Dcal(\Acal)$.
\end{proposition}
\begin{proof}
    It follows from the description in Lemma \ref{lem:perfect level cohomology simple presentation}, the proof of Theorem \ref{thm:FF curve descent}, and \cite[Theorem 5.3]{RJRC22}.
\end{proof}

Next, we define a cohomology theory for $(\varphi,\Gamma_K)$-modules over $\Bcal_{K,\Acal}$.
\begin{definition}\label{def:phi cohomology imperfect}
    Let $N=\{N^{I}\}_{I}$ be a $(\varphi,\Gamma_K)$-module over $\Bcal_{K,\Acal}$ defined over $(0,t]$.
    We define the $\varphi$-cohomology of $N$ as 
    $$R\Gamma_{\varphi}(N)=\fib(\Phi-1\colon\varinjlim_{0<s\leq t}\Rvarprojlim_{0<r\leq s}N^{[r,s]}\to \varinjlim_{0<s\leq t}\Rvarprojlim_{0<r\leq s}N^{[r,s]})\in \Dcal(\Acal),$$
    where $\Phi$ is induced by the $\varphi$-semilinear morphisms
    $$\Phi\colon N^{[r,s]}\to N^{[r/p,s/p]}.$$
    It does not depend on the choice of $t$.
    Moreover, the continuous $\Gamma_K$-action on $N$ induces a continuous $\Gamma_K$-action on $R\Gamma_{\varphi}(N)$.
    We define the $(\varphi,\Gamma_K)$-cohomology of $N$ as
    $$R\Gamma_{\varphi,\Gamma_K}(N)=R\Gamma(\Gamma_K,R\Gamma_{\varphi}(N)) \in \Dcal(\Acal).$$
\end{definition}
\begin{remark}
    Let $N=\{N^{I}\}_{I}$ be a $(\varphi,\Gamma_K)$-module over $B_{K,A}$.
    If $A$ is an affinoid $\Qbb_p$-algebra, then the derived limit $\Rvarprojlim_{0<r\leq s}N^{[r,s]}$ is static (i.e., $R^1\varprojlim_{0<r\leq s}N^{[r,s]}=0$), and $\Ncal=\varinjlim_{0<s\leq t}\varprojlim_{0<r\leq s}N^{[r,s]}$ is the corresponding $(\varphi,\Gamma_K)$-module over the Robba ring $\Rcal_A$.
    By the definition, we can identify $R\Gamma_{\varphi,\Gamma_K}(N)$ with the cohomology of the $(\varphi,\Gamma_K)$-module $\Ncal$ over $\Rcal_A$ defined via the Herr complex.
\end{remark}

We want a simple presentation of $R\Gamma_{\varphi,\Gamma_K}(N)$ for a $(\varphi,\Gamma_K)$-module $N=\{N^{I}\}_{I}$ over $\Bcal_{K,\Acal}$ as in Lemma \ref{lem:perfect level cohomology simple presentation}.
However, in this setting, $\Phi\colon N^{[r,s]}\to N^{[r/p,s/p]}$ is not an isomorphism since the morphism $\varphi\colon B_{K,A}^{[r,s]}\to B_{K,A}^{[r/p,s/p]}$ is not an isomorphism.
Therefore, the same argument as in the proof of Lemma \ref{lem:perfect level cohomology simple presentation} does not work.
Hence, we first compare the cohomology $R\Gamma_{\varphi,\Gamma_K}(N)$ and $R\Gamma_{\varphi,\Gamma_K}(N\otimes_{\Bcal_{K,\Acal}} \Bcal_{K,\infty,\Acal})$, and then use Lemma \ref{lem:perfect level cohomology simple presentation} to obtain a simple presentation of $R\Gamma_{\varphi,\Gamma_K}(N)$.
\begin{theorem}\label{thm:deperf and cohomology}
    Let $N=\{N^J\}_J$ be a $(\varphi,\Gamma_K)$-module over $B_{K,A}$.
    Let $I=[p^{-k},p^{-l}]$ be a closed interval for $k,l\in\Zbb$.
    Then for $l$ sufficiently large, the natural morphism
    $$R\Gamma(\Gamma_K,N^{I}) \to R\Gamma(\Gamma_K,N^{I}\otimes_{(B_{K,A}^{I})_{\square}} B_{K,\infty,A}^{I})$$
    is an equivalence.
\end{theorem}
Let $N=\{N^J\}_J$ be a $(\varphi,\Gamma_K)$-module over $\Bcal_{K,\Acal}$.
We set $$N_n^{I}=N^{I}\otimes_{(B_{K,A}^{I})_{\square}} \left(B_{K,n+1,A}^{I}/B_{K,n,A}^{I}\right).$$
Since the inclusion $$N^{I}\otimes_{(B_{K,A}^{I})_{\square}} B_{K,n,A}^{I}\to N^{I}\otimes_{(B_{K,A}^{I})_{\square}} B_{K,n+1,A}^{I}$$ 
has the $\Gamma_K$-equivariant retraction 
$$R_n \colon N^{I}\otimes_{(B_{K,A}^{I})_{\square}} B_{K,n+1,A}^{I}\to N^{I}\otimes_{(B_{K,A}^{I})_{\square}} B_{K,n,A}^{I},$$
where $R_n$ is induced by the normalized trace map $B_{K,n+1}^{I}\to B_{K,n}^{I}$,
we get a $\Gamma_K$-equivariant isomorphism 
\begin{align*}
    N_n^{I}
    \cong\Ker(R_n \colon N^{I}\otimes_{(B_{K,A}^{I})_{\square}} B_{K,n+1,A}^{I}\to N^{I}\otimes_{(B_{K,A}^{I})_{\square}} B_{K,n,A}^{I}).
\end{align*}
Therefore, we get a $\Gamma_K$-equivariant isomorphism $$N^{I}\oplus \left(\bigoplus_{n\geq 0}N_n^{I}\right)\cong N^{I}\otimes_{(B_{K,A}^{I})_{\square}} B_{K,\infty,A}^{I}.$$
Let $\gamma\in \Gamma_{K}$ denote a topological generator of the largest procyclic subgroup $\Gamma^{\prime}_K$ of $\Gamma_K$.
To prove Theorem \ref{thm:deperf and cohomology}, it is enough to show the following theorem:
\begin{theorem}\label{thm:deperf and cohomology 2}
We fix a sufficiently large integer $l\in \Zbb$. 
Then there exists $n^{\prime}$ such that for any $n\geq n^{\prime}$, 
\begin{align*}
    \gamma-1 \colon N_n^{[p^{-l-1},p^{-l}]} \to N_n^{[p^{-l-1},p^{-l}]},\\
    \gamma-1 \colon N_n^{[p^{-l},p^{-l}]} \to N_n^{[p^{-l},p^{-l}]}
\end{align*}
are isomorphisms.
\end{theorem}
\begin{proof}[Proof of Theorem \ref{thm:deperf and cohomology 2} $\implies$ Theorem \ref{thm:deperf and cohomology}]
    Since we have a $\Gamma_K$-equivariant isomorphism
    $$\Phi \colon N_n^{[p^{-l-1},p^{-l}]} \overset{\sim}{\to} N_{n-1}^{[p^{-l-2},p^{-l-1}]}$$
    for $n\geq 1$, we may assume $n^{\prime}=0$ by replacing $l$ with $l+n^{\prime}$.
    There is a resolution of the trivial $\Qbb_{p,\square}[\Gamma_{K}^{\prime}]$-module $\Qbb_p$
    $$0\to \Qbb_{p,\square}[\Gamma_{K}^{\prime}] \overset{\gamma-1}{\longrightarrow} \Qbb_{p,\square}[\Gamma_{K}^{\prime}] \to \Qbb_{p} \to 0,$$
    so we obtain an isomorphism $R\Gamma(\Gamma^{\prime}_{K},N) \simeq \fib (\gamma-1 \colon N \to N)$ for a $\Qbb_{p,\square}[\Gamma^{\prime}_K]$-module $N$.
    In particular, for $n\geq 0$, we have $R\Gamma(\Gamma_{K}^{\prime},N_n^{[p^{-l-1},p^{-l}]})=0.$
    Since we have an equivalence 
    $$R\Gamma(\Gamma_K/\Gamma_{K}^{\prime},R\Gamma(\Gamma_{K}^{\prime},N_n^{[p^{-l-1},p^{-l}]}))\simeq R\Gamma(\Gamma_{K},N_n^{[p^{-l-1},p^{-l}]}),$$ we obtain $R\Gamma(\Gamma_{K},N_n^{[p^{-l-1},p^{-l}]})=0$ for any $n\geq 0$.
    Therefore, the natural morphism $R\Gamma(\Gamma_K,N^{[p^{-l-1},p^{-l}]}) \to R\Gamma(\Gamma_K,N^{[p^{-l-1},p^{-l}]}\otimes_{(B_{K,A}^{[p^{-l-1},p^{-l}]})_{\square}} B_{K,\infty,A}^{[p^{-l-1},p^{-l}]})$ is an equivalence.
    By applying $\varphi$, we find that $$R\Gamma(\Gamma_K,N^{[p^{-l^{\prime}-1},p^{-l^{\prime}}]}) \to R\Gamma(\Gamma_K,N^{[p^{-l^{\prime}-1},p^{-l^{\prime}}]}\otimes_{(B_{K,A}^{[p^{-l^{\prime}-1},p^{-l^{\prime}}]})_{\square}} B_{K,\infty,A}^{[p^{-l^{\prime}-1},p^{-l^{\prime}}]})$$ is also an equivalence for $l^{\prime} \geq l$.
    Similarly, we can show that 
    $$R\Gamma(\Gamma_K,N^{[p^{-l^{\prime}},p^{-l^{\prime}}]}) \to R\Gamma(\Gamma_K,N^{[p^{-l^{\prime}},p^{-l^{\prime}}]}\otimes_{(B_{K,A}^{[p^{-l^{\prime}},p^{-l^{\prime}}]})_{\square}} B_{K,\infty,A}^{[p^{-l^{\prime}},p^{-l^{\prime}}]})$$ is also an equivalence for $l^{\prime} \geq l$.
    Therefore, by a gluing argument, we find that $$R\Gamma(\Gamma_K,N^{[p^{-k},p^{-l^{\prime}}]}) \to R\Gamma(\Gamma_K,N^{[p^{-k},p^{-l^{\prime}}]}\otimes_{(B_{K,A}^{[p^{-k},p^{-l^{\prime}}]})_{\square}} B_{K,\infty,A}^{[p^{-k},p^{-l^{\prime}}]})$$ is also an equivalence for $k\geq l^{\prime} \geq l$, which proves Theorem \ref{thm:deperf and cohomology}.
\end{proof}

The proof of Theorem \ref{thm:deperf and cohomology 2} is similar to the proof of Theorem \ref{thm:locally analytic descent}.
First, we prove Theorem \ref{thm:deperf and cohomology 2} when $A$ is a Banach $\Qbb_p$-algebra, then we prove it for general $A$ by using analogs of Proposition \ref{prop:la descent inj} and Proposition \ref{prop:la descent surj}.

\begin{proof}[Proof of Theorem \ref{thm:deperf and cohomology 2}]
    In the following proof, we consider only $$\gamma-1 \colon N_n^{[p^{-l-1},p^{-l}]} \to N_n^{[p^{-l-1},p^{-l}]}.$$
    The other case can be proved similarly. 
    If $A$ is a Banach $\Qbb_p$-algebra, then the theorem can be proved in the same way as in \cite[Lemma 3.7]{Bel24}.
    Let $A$ be an algebraic-affinoid $\Qbb_{p,\square}$-algebra, and $B$ be an affinoid $\Qbb_p$-algebra of definition of $A$.
    For the same reason as in Lemma \ref{lem:reduction to red}, we may assume that $A$ is an integral domain.
    Then we can take an injection to a Banach $\Qbb_p$-algebra $A\to A^{\prime}$ as in Lemma \ref{lem:injection to Banach}.
    For $n$ sufficiently large, 
    $$\gamma-1 \colon N_n^{[p^{-l-1},p^{-l}]}\otimes_{A_{\square}}A^{\prime} \to N_n^{[p^{-l-1},p^{-l}]}\otimes_{A_{\square}}A^{\prime}$$ is an isomorphism.
    Moreover, we have the following commutative diagram:
    $$
    \xymatrix{
        N_n^{[p^{-l-1},p^{-l}]}\ar[r]^-{\gamma-1}\ar[d] &N_n^{[p^{-l-1},p^{-l}]}\ar[d]\\
        N_n^{[p^{-l-1},p^{-l}]}\otimes_{A_{\square}}A^{\prime}\ar[r]^-{\gamma-1} &N_n^{[p^{-l-1},p^{-l}]}\otimes_{A_{\square}}A^{\prime},
    }
    $$
    where the vertical morphisms are injective since $N_n^{[p^{-l-1},p^{-l}]}$ is flat over $A_{\square}$.
    Therefore, we find that for $n$ sufficiently large, $\gamma-1 \colon N_n^{[p^{-l-1},p^{-l}]}\to N_n^{[p^{-l-1},p^{-l}]}$ is injective.
    Next, we prove that the above morphism $\gamma-1$ is surjective.
    By the same argument as in the proof of Proposition \ref{prop:la descent surj}, we can take a finite free $B_{K,B}^{[p^{-l-1},p^{-l}]}$-module $L$ with a continuous semilinear $\Gamma_K$-action and a $\Gamma_K$-equivariant surjection $$L\otimes_{B_{\square}}A=L\otimes_{(B_{K,B}^{[p^{-l-1},p^{-l}]})_{\square}} B_{K,A}^{[p^{-l-1},p^{-l}]}\to N^{[p^{-l-1},p^{-l}]}.$$
    Then for $n$ sufficiently large, $$\gamma-1 \colon L_n^{[p^{-l-1},p^{-l}]} \to L_n^{[p^{-l-1},p^{-l}]}$$ is an isomorphism.
    Moreover, we have the following commutative diagram:
    $$
    \xymatrix{
        L_n^{[p^{-l-1},p^{-l}]}\otimes_{B_{\square}}A \ar[r]^-{\gamma-1}\ar[d] &L_n^{[p^{-l-1},p^{-l}]}\otimes_{B_{\square}}A\ar[d]\\
        N_n^{[p^{-l-1},p^{-l}]}\ar[r]^-{\gamma-1} &N_n^{[p^{-l-1},p^{-l}]},
    }
    $$
    where the vertical morphisms are surjective.
    Therefore, we find that for $n$ sufficiently large, $\gamma-1 \colon N_n^{[p^{-l-1},p^{-l}]}\to N_n^{[p^{-l-1},p^{-l}]}$ is surjective.
\end{proof}

By using Theorem \ref{thm:deperf and cohomology}, we can prove the following theorem:
\begin{theorem}\label{thm:comparison of cohomology}
    Let $N=\{N^{I}\}_I$ be a $(\varphi,\Gamma_K)$-module over $\Bcal_{K,\Acal}$, and 
    \begin{align*}
        N\otimes_{\Bcal_{K,\Acal}}\Bcal_{K,\infty,\Acal}=&\{N^{I}\otimes_{\Bcal_{K,\Acal}^{I}} \Bcal_{K,\infty,\Acal}^{I}\}_I\\
        =&\{N^{I}\otimes_{(B_{K,A}^{I})_{\square}} B_{K,\infty,A}^{I}\}_I
    \end{align*}
    be a $(\varphi,\Gamma_K)$-module over $\Bcal_{K,\infty,\Acal}$ associated to $N$.
    Then the natural morphism 
    $$R\Gamma_{\varphi,\Gamma_K}(N) \to R\Gamma_{\varphi,\Gamma_K}(N\otimes_{\Bcal_{K,\Acal}}\Bcal_{K,\infty,\Acal})$$
    is an equivalence in $\Dcal(\Acal)$.
\end{theorem}
\begin{proof}
    We have a natural equivalence $$R\Gamma_{\varphi}(N)\simeq \varinjlim_{0<l}\Rvarprojlim_{0<l\leq k}\fib(\Phi-1\colon N^{[p^{-k},p^{-l}]}\to N^{[p^{-k},p^{-l-1}]}).$$
    Since the functor $R\Gamma(\Gamma_K,-)$ commutes with all small limits and colimits, it is enough to show that
    \begin{align*}
        R\Gamma(\Gamma_K,\fib(\Phi-1\colon &N^{[p^{-k},p^{-l}]}\to N^{[p^{-k},p^{-l-1}]}))\\
        \to R\Gamma(\Gamma_K,\fib(\Phi-1\colon &N^{[p^{-k},p^{-l}]}\otimes_{(B_{K,A}^{[p^{-k},p^{-l}]})_{\square}} B_{K,\infty,A}^{[p^{-k},p^{-l}]} \\
        \to &N^{[p^{-k},p^{-l-1}]}\otimes_{(B_{K,A}^{[p^{-k},p^{-l-1}]})_{\square}} B_{K,\infty,A}^{[p^{-k},p^{-l-1}]}))
    \end{align*}
    is an equivalence for $l$ sufficiently large.
    This follows from Theorem \ref{thm:deperf and cohomology}.
\end{proof}

As a corollary of the proof, we obtain a simple presentation of the cohomology of $(\varphi,\Gamma_K)$-modules over $\Bcal_{K,\Acal}$.
\begin{corollary}\label{cor:phi Gamma cohomology simple presentation}
     Let $N=\{N^{I}\}_I$ be a $(\varphi,\Gamma_K)$-module over $\Bcal_{K,A}$.
     Then for $l$ sufficiently large and $k>l$, we have the natural equivalence
     $$R\Gamma(\Gamma_K,\fib(\Phi-1\colon N^{[p^{-k},p^{-l}]}\to N^{[p^{-k},p^{-l-1}]})) \simeq R\Gamma_{\varphi,\Gamma_K}(N).$$
\end{corollary}
\begin{proof}
    It follows from the proof of Theorem \ref{thm:comparison of cohomology} and Lemma \ref{lem:perfect level cohomology simple presentation}.
\end{proof}
\begin{remark}
    If $A$ is an affinoid $\Qbb_p$-algebra, the above corollary was proved in \cite[Proposition 3.8]{Bel24}.
\end{remark}

\begin{corollary}\label{cor:basechange of phi Gamma cohomology}
Let $\Acal'=(A',A^{'+})_{\square}$ be an algebraic-affinoid analytic $\Qbb_{p,\square}$-algebra or an analytic $\Qbb_{p,\square}$-algebra associated to a complete Tate affinoid pair $(A',A^{'+})$ over $\Qbb_p$, and $\Acal \to \Acal'$ be a morphism of analytic $\Qbb_{p,\square}$-algebras.
Let $N=\{N^{I}\}_I$ be a $(\varphi,\Gamma_K)$-module over $\Bcal_{K,\Acal}$, and $N\otimes_{\Acal}\Acal'=\{N^{I}\otimes_{\Acal} \Acal' \}_I$ be a $(\varphi,\Gamma_K)$-module over $\Bcal_{K,\Acal'}$ obtained from $N$.
Then we have an equivalence 
$$R\Gamma_{\varphi,\Gamma_K}(N)\otimes_{\Acal}^{\Lbb}\Acal'\simeq R\Gamma_{\varphi,\Gamma_K}(N\otimes_{\Acal}\Acal').$$
\end{corollary}
\begin{proof}
    It follows from Lemma \ref{lem:fixed part and tensor products} and Corollary \ref{cor:phi Gamma cohomology simple presentation}.
\end{proof}

Let us prove the dualizability of the cohomology of $(\varphi,\Gamma_K)$-modules over $\Bcal_{K,\Acal}$.
For the proof, we introduce the following ring.
\begin{construction}
    Let $K_0^{\prime}$ be as in Construction \ref{construction:coefficient rings}.
    Let $[r,s]\subset (0,\infty)$ be a closed interval.
    Then we define $K_0^{\prime}$-algebras $\Ocal_{\Bbb_{K_0^{\prime}}}(\Bbb_{K_0^{\prime}}^{(r,s)})^{+}$ and $\Ocal_{\Bbb_{K_0^{\prime}}}(\Bbb_{K_0^{\prime}}^{(r,s)})$ as
    \begin{align*}
    &\Ocal_{\Bbb_{K_0^{\prime}}}(\Bbb_{K_0^{\prime}}^{(r,s)})^{+}=\{\sum_{n\in \Zbb}a_nT^n\in K_0^{\prime}[[T]] \mid \mbox{$a_n p^{[rn]},a_n p^{[sn]}\in \Ocal_{K_0^{\prime}}$ for any $n\in\Zbb$}\},\\
    &\Ocal_{\Bbb_{K_0^{\prime}}}(\Bbb_{K_0^{\prime}}^{(r,s)})
    =\{\sum_{n\in \Zbb}a_nT^n\in K_0^{\prime}[[T]] \mid 
    \mbox{There exists $l\in\Zbb$ such that}\\
    &\mbox{$a_n p^{[rn]},a_n p^{[sn]}\in p^l\Ocal_{K_0^{\prime}}$ for any $n\in \Zbb$}\},
    \end{align*}
    where $[x]$ stands for the largest integer less than or equal to $x$.
    We note $\Ocal_{\Bbb_{K_0^{\prime}}}(\Bbb_{K_0^{\prime}}^{(r,s)})=\Ocal_{\Bbb_{K_0^{\prime}}}(\Bbb_{K_0^{\prime}}^{(r,s)})^{+}[1/p]$.
    We have an isomorphism of $K_0^{\prime}$-modules 
    $$\prod_{n\in\Zbb} \Ocal_{K_0^{\prime}} \to \Ocal_{\Bbb_{K_0^{\prime}}}(\Bbb_{K_0^{\prime}}^{(r,s)})^{+} ;\; (a_n) \mapsto \sum_{n\in \Zbb}a_np^{-\min\{[rn],[sn]\}}T^n.$$
    By using this isomorphism, we endow $\Ocal_{\Bbb_{K_0^{\prime}}}(\Bbb_{K_0^{\prime}}^{(r,s)})^{+}$ with the product topology, and equip $\Ocal_{\Bbb_{K_0^{\prime}}}(\Bbb_{K_0^{\prime}}^{(r,s)})=\Ocal_{\Bbb_{K_0^{\prime}}}(\Bbb_{K_0^{\prime}}^{(r,s)})^{+}[1/p]$ with the topology for which $\Ocal_{\Bbb_{K_0^{\prime}}}(\Bbb_{K_0^{\prime}}^{(r,s)})^{+}\subset \Ocal_{\Bbb_{K_0^{\prime}}}(\Bbb_{K_0^{\prime}}^{(r,s)})$ is an open subset.
    For $0<r<r^{\prime}\leq s^{\prime}<s$, we have natural continuous injections
    $$\Ocal_{\Bbb_{K_0^{\prime}}}(\Bbb_{K_0^{\prime}}^{[r,s]})\hookrightarrow \Ocal_{\Bbb_{K_0^{\prime}}}(\Bbb_{K_0^{\prime}}^{(r,s)}) \hookrightarrow \Ocal_{\Bbb_{K_0^{\prime}}}(\Bbb_{K_0^{\prime}}^{[r',s']}).$$
    For $s$ sufficiently small, the image of the continuous injection 
    \begin{align}
    \Ocal_{\Bbb_{K_0^{\prime}}}(\Bbb_{K_0^{\prime}}^{(\frac{pr}{(p-1)e},\frac{ps}{(p-1)e})})\to \Ocal_{\Bbb_{K_0^{\prime}}}(\Bbb_{K_0^{\prime}}^{[\frac{pr^{\prime}}{(p-1)e},\frac{ps^{\prime}}{(p-1)e}]}) \cong B_K^{[r',s']}\label{eq3}
    \end{align}
    is stable under the $\Gamma_K$-action.
   Moreover, there exists a unique dotted arrow $\varphi$ rendering the following diagram commutative:
   $$
   \xymatrix{
    \Ocal_{\Bbb_{K_0^{\prime}}}(\Bbb_{K_0^{\prime}}^{(\frac{pr}{(p-1)e},\frac{ps}{(p-1)e})})\ar[r]\ar@{.>}[d]^-{\varphi} &\Ocal_{\Bbb_{K_0^{\prime}}}(\Bbb_{K_0^{\prime}}^{[\frac{pr^{\prime}}{(p-1)e},\frac{ps^{\prime}}{(p-1)e}]})\ar[r]^-{\cong} &B_{K}^{[r',s']}\ar[d]^-{\varphi}\\
    \Ocal_{\Bbb_{K_0^{\prime}}}(\Bbb_{K_0^{\prime}}^{(\frac{r}{(p-1)e},\frac{s}{(p-1)e})})\ar[r] &\Ocal_{\Bbb_{K_0^{\prime}}}(\Bbb_{K_0^{\prime}}^{[\frac{r^{\prime}}{(p-1)e},\frac{s^{\prime}}{(p-1)e}]})\ar[r]^-{\cong} &B_{K}^{[r^{\prime}/p,s^{\prime}/p]}.
   }
   $$ 
   Let $B_K^{(r,s)}$ denote $\Ocal_{\Bbb_{K_0^{\prime}}}(\Bbb_{K_0^{\prime}}^{(\frac{pr}{(p-1)e},\frac{ps}{(p-1)e})})$ equipped with the $\Gamma_K$-action induced by (\ref{eq3}) and $\varphi\colon \Ocal_{\Bbb_{K_0^{\prime}}}(\Bbb_{K_0^{\prime}}^{(\frac{pr}{(p-1)e},\frac{ps}{(p-1)e})}) \to \Ocal_{\Bbb_{K_0^{\prime}}}(\Bbb_{K_0^{\prime}}^{(\frac{r}{(p-1)e},\frac{s}{(p-1)e})})$.
   This definition does not depend on the choice of $r^{\prime},s^{\prime}$.
\end{construction}

\begin{theorem}\label{thm:dualizability}
    Let $N=\{N^{I}\}_I$ be a $(\varphi,\Gamma_K)$-module over $\Bcal_{K,\Acal}$.
    Then the $(\varphi,\Gamma_K)$-cohomology $R\Gamma_{\varphi,\Gamma_K}(N)$ of $N$ is a dualizable object in $\Dcal(\Acal)$. 
\end{theorem}
\begin{proof}
By Corollary \ref{cor:dualizable indep of A+}, it suffices to show that $R\Gamma_{\varphi,\Gamma_K}(N)$ is a dualizable object in $\Dcal(A_{\square})$.
We put $N^{(p^{-k},p^{-l})}=N^{[p^{-k},p^{-l}]}\otimes_{(B_{K,A}^{[p^{-k},p^{-l}]})_{\square}} B_{K,A}^{(p^{-k},p^{-l})}.$  
Since $B_K^{(p^{-k},p^{-l})}$ is isomorphic to $\left(\prod \Ocal_{K_0^{\prime}}\right)[1/p]$,
it is compact as an object of $\Dcal(\Qbb_{p,\square})$ (\cite[Theorem 3.27]{And21}).
Therefore, $B_{K,A}^{(p^{-k},p^{-l})}$ is a compact object of $\Dcal(A_{\square})$.
Since $N^{(p^{-k},p^{-l})}$ is a finite projective $B_{K,A}^{(p^{-k},p^{-l})}$-module, $N^{(p^{-k},p^{-l})}$ is also a compact object of $\Dcal(A_{\square})$.
We have morphisms 
\begin{align*}
&\fib(\Phi-1\colon N^{[p^{-k-1},p^{-k+2}]}\to N^{[p^{-k-1},p^{-k+1}]}) \\
\to &\fib(\Phi-1\colon N^{(p^{-k-1},p^{-k+2})}\to N^{(p^{-k-1},p^{-k+1})}) \\
\to &\fib(\Phi-1\colon N^{[p^{-k},p^{-k+1}]}\to N^{[p^{-k},p^{-k}]}).
\end{align*}
By taking the $\Gamma_K$-cohomology, we obtain morphisms 
\begin{align*}
    &R\Gamma(\Gamma_K,\fib(\Phi-1\colon N^{[p^{-k-1},p^{-k+2}]}\to N^{[p^{-k-1},p^{-k+1}]})) \\
    \to &R\Gamma(\Gamma_K,\fib(\Phi-1\colon N^{(p^{-k-1},p^{-k+2})}\to N^{(p^{-k-1},p^{-k+1})})) \\
    \to &R\Gamma(\Gamma_K,\fib(\Phi-1\colon N^{[p^{-k},p^{-k+1}]}\to N^{[p^{-k},p^{-k}]})).
    \end{align*}
By Corollary \ref{cor:phi Gamma cohomology simple presentation}, the first and third objects in the above morphisms are isomorphic to $R\Gamma_{\varphi,\Gamma_K}(N)$, and the composition of the above morphisms is an isomorphism.
Therefore, $R\Gamma_{\varphi,\Gamma_K}(N)$ is a direct summand of $$R\Gamma(\Gamma_K,\fib(\Phi-1\colon N^{(p^{-k-1},p^{-k+2})}\to N^{(p^{-k-1},p^{-k+1})})).$$
Since $\fib(\Phi-1\colon N^{(p^{-k-1},p^{-k+2})}\to N^{(p^{-k-1},p^{-k+1})})$ is a compact object of $\Dcal(A_{\square})$, $R\Gamma(\Gamma_K,\fib(\Phi-1\colon N^{(p^{-k-1},p^{-k+2})}\to N^{(p^{-k-1},p^{-k+1})}))$ is also a compact object of $\Dcal(A_{\square})$ by Proposition \ref{prop:compact nuclear group cohomology}.
Therefore, $R\Gamma_{\varphi,\Gamma_K}(N)$ is also a compact object of $\Dcal(A_{\square})$.
On the other hand, $\fib(\Phi-1\colon N^{[p^{-k},p^{-k+1}]}\to N^{[p^{-k},p^{-k}]})$ is a nuclear object of $\Dcal(A_{\square})$, so $R\Gamma_{\varphi,\Gamma_K}(N) \simeq R\Gamma(\Gamma_K,\fib(\Phi-1\colon N^{[p^{-k},p^{-l}]}\to N^{[p^{-k},p^{-l-1}]}))$ is also a nuclear object of $\Dcal(A_{\square})$ by Proposition \ref{prop:compact nuclear group cohomology}.
Therefore, $R\Gamma_{\varphi,\Gamma_K}(N)$ is a compact and nuclear object of $\Dcal(A_{\square})$, which proves the theorem by Lemma \ref{lem:dualizable compact nuclear}.
\end{proof}


\section{Classification of $(\varphi,\Gamma)$-modules of rank $1$}
In this section, we classify $(\varphi,\Gamma)$-modules of rank $1$.
\subsection{Continuous characters of $K^{\times}$}

Let $L$ be a complete non-archimedean field extension of $\Qbb_p$.
Let $A$ be an algebraic-affinoid $L_{\square}$-algebra, and $A^{\times}$ denote its (condensed) multiplicative group of $A$, defined by $A^{\times}(S)=A(S)^{\times}$ for any extremally disconnected set $S$.
For a locally profinite group $G$, we call a homomorphism of condensed groups $G \to A^{\times}$, where we regard a topological group $G$ as a condensed group, a \textit{continuous character}.
It corresponds bijectively to a continuous $A$-linear action of $G$ on $A$.
\begin{remark}
    If $A$ is an affinoid $\Qbb_p$-algebra, then the above notion of continuous characters agrees with the usual one of continuous characters.
\end{remark}

\begin{proposition}\label{prop:continuous character of profinite group}
   Assume that $G$ is compact.
   Let $\delta\colon G\to A^{\times}$ be a continuous character. 
   Then the exists an affinoid $L$-algebra of definition $A'\subset A$ such that $\delta$ factors through $A'^{\times}\subset A^{\times}$.
\end{proposition}
\begin{proof}
    We take an affinoid $L$-algebra of definition $A'\subset A$.
    We regard $\delta$ as an continuous $A'$-linear action of $G$ on $A$.
    By Proposition \ref{prop:relatively discrete action filtered colimit}, there exists a relatively discrete and finitely generated $A'$-submodule $1\in M\subset A$ stable under the action of $G$.
    Then, for any $g\in G$ and any integer $n\geq 0$, we have $g(1)^n=g^n(1)\in M$, so $g(1)\in A$ is integral over $A'$. 
    Therefore, by Lemma \ref{lem:A^b is integrally closed} and Proposition \ref{prop:A^b integral}, we get $g(1)\in A^b$ for any $g\in G$.
    By replacing $A'$ with a finite $A'$-algebra generated by $A^b\cap M$, we may assume that $g(1)\in A'$ for any $g\in G$.
    Then $A'\subset A$ is stable under the action of $G$, which implies that $\delta$ factors through $A'^{\times}\subset A^{\times}$.
\end{proof}

We construct a moduli space of continuous characters of $K^{\times}$ in $\AlgAnSp_{L_{\square}}$.
Let $\varpi$ be a uniformizer of $K$.
Then we have a decomposition $K^{\times}=\varpi^{\Zbb}\times \Ocal_{K}^{\times}$.
Hence, continuous characters of $K^{\times}$ correspond bijectively to pairs consisting of continuous characters of $\varpi^{\Zbb}$ and of $\Ocal_K^{\times}$.
Since continuous characters $\varpi^{\Zbb}\to A^{\times}$ corresponds bijectively to elements of $A^{\times}$, it suffices to consider continuous characters of $\Ocal_K^{\times}$.

Let $\Wcal$ denote the rigid analytic generic fiber of $\Spf \Ocal_L[[\Ocal_K^{\times}]]$.
Then, for any affinoid adic space $\Spa(A,A^+)$ over $L$, there is a natural bijection 
$$\Wcal(\Spa(A,A^+))\cong\{\mbox{continuous characters $\Ocal_K^{\times}\to A^{\times}$}\}.$$

\begin{proposition}\label{prop:continuous characters moduli}
    For an algebraic-affinoid pair $(A,A^+)$ over $L_{\square}$, there is a natural bijection
    $$\Wcal_{\square}(\AnSpec(A,A^+)_{\square})\cong\{\mbox{continuous characters $\Ocal_K^{\times}\to A^{\times}$}\},$$
    where $\Wcal_{\square}$ is obtained from $\Wcal$ via the functor in Proposition \ref{prop:adic space to alg-aff space}.
\end{proposition}
\begin{proof}
    We note that $\Wcal$ is a Stein space.
    In other words, there exist affinoid open subspaces $\Spa(R_1,R_1^+)\subset \Spa(R_2,R_2^+)\subset \cdots \subset \Wcal$ such that $\Wcal=\bigcup_{i\geq 1} \Spa(R_i,R_i^+)$.
    By the construction of the functor in Proposition \ref{prop:adic space to alg-aff space}, we have
    $$\Wcal_{\square}=\varinjlim_{i\geq 1} \AnSpec(R_i,R_i^+)_{\square}.$$
    Since $\AnSpec(A,A^+)_{\square}$ is compact in the category of analytic sheaves on $\AlgAff_{L_{\square}}^{\op}$, we get 
    $$\Wcal_{\square}(\AnSpec(A,A^+)_{\square})=\varinjlim_{i\geq 1} \Hom_{L_{\square}}((R_i,R_i^+)_{\square},(A,A^+)_{\square}).$$
    By Lemma \ref{lem:basic property of aff pair of def} (2), there is an isomorphism 
    $$\Hom_{L_{\square}}((R_i,R_i^+)_{\square},(A,A^+)_{\square})=\varinjlim_{(A',A'^+)}\Hom_{L}(\Spa(A',A'^+),\Spa(R_i,R_i^+)),$$ 
    where the colimit is taken over all affinoid pairs of definition $(A',A'^+)$ of $(A,A^+)$.
    Then, we also get
    \begin{align*}
        \Wcal_{\square}(\AnSpec(A,A^+)_{\square})=&\varinjlim_{(A',A'^+)}\varinjlim_{i\geq 1}\Hom_{L}(\Spa(A',A'^+),\Spa(R_i,R_i^+))\\
        =&\varinjlim_{(A',A'^+)}\Wcal(\Spa(A',A'^+)).
    \end{align*}
    Now, the claim follows from Proposition \ref{prop:continuous character of profinite group}.
\end{proof}

\begin{corollary}\label{cor:moduli of continuous charcaters of K^{times}}
    For an algebraic-affinoid pair $(A,A^+)$ over $L_{\square}$, there is a bijection
    $$(\Wcal_{\square}\times \Gbb_m^{\alg})(\AnSpec(A,A^+)_{\square})\cong\{\mbox{continuous characters $K^{\times}\to A^{\times}$}\},$$
    where $\Gbb_m^{\alg}=\AnSpec(L[T^{\pm 1}],\Zbb)_{\square}$.
    We note that this bijection depends on the choice of $\varpi\in K$.
\end{corollary}
\begin{proof}
    There is a natural bijection
    $$\Gbb_m^{\alg}(\AnSpec(A,A^+)_{\square})\cong A^{\times}(\ast),$$
    so we get the claim.
\end{proof}

For later use, we prove the following lemma.
\begin{lemma}\label{lem:character lift}
Let $A$ be an algebraic-affinoid $L_{\square}$-algebra, and $I\subset A$ be a nilpotent ideal.
Let $\bar{\delta}\colon K^{\times}\to (A/I)^{\times}$ be a continuous character.
Then, there exists a continuous character $\delta\colon K^{\times}\to A^{\times}$ that is a lift of $\bar{\delta}$.
\end{lemma}
\begin{proof}
    Since $K^{\times}=\varpi^{\Zbb}\times \Ocal_K^{\times}$, it suffices to show the claim for continuous characters of $\Ocal_K^{\times}$ and $\Zbb$.
    For $\Zbb$, it follows from the fact that $A^{\times}(\ast)\to (A/I)^{\times}(\ast)$ is surjective.
    Let $\bar{\delta}\colon \Ocal_K^{\times}\to (A/I)^{\times}$ be a continuous character.
    By Proposition \ref{prop:continuous character of profinite group}, there exists an affinoid $L$-algebra of definition $\bar{A}'\subset A/I$ such that $\bar{\delta}$ factors through $(\bar{A}')^{\times}\subset (A/I)^{\times}$.
    By the same argument as in the proof of Lemma \ref{lem:basic property of aff pair of def}, we take an affinoid $L$-algebra of definition $A'\subset A$ such that the image of $A'$ in $A/I$ is $\bar{A}'$.
    The continuous character $\bar{\delta}$ corresponds to a morphism
    $\Spa (\bar{A}',\bar{A}'^{\circ})\to \Wcal$
    by Proposition \ref{prop:continuous characters moduli}.
    Since $\Wcal$ is smooth over $\Spa L$, there exists a lift $\Spa (A',A'^{\circ})\to \Wcal$ of $\Spa (\bar{A}',\bar{A}'^{\circ})\to \Wcal$, which gives a lift $\delta \colon \Ocal_K^{\times}\to A'^{\times}\subset A^{\times}$ of $\bar{\delta}$.
\end{proof}

\subsection{Classification of $(\varphi,\Gamma)$-modules of rank $1$}
First, we recall the construction of $(\varphi,\Gamma_K)$-modules from continuous characters of $K^{\times}$ introduced in \cite{Nak13, KPX14}.

Let $K_0$ be the maximal unramified extension of $\Qbb_p$ in $K$, and we write $f=[K_0\colon \Qbb_p]$.
We fix a morphism $K_0\to W(\Ocal_{C^{\flat}})[1/p]$, which induces a morphism $K_0\to \tilde{B}^{I}_{\overline{K}}$ for each closed interval $I\subset (0,\infty)$.
Let $\Acal=(A,A^+)_{\square}$ be an algebraic-affinoid analytic $K_{0,\square}$-algebra.
We define an analytic $\Acal$-algebra
\begin{align*}
    \tilde{\Bcal}_{\bar{K},K_0,\Acal}^{I}=(\tilde{B}_{\bar{K}}^{I}, \tilde{B}_{\bar{K}}^{I,+})_{\square}\otimes_{K_{0,\square}}^{\Lbb}\Acal,
\end{align*}
which has a natural continuous action of $G_K$ and the Frobenius isomorphism 
$$\varphi^f\colon \tilde{\Bcal}_{\bar{K},K_0,\Acal}^{I} \to \tilde{\Bcal}_{\bar{K},K_0,\Acal}^{I/p^f}.$$
We also define an $A$-algebra
\begin{align*}
\tilde{B}_{\bar{K},K_0,A}^{I}=\tilde{B}_{\bar{K}}^{I}\otimes_{K_{0,\square}}A.
\end{align*}
As in Lemma \ref{lem:underlying algebra}, we can show that the underlying (condensed) ring of $\tilde{\Bcal}_{\bar{K},K_0,\Acal}^{I}$ is $\tilde{B}_{\bar{K},K_0,A}^{I}$. 
As in Definition \ref{def:perfect module}, we define $(\varphi^f,G_K)$-modules over $\tilde{\Bcal}_{\overline{K},K_0,\Acal}$.
Let $\VB_{\tilde{\Bcal}_{\bar{K},K_0,\Acal}}^{\varphi^f,G_K}$ denote the category of $(\varphi^f,G_K)$-modules over $\tilde{\Bcal}_{\bar{K},K_0,\Acal}$.

\begin{lemma}\label{lem:cat equiv}
There is a natural categorical equivalence 
$$\VB_{\tilde{\Bcal}_{\bar{K},\Acal}}^{\varphi,G_K}\simeq \VB_{\tilde{\Bcal}_{\bar{K},K_0,\Acal}}^{\varphi^f,G_K}.$$
\end{lemma}
\begin{proof}
    The natural morphism $\tilde{\Bcal}_{\bar{K},\Acal}^{I}\to \tilde{\Bcal}_{\bar{K},K_0,\Acal}^{I}$ is $(\varphi^f,G_K)$-equivariant.
    Therefore, by applying $-\otimes_{\tilde{\Bcal}_{\bar{K},\Acal}^{I}}\tilde{\Bcal}_{\bar{K},K_0,\Acal}^{I}$, we get a functor $\VB_{\tilde{\Bcal}_{\bar{K},\Acal}}^{\varphi,G_K}\to \VB_{\tilde{\Bcal}_{\bar{K},K_0,\Acal}}^{\varphi^f,G_K}$.
    Let us construct a quasi-inverse functor.
    Let $\varphi$ denote the element of the Galois group $\Gal(K_0/\Qbb_p)$ corresponding to the $p$-th power Frobenius.
    The morphism $K_0\to \tilde{B}_{\bar{K}}^{I}$ induces an isomorphism
    \begin{align*}
        \tilde{B}_{\bar{K}}^{I}\otimes_{\Qbb_{p,\square}} K_0 \cong \prod_{i=0}^{f-1} \varphi^{i}_*\tilde{B}_{\bar{K}}^{I},
    \end{align*}
    where $\varphi^{i}_*\tilde{B}_{\bar{K}}^{I}$ is the $K_0$-algebra whose structure morphism is 
    $$K_0\overset{\varphi^i}{\lra}K_0\to \tilde{B}_{\bar{K}}^{I}.$$
    Therefore, we get an isomorphism
    \begin{align*}
        \tilde{B}_{\bar{K},A}^{I}&\cong (\tilde{B}_{\bar{K}}^{I}\otimes_{\Qbb_{p,\square}} K_0)\otimes_{K_{0,\square}}A\\
        &\cong \prod_{i=0}^{f-1} \varphi^{i}_*\tilde{B}_{\bar{K},K_0,A}^{I}.
    \end{align*}
    Under this isomorphism, the Frobenius morphism 
    $$\varphi \colon \tilde{B}_{\bar{K},A}^{I} \to \tilde{B}_{\bar{K},A}^{I/p}$$ 
    corresponds to 
    $$
    \begin{array}{ccl}
        \prod_{i=0}^{f-1} \varphi^{i}_*\tilde{B}_{\bar{K},K_0,A}^{I} &\lra &\prod_{i=0}^{f-1} \varphi^{i}_*\tilde{B}_{\bar{K},K_0,A}^{I/p} \\
        \rotatebox{90}{$\in$}       &                &  \quad\quad\rotatebox{90}{$\in$}\\
        (x_0,\ldots,x_{f-2},x_{f-1})&\longmapsto &(\varphi(x_{f-1}),\varphi(x_0),\ldots,\varphi(x_{f-2})).
    \end{array}$$            
    Therefore, we get a functor $\VB_{\tilde{B}_{\bar{K},K_0,A}}^{\varphi^f,G_K}\to \VB_{\tilde{B}_{\bar{K},A}}^{\varphi,G_K}$ defined via
    $$\left\{M^{I}\right\}_{I}\mapsto \left\{\prod_{i=0}^{f-1} \varphi^i_*M^{I}\right\}_{I},$$
    which is a quasi-inverse functor.
\end{proof}

\begin{definition}\label{def:character type}
Let $\varpi$ be a uniformizer of $K$.
For a continuous character $\delta \colon K^{\times}\to A^{\times}$, we define $\tilde{B}_{\bar{K},K_0,A}(\delta,\varpi)\in\VB_{\tilde{\Bcal}_{\bar{K},K_0,\Acal}}^{\varphi^f,G_K}$ as follows:
The choice of $\varpi$ gives the decomposition $K^{\times}=\varpi^{\Zbb}\times \Ocal_K^{\times}$, and the local class field theory provides a continuous character
$$\delta_{\varpi}\colon G_K\to G_K^{\ab}\to K^{\times}/\varpi^{\Zbb}=\Ocal_K^{\times}\overset{\delta}{\lra}A^{\times},$$
where we note that this character depends on the choice of $\varpi$.
We define $$\tilde{B}_{\bar{K},K_0,A}(\delta,\varpi)=\{\tilde{B}_{\bar{K},K_0,A}^{I}e\}_{I}$$
where the Frobenius $\Phi$ acts on $e$ by $\Phi(e)=\delta(\varpi)e$, and $G_K$ acts on $e$ via $\delta_{\varpi}$.
Via categorical equivalences in Lemma \ref{lem:cat equiv}, Theorem \ref{thm:FF curve descent}, and Theorem \ref{thm:deperfection}, we define:
\begin{itemize}
    \item $\tilde{B}_{\bar{K},A}(\delta,\varpi)\in\VB_{\tilde{\Bcal}_{\bar{K},\Acal}}^{\varphi,G_K}$,
    \item $\tilde{B}_{K_{\infty},A}(\delta,\varpi)\in\VB_{\tilde{\Bcal}_{K_{\infty},\Acal}}^{\varphi,\Gamma_K}$,
    \item $B_{K,\infty,A}(\delta,\varpi)\in \VB_{\Bcal_{K,\infty,\Acal}}^{\varphi,\Gamma_K}$, 
    \item $B_{K,A}(\delta,\varpi)\in\VB_{\Bcal_{K,\Acal}}^{\varphi,\Gamma_K}$.
\end{itemize}
\end{definition}

The following lemma immediately follows from the definition.
\begin{lemma}\label{lem:character tensor product}
    \begin{enumerate}
    \item Let $\delta_1,\delta_2 \colon K^{\times}\to A^{\times}$ be continuous characters.
    Then there is a natural isomorphism
    $$\tilde{B}_{\bar{K},K_0,A}(\delta_1,\varpi)\otimes \tilde{B}_{\bar{K},K_0,A}(\delta_2,\varpi)\cong \tilde{B}_{\bar{K},K_0,A}(\delta_1\delta_2,\varpi).$$
    \item Let $\Acal'=(A',A'^{+})_{\square}$ be an algebraic-affinoid analytic $K_{0,\square}$-algebra, and $\Acal \to \Acal'$ be a morphism of analytic $K_{0,\square}$-algebras.
    Then, for a continuous character $\delta\colon K^{\times}\to A^{\times}$, there is a natural isomorphism
    $$\tilde{B}_{\bar{K},K_0,A}(\delta,\varpi)\otimes_{\Acal}\Acal' \cong \tilde{B}_{\bar{K},K_0,A'}(\delta,\varpi).$$
    \end{enumerate}
\end{lemma}

\begin{theorem}\label{thm:indep of uniformizer}
Let $\varpi,\varpi'$ be uniformizers of $K$.
Then, for a continuous character $\delta \colon K^{\times}\to A^{\times}$, there is an isomorphism
$$\tilde{B}_{\bar{K},K_0,A}(\delta,\varpi) \cong \tilde{B}_{\bar{K},K_0,A}(\delta,\varpi').$$
\end{theorem}
\begin{proof}
    If $\delta$ is an unramified character (that is, if $\delta|_{\Ocal_K^{\times}}$ is trivial), then the desired isomorphism immediately follows from the definition.
    By Lemma \ref{lem:character tensor product} (1) and twisting by an unramified character, we may assume $\delta(\varpi')=1$.
    Then, the continuous character $\delta\colon W_K^{\ab}\cong K^{\times}\overset{\delta}{\lra}A^{\times}$ uniquely extends to a continuous character
    $\delta \colon G_K^{\ab} \to A^{\times}$. 
    We write $a=\delta(\varpi)$.
    By construction, it suffices to show that there exists $b\in (A\otimes_{K_{0,\square}}\breve{\Qbb}_p)^{\times}$ such that $\varphi_f(b)/b=a$, where $\breve{\Qbb}_p$ is the completion of the maximal unramified extension of $\Qbb_p$, and $\varphi_f$ is induced by the $p^f$-th power Frobenius on $\breve{\Qbb}_p$. 
    In fact, the choice of $b$ provides the desired isomorphism
    $$\tilde{B}_{\bar{K},K_0,A}(\delta,\varpi)=\{\tilde{B}_{\bar{K},K_0,A}^{I}e\}\to \tilde{B}_{\bar{K},K_0,A}(\delta,\varpi')=\{\tilde{B}_{\bar{K},K_0,A}^{I}e'\} ;\; e\mapsto be'.$$
    Since the restriction $\delta|_{\Ocal_K^{\times}}$ is irrelevant for proving this claim, we will henceforth regard $\delta$ as a character of $\varpi^{\hat{\Zbb}} \cong \hat{\Zbb}$.
    
    By Proposition \ref{prop:continuous character of profinite group}, there exists an affinoid $K_{0}$-algebra of definition $A'\subset A$ such that $\delta \colon \hat{\Zbb}\to A^{\times}$ factors through $A'^{\times}\subset A^{\times}$.
    Therefore, after replacing $A$ with $A'$, we may assume that $A$ is an affinoid $K_0$-algebra.

    Let $A_0$ be a ring of definition of $A$.
    Since $\delta$ is continuous, there exists an integer $m\geq 1$ such that $a^m=\delta(m) \in 1+pA_0$.
    Therefore, $\delta$ factors as 
    $$\delta\colon \hat{\Zbb} \to \Zbb/n\Zbb\times \Zbb_p \to A^{\times},$$
    where $n$ is the prime to $p$ part of $m$.
    Therefore, we get a morphism
    $$K_{0,\square}[\Zbb/n\Zbb\times \Zbb_p]\cong K_0[[T]][X]/(X^n-1) \to A.$$
    Moreover, for $r>0$ sufficiently small, the above morphism factors as
    $$\begin{array}{ccccl}
    K_0[[T]][X]/(X^n-1) &\lra &K_0\langle T/p^r\rangle[X]/(X^n-1) &\lra &A \\
    & &\rotatebox{90}{$\in$}       &                &  \rotatebox{90}{$\in$}\\
    & & (T+1)X &\longmapsto &a.
    \end{array}$$
    Therefore, we may assume $A=K_0\langle T/p^r\rangle[X]/(X^n-1)$ and $a=TX+X$.
    Moreover, $K_0\langle T/p^r\rangle[X]/(X^n-1)$ is a product of algebras of the form $K_{0,d}\langle T/p^r\rangle$ where $K_{0,d}$ is the unramified extension of $K_0$ of degree $d$, so we may assume that $A=K_{0,d}\langle T/p^r\rangle$ and $a=(T+1)x$ for some $d$-th root of unity $x\in K_{0,d}$.
    Let $\sigma$ be the $p^f$-th power Frobenius on $K_{0,d}$.
    Then, there is an isomorphism
    \begin{align*}
        K_{0,d}\langle T/p^r\rangle\hotimes_{K_0}\breve{\Qbb}_p \to \prod_{i=0}^{d-1}\breve{\Qbb}_p\langle T/p^r \rangle ;\; \alpha\otimes\beta \mapsto (\sigma^i(\alpha)\beta)_i,
    \end{align*}
    and under this isomorphism, $(T+1)x$ corresponds to $((T+1)\sigma^i(x))_i$.
    Moreover, under this isomorphism, $\id\otimes\varphi_f$ on $K_{0,d}\langle t/p^r\rangle\hotimes_{K_0}\breve{\Qbb}_p$ corresponds to the morphism
    \begin{align*}
        &\prod_{i=0}^{d-1}\breve{\Qbb}_p\langle T/p^r \rangle  \to \prod_{i=0}^{d-1}\breve{\Qbb}_p\langle T/p^r \rangle ;\; \\
        &(a_0,\ldots,a_{d-2},a_{d-1})\mapsto (\varphi_f(a_{d-1}),\varphi_f(a_0),\ldots,\varphi_f(a_{d-2})).
    \end{align*}
    Therefore, it is enough to show that there exists $b\in \breve{\Qbb}_p\langle T/p^r \rangle$ such that 
    $$\varphi_f(b)/b=\prod_{i=0}^{d-1}(T+1)\sigma^i(x).$$
    The right-hand side can be written as $\zeta+c$ where $\zeta \in \breve{\Qbb}_p$ is a root of unity and $c$ is a topologically nilpotent element in $T\breve{\Qbb}_p\langle T/p^r \rangle$.
    There exists $b'\in \breve{\Zbb}_p^{\times}$ such that $\varphi_f(b')/b'=\zeta$, so we may assume $\zeta=1$.
    Noting that $\varphi_f-\id\colon \breve{\Zbb}\to \breve{\Zbb}$ is surjective, we can construct $b\in T\breve{\Qbb}_p\langle T/p^r \rangle$ inductively so that $\varphi_f(1+b)/(1+b)=1+c$.
\end{proof}

From now on, we omit $\varpi$ from the notation.
\begin{definition}
    A $(\varphi,\Gamma_K)$-module $M$ over $\Bcal_{K,\infty,\Acal}$ of rank $1$ is said to be \textit{of character type} if there exists a continuous character $\delta\colon K^{\times}\to A^{\times}$ and a finite projective $\Acal$-module $V$ of rank $1$ such that $M\cong B_{K,\infty,A}(\delta)\otimes_{\Acal}V$.
\end{definition}

We want to prove that every $(\varphi,\Gamma_K)$-module $M$ over $\Bcal_{K,\infty,\Acal}$ of rank $1$ is of character type (under some mild assumption).
The following case has already been proved by Kedlaya-Pottharst-Xiao.
\begin{theorem}[{\cite[Theorem 6.2.14]{KPX14}}]\label{thm:classfication of line bundle affinoid}
    If $A$ is an affinoid $K_0$-algebra, then every $(\varphi,\Gamma_K)$-module over $\Bcal_{K,\infty,\Acal}$ of rank $1$ is of character type.
    More precisely, for a $(\varphi,\Gamma_K)$-module $M$ over $\Bcal_{K,\infty,\Acal}$ of rank $1$, there exist a unique continuous character $\delta\colon K^{\times} \to A^{\times}$ and a unique finite projective $\Acal$-module $V$ of rank $1$ such that $$M\cong B_{K,\infty,A}(\delta)\otimes_{\Acal} V.$$
\end{theorem}

Our strategy is to reduce to the affinoid case. 
Before giving the proof, we present some preliminary results.
First, let us prove the uniqueness.

\begin{lemma}\label{lem:uniqueness of character}
    Let $\delta_1,\delta_2\colon K^{\times}\to A^{\times}$ be continuous characters, and $V_1,V_2$ be finite projective $\Acal$-modules of rank $1$.
    If $B_{K,\infty,A}(\delta_1)\otimes_{\Acal} V_1\cong B_{K,\infty,A}(\delta_2)\otimes_{\Acal} V_2$, then $\delta_1=\delta_2$ and $V_1 \cong V_2$.
\end{lemma}
\begin{proof}
    This claim holds when $A$ is an affinoid $K_0$-algebra by \cite[Theorem 6.2.14]{KPX14}.
    We set $\delta=\delta_1\delta_2^{-1}$ and $V=V_1\otimes V_2^{-1}$.
    We want to show $\delta=1$ and $V\cong A$.
    By the assumption, $B_{K,\infty,A}(\delta)\otimes_{\Acal} V$ is trivial. 
    Let $\mfrak \subset A$ be a maximal ideal.
    Then, $A/\mfrak^n$ is a finite $K_0$-algebra for any $n\geq 1$.
    In particular, $A/\mfrak^n$ is an affinoid $K_0$-algebra.
    Therefore, the character $K^{\times} \overset{\delta}{\lra} A^{\times}\to (A/\mfrak^n)^{\times}$ is trivial by the affinoid case.
    Since $A(\ast)$ is noetherian, the morphism $A(\ast)\to \prod_{\mfrak,n} A/\mfrak^n(\ast)$ is injective.
    By Lemma \ref{lem:condensification functor}, the morphism $A\to \prod_{\mfrak,n} A/\mfrak^n$ is also injective, so $\delta\colon K^{\times}\to A^{\times}$ is trivial.
    Moreover, there is an isomorphism 
    $$A\cong H^0_{\varphi,\Gamma_K}(B_{K,\infty,A})\cong H^0_{\varphi,\Gamma_K}(B_{K,\infty,A}\otimes_{\Acal}V)\cong V,$$
    which proves the claim.
\end{proof}

\begin{lemma}\label{lem:charcater type analytic local}
    Let $M=\{M^I\}_I$ be a $(\varphi,G_K)$-module over $\tilde{\Bcal}_{\bar{K},\Acal}$ of rank $1$.
    If there exists an affinoid covering $\{\Acal=(A,A^+)_{\square}\to \Bcal_i=(B_i,B_i^+)_{\square}\}_{i=1}^n$ such that, for all $i$, $M_i=M \otimes_{\Acal}\Bcal_i$ is of character type, then $M$ is also of character type.
\end{lemma}
\begin{proof}
    We take a unique continuous character $\delta_i \colon K^{\times}\to B_i^{\times}$, a finite projective $\Bcal_i$-module $V_i$ of rank $1$, and an isomorphism $M_i\cong \tilde{B}_{\bar{K},B_i}(\delta_i)\otimes_{\Bcal_i} V_i.$
    By gluing $\delta_i$, we get a continuous character $\delta\colon K^{\times}\to A^{\times}.$
    By replacing $M$ with $M\otimes \tilde{B}_{\bar{K},A}(\delta^{-1})$, we may assume $\delta=1$.
    We write $V=R\Gamma_{\varphi}(M)\in \Dcal(\Acal)$.
    We note that there is an isomorphism
    $$V\otimes_{\Acal}^{\Lbb} \Bcal_i \cong R\Gamma_{\varphi}(M_i),$$
    which is a finite projective $\Bcal_i$-module by the proof of Theorem \ref{thm:Galois representation fully faithful functor}.
    Therefore, by Proposition \ref{prop:descent finite projective Fredholm} (1), $V$ is dualizable over $\Acal$.
    In particular, $V$ is nuclear in $\Dcal(A_{\square})$ by Lemma \ref{lem:dualizable compact nuclear} and Lemma \ref{lem:A^+ nuclear}.
    
    Let $I\subset (0,\infty)$ be a closed interval.
    Then, the morphism $V\otimes_{\Acal}^{\Lbb} \tilde{B}_{\bar{K},A}^I\to M^I$ becomes an isomorphism after applying $-\otimes_{\Acal}^{\Lbb}\Bcal_i$.
    Therefore, $V\otimes_{\Acal}^{\Lbb} \tilde{B}_{\bar{K},A}^I\to M^I$ is already an isomorphism.
    Since $V\otimes_{A_{\square}}^{\Lbb} \tilde{B}_{\bar{K},A}^I$ is a nuclear object in $\Dcal(A_{\square})$, it is $\Acal$-complete by Lemma \ref{lem:A^+ nuclear}.
    Therefore, the morphism $V\otimes_{A_{\square}}^{\Lbb} \tilde{B}_{\bar{K},A}^I\to V\otimes_{\Acal}^{\Lbb} \tilde{B}_{\bar{K},A}^I$ is an isomorphism, and $V\otimes_{A_{\square}}^{\Lbb} \tilde{B}_{\bar{K},A}^I\to M^I$ is also an isomorphism.
    Since $\tilde{B}_{\bar{K},A}^I$ is faithfully flat over $A_{\square}$, $V$ is a finite projective $A_{\square}$-module of rank $1$ by Theorem \ref{thm:descent of finite projective modules}. 
    We note that $V=R\Gamma_{\varphi}(M)$ admits a natural $A$-linear continuous $G_K$-action induced by that on $M$.
    Since $V$ is a finite projective $A_{\square}$-module of rank $1$, this $G_K$-action defines a continuous character $\tau \colon G_K\to A^{\times}$.
    The continuous $G_K$-action on $V\otimes_{\Acal} \Bcal_i\cong V_i$ is trivial, so the composition of $\tau$ and the natural map $A^{\times}\to B_i^{\times}$ is trivial for all $i$.
    Since the morphism $A^{\times}\to \prod_{i=1}^n B_i^{\times}$ is injective, $\tau$ is also trivial. 
    Hence, we get an isomorphism $M\cong V\otimes_{\Acal}\tilde{B}_{\bar{K},A}$ of $(\varphi,G_K)$-modules, which proves that $M$ is of character type.
\end{proof}

In the same way as above, we can also show the following.
\begin{lemma}\label{lem:character type ff local}
    Let $M$ be a $(\varphi,G_K)$-module over $\tilde{\Bcal}_{\bar{K},\Acal}$ of rank $1$.
    Let $B$ be a relatively discrete and finitely generated $\Acal$-algebra, and we define an analytic $\Acal$-algebra $\Bcal$ whose underlying condensed ring is $B$ and whose analytic ring structure is induced from $\Acal$.
    Assume that $B$ is faithfully flat as an $\Acal$-module.
    If $M \otimes_{\Acal}\Bcal$ is of character type, then $M$ is also of character type.
\end{lemma}

\begin{lemma}\label{lem:etale bundle character type}
    Let $M$ be a $(\varphi,G_K)$-module over $\tilde{\Bcal}_{\bar{K},\Acal}$ of rank $1$.
    If $M$ is isomorphic to $\tilde{B}_{\bar{K},A}$ as $\varphi$-modules, then $M$ is of character type.
\end{lemma}
\begin{proof}
    It easily follows from the proof of Theorem \ref{thm:Galois representation fully faithful functor}.
\end{proof}

\begin{lemma}\label{lem:restriction character type}
    Let $L$ be a finite extension of $K$, and $L_0$ be the maximal unramified extension of $\Qbb_p$ in $L$.
    Assume that $\Acal=(A,A^+)_{\square}$ is an algebraic-affinoid analytic $L_{0,\square}$-algebra.
    Let $M$ be a $(\varphi,G_K)$-module over $\tilde{\Bcal}_{\bar{K},\Acal}$ of rank $1$.
    \begin{enumerate}
        \item We assume that $M$ is of character type, and we take a character $\delta \colon K^{\times}\to A^{\times}$ and a finite projective $\Acal$-module $V$ such that $M\cong\tilde{B}_{\bar{K},A}(\delta)\otimes_{\Acal} V$.
        Let $\delta_L\colon L^{\times}\to K^{\times} \overset{\delta}{\to} A^{\times}$ be the composition of $\delta$ and the norm map $L^{\times}\to K^{\times}$.
        Then there is an isomorphism
        $$M\cong\tilde{B}_{\bar{L},A}(\delta_L)\otimes_{\Acal} V$$
        as $(\varphi,G_L)$-modules over $\tilde{\Bcal}_{\bar{K},\Acal}=\tilde{\Bcal}_{\bar{L},\Acal}$.
        \item If $M$ is of character type as a $(\varphi,G_L)$-module, then it is also of character type as a $(\varphi,G_K)$-module.
    \end{enumerate}
\end{lemma}
\begin{proof}
    The claim (1) easily follows from the construction. 
    Let us prove (2).
    We take a character $\delta \colon L^{\times}\to A^{\times}$ and a finite projective $\Acal$-module $V$ such that $M\cong\tilde{B}_{\bar{L},A}(\delta)\otimes_{\Acal} V$. 
    By replacing $M$ with $M\otimes_{\Acal} V^{-1}$, we may assume that $V$ is trivial.
    Let $d=[L_0: K_0]$, and $\pi_L$ be a uniformizer of $L$.
    We set $a=\delta(\pi_L)$.
    By Lemma \ref{lem:character type ff local}, after replacing $A$ with $A[X]/(X^d-a)$, we may assume that there exists $b\in A$ such that $b^d=a$.
    We take the unramified character $\tau \colon K^{\times}\to A^{\times}$ such that $\tau(\pi_K)=b$, where $\pi_K$ is a uniformizer of $K$.
    We replace $M$ with $M\otimes \tilde{B}_{\bar{K},A}(\tau^{-1})$.
    Then as a $(\varphi,G_L)$-module, we have $M\cong \tilde{B}_{\bar{L},A}(\delta\cdot \tau_L^{-1})$, where $\tau_L$ is as in (1).
    By construction, the continuous character $\delta\cdot \tau_L^{-1}\colon W_L^{\ab}\cong L^{\times}\to A^{\times}$ extends to a continuous character $G_L^{\ab}\to A^{\times}$. 
    Hence, by the proof of Theorem \ref{thm:indep of uniformizer}, $M$ is trivial as a $\varphi$-module. 
    Therefore, we get the claim by Lemma \ref{lem:etale bundle character type}.
\end{proof}

\begin{lemma}\label{lem:regular}
    Let $R\to S$ be a smooth morphism of affinoid $\Qbb_p$-algebras. 
    If $R$ is normal, then $S$ is also normal.
\end{lemma}
\begin{proof}
    By \cite[Satz 3.3.3]{BKKN67}, $R$ is excellent.
    By  \cite[Th\'{e}or\`{e}me]{Andre74} (or \cite[0H7U]{stacks-project}) and \cite[Proposition 2.9]{FRG3}, we find that $R\to S$ is regular.
    Then, the claim follows from \cite[Corollary 21.E]{Mat80}.
\end{proof}

\begin{theorem}\label{thm:classification rank 1}
    Let $M=\{M^I\}_I$ be a $(\varphi,\Gamma_K)$-module over $\Bcal_{K,\infty,\Acal}$ of rank $1$.
    Let $r$ be a sufficiently large integer such that we can define $B_{K,A}^{[p^{-r},p^{-r+1}]}$, and we write $I=[p^{-r},p^{-r+1}]$.
    Assume that there exists an affinoid covering $\{\Acal\to \Bcal_i\}_{i=1}^n$ such that $M^{I}\otimes_{\Acal}\Bcal_i$ is a free $\Bcal_{K,\infty,\Bcal_i}^I$-module of rank $1$.
    Then $M$ is of character type.
\end{theorem}

\begin{remark}
    We expect that for any $(\varphi,\Gamma_K)$-module $M$ over $\Bcal_{K,\infty,\Acal}$, there exists an affinoid covering $\{\Acal\to \Bcal_i\}_{i=1}^n$ such that $M^{I}\otimes_{\Acal}\Bcal_i$ is a finite free $\Bcal_{K,\infty,\Bcal_i}$-module.
    If $A$ is an affinoid $K_0$-algebra, then we can show this by a standard argument.
\end{remark}

\begin{proof}
    By Lemma \ref{lem:character type ff local}, after replacing $\Acal$ with $\Acal\otimes_{K_{0,\square}} K'_{0,\square}$, where $K'_0$ is the maximal unramified extension of $\Qbb_p$ in $K_{\infty}$, we may assume that $\Acal$ is a $K'_{0,\square}$-algebra.
    Moreover, by Lemma \ref{lem:restriction character type}, we may replace $K$ with $K\cdot K'_0$.
    By replacing $K$ with $K_n=K(\zeta_{p^n})$ for $n$ sufficiently large, we may assume that $\Gamma_K=\Gamma_{K_n}$ acts on the set of connected components of $\Spec B_{K,n,K_0}^I$ trivially.
    Since $B_{K,n,K_0}^I$ is isomorphic to $K_0\otimes_{\Qbb_{p}} K_0\langle T/p^a, p^b/T\rangle$ for some $a,b$, for an affinoid $K'_0=K_0$-algebra $R$ which is an integral domain, the set of connected components of $\Spec B_{K,n,R}^I$ is isomorphic to that of $\Spec B_{K,n,K_0}^I$.
    Therefore, for $R$ as above, $\Gamma_K$ acts on the set of connected components of $\Spec B_{K,n,R}^I$ trivially.

    By Lemma \ref{lem:charcater type analytic local}, we may assume that $M^{I}$ is a free $\Bcal_{K,\infty,\Acal}^I$-module of rank $1$.
    Moreover, by Proposition \ref{prop:Zariski open}, we may assume that $\Acal$ is of the form $(R,R^+)_{\square}[1/f]$ where $(R,R^+)$ is an affinoid pair over $(K_0,\Ocal_{K_0})$ of weakly finite type and $f\in R$.
    
    First, we prove the case where $R$ is reduced.
    By replacing $R$ with the integral closure of $R$ in $R_f$, we may assume that $R\to R_f$ is injective and integrally closed.
    By the proof of Theorem \ref{thm:deperfection}, for $n$ sufficiently large, we obtain a $\Gamma_K$-stable free $B_{K,n,R_f}^I$-submodule $M_n^I \subset M^I$ of rank $1$ such that $M_n^I\otimes_{(B_{K,n,R_f}^I)_{\square}}B_{K,\infty,R_f}^I \cong M^I$.
    We take $e \in M_n^I$ such that $M_n^I=B_{K,n,R_f}^I\cdot e.$
    For $\gamma \in \Gamma_K$, we define $x_{\gamma}\in (B_{K,n,R_f}^I)^{\times}$ so that $\gamma (e)=x_{\gamma} e$.

    \begin{claim}\label{claim1}
    For any $\gamma\in \Gamma_K$, $x_{\gamma}\in (B_{K,n,R}^I)^{\times}$.
    In particular, $B_{K,n,R}^I\cdot e \subset B_{K,n,R_f}^I=M_n^I$ is stable under the action of $\Gamma_K$.
    \end{claim}
	\begin{proof}[Proof of Claim \ref{claim1}]
        Since $x_{\gamma}\cdot\gamma(x_{\gamma^{-1}})=1$, it suffices to show $x_{\gamma}\in B_{K,n,R}^I$.
        First, we assume that $R$ is normal.
        Since $\Gamma_K$ is compact, by the proof of Proposition \ref{prop:relatively discrete action filtered colimit}, there exists $m>0$ such that $\Gamma_K(B_{K,n,R}^I\cdot e)\subset f^{-m} B_{K,n,R}^I\cdot e$.
        Then, for any $j>0$, we get
        \begin{align}\label{eq:positivity}
            x_{\gamma^j}=\gamma^{j-1}(x_{\gamma})\cdot \gamma^{j-2}(x_{\gamma})\cdot \cdots \cdot x_{\gamma}\in f^{-m}B_{K,n,R}^I.
        \end{align}
        By Lemma \ref{lem:regular}, $B_{K,n,R}^I$ is also normal.
        Let $\Lambda$ be the set of prime divisors of $\Spec B_{K,n,R}^I$ that occur as irreducible components of the preimage of prime divisors of $\Spec R$.
        We note that $\Gamma_K$ acts on $\Lambda$ trivially.
        We set $V(x_{\gamma})=a_1D_1+\cdots+a_sD_s+b_1E_1+\cdots+b_tE_t$ where $D_1,\ldots,D_s,E_1,\ldots, E_t$ are pairwise distinct prime divisors of $\Spec B_{K,n,R}^I$ such that $D_i$ (resp.\ $E_j$) is contained (resp.\ not contained) in $\Lambda$. 
        We need to show $$a_1,\ldots,a_s,b_1,\ldots,b_t \geq 0.$$
        We set $V(f^m)=c_1D'_1+\cdots+c_uD'_u$ where $D'_1,\ldots,D'_u$ are pairwise distinct prime divisors of $\Spec B_{K,n,R}^I$, where we note that $D'_i$ lies in $\Lambda$. 
        Then, we have $$a_1D_1+\cdots+a_sD_s+c_1D'_1+\cdots+c_uD'_u+b_1E_1+\cdots+b_tE_t\geq 0,$$
        so we get $b_1,\ldots,b_t \geq 0.$
        Moreover, by \eqref{eq:positivity}, we get
        \begin{align*}
        &j(a_1D_1+\cdots+a_sD_s)+c_1D'_1+\cdots+c_uD'_u\\
        +&(\mbox{a linear combination of prime divisors outside $\Lambda$})\geq 0
    \end{align*}
        for any $j>0$.
        Therefore, we obtain $a_1,\ldots,a_s\geq 0$.

        Next, we prove the general case.
        Let $S$ be the integral closure of $R$ in its total ring of fractions.
        Since $S$ is finite over $R$, $S$ is also an affinoid $K_0$-algebra.
        Since $R\to R_f$ is integrally closed, there is an exact sequence 
        $$0\to R \to R_f\times S \to S_f$$
        of $R_{\square}$-modules.
        Since $B_{K,n,R}^I$ is flat over $R_{\square}$, we get an exact sequence
        $$0\to B_{K,n,R}^I\to B_{K,n,R_f}^I\times B_{K,n,S}^I\to B_{K,n,S_f}^I.$$
        By the previous paragraph, we have $x_{\gamma}\in B_{K,n,S}^I$.
        Therefore, we get 
        $$x_{\gamma}\in B_{K,n,R_f}^I\cap B_{K,n,S}^I=B_{K,n,R}^I.$$
    \end{proof}
    We write $I_i=[p^{-r+i},p^{-r+i}]$ ($i=0,1$).
    Then we get 
    $$\Phi\colon M^{I_1}=B_{K,n,R_f}^{I_1}\cdot e \to M^{I_0}=B_{K,n}^{I_0}\cdot e.$$
    Let $a\in (B_{K,n,R_f}^{I_0})^{\times}$ such that $\Phi(e)=a\cdot e$.
    We write $t=\log[\varepsilon] \in B_{K,n,R_f}^{I_0}$ and $L=B_{K,n}^{I_0}/t$, which is a finite extension of $K$ contained in $K_{\infty}$.
    By replacing $R$ with $R\otimes_{K_{0,\square}} L$, we may assume that $R$ is an $L$-algebra.

    \begin{claim}\label{claim2}
    There exist $b\in (R_f)^{\times}$ and $c\in (B_{K,n,R}^{I_0})^{\times}$ such that $a=bc$.
    \end{claim}
	\begin{proof}[Proof of Claim \ref{claim2}]
        To simplify the notation, we write $B=B_{K,n,R}^{I_0}(\ast)$.
        We note $B_{K,n,R_f}^{I_0}(\ast)=B_f=B[1/f].$
        By the commutativity of the Frobenius $\varphi$ and the $\Gamma_K$-action, we obtain
        \begin{align}\label{eq:claim2}
        \gamma(a)/a=\varphi(x_{\gamma})/x_{\gamma}\in B^{\times}.
        \end{align}
        We note that the proof of Claim \ref{claim2} relies only on the above.
        Let $\alpha\colon B_f \to R_f$ be the composition of the projection morphism $B\to B_{K,n,R_f}^{I_0}/t=R_f\otimes_{\Qbb_p}L$ and $R_f\otimes_{\Qbb_p}L\to R_f$ induced from the $L$-algebra structure on $R_f$.
        By replacing $a$ with $a\alpha(a)^{-1}$, we may assume $\alpha(a)=1$.
        It is enough to show $a\in B$ since we can show $a^{-1}\in B$ by the same argument.
        We take $t_L\in B_{K,n,L}^{I_0}$ such that $\Ker(B_{K,n,L}^{I_0}\to L)=(t_L)$.
        Let $B^{\wedge}$ (resp.\ $(B_f)^{\wedge}$) be the $t_L$-adic completion of $B$ (resp.\ $B_f$).
        Since $t_L,f$ form a regular sequence in $B$, there is an exact sequence
        $$0\to B \to B^{\wedge}\times B_f \to B_f^{\wedge},$$
        where we note that the natural morphism 
        $B^{\wedge}[1/f]\to (B_f)^{\wedge}$ is injective.
        Therefore, it suffices to show that the image of $a$ in $B_f^{\wedge}$ is contained in $B^{\wedge}$.
        We prove that the image of $a$ in $B_f/t_L^m$ is contained in $B/t_L^m$ by induction on $m$.
        The case $m=1$ is clear since $\alpha(a)=1$.
        Assume that the statement holds for $m$.
        We take $a_1,\ldots,a_{m-1}\in R$ such that $a\equiv (1+a_1t_L)\cdots(1+a_{m-1}t_L^{m-1}) \mod t_L^{m}$.
        We set 
        $$a'=a(1+a_1t_L)^{-1}\cdots(1+a_{m-1}t_L^{m-1})^{-1}\in (B_f^{\wedge})^{\times}.$$
        By \eqref{eq:claim2}, for any $\gamma\in \Gamma_L$, we have $\gamma(a')/a'\in (B^{\wedge})^{\times}.$
        We write $a'=1+a_mt_L^m$ where $a_m\in B_f^{\wedge}$.
        Then, for $\gamma\in \Gamma_L$, we have 
        \begin{align*}
            \gamma(a')/a'\equiv 1+t_L^m(\chi(\gamma)a_m-a_m) \mod t_L^{m+1},
        \end{align*}
        where $\chi$ is the $p$-adic cyclotomic character.
        Therefore, we get 
        $$((\chi(\gamma)-1)a_m \mod t_L)\in R$$ 
        for $\gamma\in \Gamma_L$.
        Since there is $\gamma\in \Gamma_L$ such that $\chi(\gamma)-1\in \Qbb_p^{\times}\subset R^{\times}$, we find that the image of $a_m$ in $R_f$ is contained in $R$, which proves the case $m+1$.
    \end{proof}
    Let $\delta \colon K^{\times}\to R_f^{\times}$ be the unramified continuous character such that $\delta(\varpi)=b$ where $\varpi$ is a uniformizer of $K$.
    By construction, $M\otimes B_{K,\infty,R_f}(\delta^{-1})$ can be obtained by the scalar extension of a $(\varphi,\Gamma_K)$-module over $B_{K,\infty,R}$ of rank $1$, and therefore, it is of character type by Theorem \ref{thm:classfication of line bundle affinoid}.
    Hence, $M$ is also of character type.

    Finally, we prove the case where $R$ is not necessarily reduced.
    We write $N=M\otimes_{B_{K,\infty,R_f}}\tilde{B}_{\bar{K},R_f}$.
    As in the proof of Lemma \ref{lem:reduction to red}, take a filtration of ideals of $R$
    $$0=J_0\subset J_1 \subset \cdots \subset J_r =\sqrt{(0)}$$
    such that for each $i$, there exists a prime ideal $\pfrak_i \subset R$ with $J_i/J_{i-1}\cong R/\pfrak_i.$
    We proceed by induction on $r$.
    The case $r=0$ is the reduced case. 
    Let $r>0$.
    We write $\bar{R}=R/J_1$.
    We take a continuous character $\bar{\delta}\colon K^{\times}\to \bar{R}_f^{\times}$ and a finite projective $\bar{R}_f$-module $\bar{V}$ of rank $1$ such that $N\otimes_{(R_f)_{\square}}\bar{R}_f \cong \tilde{B}_{\bar{K},\bar{R}_f}(\bar{\delta})\otimes \bar{V}$.
    Since $J_1$ is a nilpotent ideal, we can lift $\bar{\delta}$ (resp.\ $\bar{V}$) to a continuous character $\delta \colon K^{\times} \to R_f^{\times}$ (resp.\ a finite projective $R_f$-module $V$ of rank $1$).
    By replacing $N$ with $N\otimes \tilde{B}_{\bar{K},R_f}(\delta^{-1})\otimes V^{-1}$, we may assume that $N\otimes_{(R_f)_{\square}}\bar{R}_f$ is a trivial $(\varphi,G_K)$-module.
    We take a lift $e \in N^I$ of $1\in \tilde{B}_{\bar{K},\bar{R}_f}^I\cong N^I\otimes_{(R_f)_{\square}}\bar{R}_f$. 
    Then we have $N^I=\tilde{B}_{\bar{K},R_f}^I\cdot e$.
    Let $a\in (B_{K,n,R_f}^{I_0})^{\times}$ such that $\Phi(e)=a\cdot e$ where $\Phi$ is the Frobenius morphism
    $$\Phi\colon N^{I_1}=\tilde{B}_{\bar{K},R_f}^{I_1}\cdot e \to N^{I_0}=\tilde{B}_{\bar{K},R_f}^{I_0}\cdot e.$$
    Then, we have $a\equiv 1 \mod J_1$.
    We fix an isomorphism $R/\pfrak_1\cong J_1;\; 1\mapsto x$.
    We write $a=1+xb$ where $b\in \tilde{B}_{\bar{K},(R/\pfrak_1)_f}^{I_1}.$
    By \cite[Proposition II.2.5]{FS24}, we have
    \begin{align*}
        \fib(\tilde{B}_{\bar{K}}^{I}\overset{\varphi-1}{\lra} \tilde{B}_{\bar{K}}^{I_1})\cong \Qbb_p.
    \end{align*}
    Therefore, we get 
    \begin{align*}
        \fib(\tilde{B}_{\bar{K},(R/\pfrak_1)_f}^{I}\overset{\varphi-1}{\lra} \tilde{B}_{\bar{K},(R/\pfrak_1)_f}^{I_1})\cong (R/\pfrak_1)_f.
    \end{align*}
    Hence, we can take $b'\in \tilde{B}_{\bar{K},(R/\pfrak_1)_f}^{I}$ such that $b+\varphi(b')-b'=0$.
    By replacing $e$ with $(1+xb')e$, we may assume $a=1$, where we note $x^2=0$.
    (Indeed, there is an integer $s>1$ such that $x^s=0$.
    Then $x^{s-1}=0$ in $R/\pfrak_1$, hence $x=0$ in $R/\pfrak_1$.
    In other words, $x^2=0$ in $R$.)
    Then, $N$ is trivial as a $\varphi$-module, so we get the claim by Lemma \ref{lem:etale bundle character type}.
\end{proof}

\begin{theorem}
Let $\Xfrak_{1,K}^{\mathrm{mod}}$ be the sheafification (with respect to the analytic topology) of the presheaf of anima on $\AlgAff_{K_{0,\square}}^{\op}$ that assigns to each $\Acal=(A,A^+)_{\square}\in \AlgAff_{K_{0,\square}}$ the groupoid of $(\varphi,\Gamma_K)$-modules over $\Bcal_{K,\infty,\Acal}$ of rank $1$ satisfying the condition in Theorem \ref{thm:classification rank 1}.
We write $\Gbb_m^{\alg}=\AnSpec (K_0[T^{\pm 1}],\Ocal_{K_0})_{\square}$.
Let $\Wcal$ denote the rigid analytic generic fiber of $\Spf \Ocal_{K_0}[[\Ocal_K^{\times}]]$, and $\Wcal_{\square}$ denote the algebraic-analytic space over $K_{0,\square}$ associated to $\Wcal$ via the functor in Proposition \ref{prop:adic space to alg-aff space}.
Then there is an equivalence
$$\Xfrak_{1,K}^{\mathrm{mod}}\cong [(\Wcal_{\square}\times \Gbb_m^{\alg})/\Gbb_m^{\alg}],$$
where the action of $\Gbb_m^{\alg}$ is trivial.
\end{theorem}
\begin{remark}
    It is likely that there exists a more suitable topology, such as the \'{e}tale or flat topology, for defining the stack, but we will not pursue this direction in this paper.
\end{remark}
\begin{proof}
    By Corollary \ref{cor:moduli of continuous charcaters of K^{times}} and Theorem \ref{thm:classification rank 1}, there is an epimorphism
    \begin{align}\label{eq:epi}
        \Wcal_{\square}\times \Gbb_m^{\alg}\to \Xfrak_{1,K}^{\mathrm{mod}}.
    \end{align}
    For a $(\varphi,\Gamma_K)$-module $M$ over $\Bcal_{K,\infty,\Acal}$ of rank $1$, we have
    $$\Aut_{\varphi,\Gamma_K}(M)\cong \Aut_{\varphi,\Gamma_K}(B_{K,\infty,A})=A^{\times}=\Gbb_m^{\alg}(\Acal).$$
    Hence, the epimorphism \eqref{eq:epi} induces the desired equivalence
    $$[(\Wcal_{\square}\times \Gbb_m^{\alg})/\Gbb_m^{\alg}]\cong \Xfrak_{1,K}^{\mathrm{mod}},$$
    where we also use Lemma \ref{lem:uniqueness of character} implicitly.
\end{proof}


\appendix
\section{The difference between dualizable complexes and perfect complexes}\label{appendix:nonFredholm}
In this appendix, we provide examples illustrating the difference between dualizable complexes and perfect complexes over algebraic-affinoid $\Qbb_{p,\square}$-algebras.
\begin{example}\label{example:counter example phi Gamma cohomology perfect}
    For simplicity, we assume $K=\Qbb_p$ and $p>2$.
    We put $A=\Qbb_p\langle T^{\pm 1}\rangle[S^{\pm 1}]$.
    Then we have a continuous character 
    $$\delta\colon \Qbb_p^{\times}=p^{\Zbb}\times \mu_{p-1} \times (1+p\Zbb_p)\to A^{\times}$$
    defined by $\delta(p)=S$, $\delta\vert_{\mu_{p-1}}=1$, and $\delta(\exp(p))=T$, where $\exp \colon p\Zbb_p \overset{\sim}{\to} 1+p\Zbb_p$ is the $p$-adic exponential map. 
    We show that $R\Gamma_{\varphi,\Gamma_{\Qbb_p}}(B_{\Qbb_{p},A}(\delta))$ is not a perfect complex.
    For a maximal ideal $\mfrak\subset A$, let $\delta_{\mfrak} \colon \Qbb_p^{\times} \to (A/\mfrak)^{\times}$ denote the continuous character defined by the composition of $\delta$ and $A^{\times}\to (A/\mfrak)^{\times}$.
    We note that $A/\mfrak$ is a finite extension of $\Qbb_p$.
    Then we have 
    \begin{align*}
    R\Gamma_{\varphi,\Gamma_{\Qbb_p}}(B_{\Qbb_{p},A}(\delta))\otimes_{A_{\square}}^{\Lbb}A/\mfrak &\simeq R\Gamma_{\varphi,\Gamma_{\Qbb_p}}(B_{\Qbb_{p},A/\mfrak}(\delta_{\mfrak})),\\
    H^2_{\varphi,\Gamma_{\Qbb_p}}(B_{\Qbb_{p},A/\mfrak}(\delta))\otimes_{A_{\square}}A/\mfrak &\simeq H^2_{\varphi,\Gamma_{\Qbb_p}}(B_{\Qbb_{p},A/\mfrak}(\delta_{\mfrak})).
    \end{align*}
    By \cite[Proposition 6.2.8]{KPX14}, $H^2_{\varphi,\Gamma_{\Qbb_p}}(B_{\Qbb_{p},A/\mfrak}(\delta_{\mfrak}))$ is non-trivial if and only if the continuous character $\delta\colon \Qbb_p^{\times}\to (A/\mfrak)^{\times}$ is equal to $x\mapsto |x|x^n$ for some $n\in \Zbb_{> 0}$.
    If $R\Gamma_{\varphi,\Gamma_{\Qbb_p}}(B_{\Qbb_{p},A}(\delta))$ is a perfect complex, then $H^2_{\varphi,\Gamma_{\Qbb_p}}(B_{\Qbb_{p},A}(\delta))$ is relatively discrete and finitely presented over $A$.
    Therefore, the support of $H^2_{\varphi,\Gamma_{\Qbb_p}}(B_{\Qbb_{p},A}(\delta))$ defines a Zariski closed subset of $\Spec A(\ast)$.
    Since a Zariski closed subset of $\Spec A(\ast)$ contains only finitely many closed points or contains uncountably many closed points, it is a contradiction.
\end{example}

\begin{example}\label{ex:nonzero pointwise zero}
    We construct a non-zero dualizable object $M\in \Dcal(\Qbb_p[T^{\pm 1}]_{\square})$ such that for every maximal ideal $\mfrak\subset \Qbb_p[T^{\pm 1}]$, $M\otimes_{\Qbb_p[T^{\pm 1}]_{\square}}^{\Lbb}\Qbb_p[T^{\pm 1}]/\mfrak=0$.
    We write $L=\left(\prod_{n\geq 0}\Zbb_p\right)[1/p]$, which is a compact object of $\Dcal(\Qbb_{p,\square})$.
    We define an endomorphism $f\colon L \to L$ as $$(a_0,a_1,a_2,a_3\ldots) \mapsto (0,a_0,pa_1,p^2a_2,\ldots),$$
    which is a trace-class morphism.
    Therefore, $M=\varinjlim(L\overset{f}{\to}L\overset{f}{\to}\cdots)$ is a nuclear object of $\Dcal(\Qbb_{p,\square})$.
    We regard $M$ as an object of $\Dcal(\Qbb_p[T^{\pm 1}]_{\square})$ by defining the action of $T$ via $f$.
    Then $M$ is also nuclear as an object of $\Dcal(\Qbb_p[T^{\pm 1}]_{\square})$ since $\Qbb_p[T^{\pm 1}]$ is nuclear over $\Qbb_{p,\square}$ (\cite[Lemma 2.9]{Mikami23}).
    Since $f^n(1,0,0,\ldots)\neq 0$ for every $n\geq 0$, $M$ is not zero.
    Moreover, we have a resolution of $\Qbb_p[T^{\pm 1}]_{\square}$-modules
    \begin{align}\label{eq5}
        0\mapsto L\otimes_{\Qbb_{p,\square}}^{\Lbb}\Qbb_p[T^{\pm 1}] \overset{\alpha}{\to} L\otimes_{\Qbb_{p,\square}}^{\Lbb}\Qbb_p[T^{\pm 1}] \to M\to 0,
    \end{align}
    where $\alpha$ is defined as $a\otimes x \mapsto a\otimes x -f(a)\otimes T^{-1}x$, so $M$ is a compact object of $\Dcal(\Qbb_p[T^{\pm 1}]_{\square})$.
    Therefore, $M$ is  a non-zero dualizable object of $\Dcal(\Qbb_p[T^{\pm 1}]_{\square})$ by Lemma \ref{lem:dualizable compact nuclear}.
    Let us prove for every maximal ideal $\mfrak\subset \Qbb_p[T^{\pm 1}]$, $M\otimes_{\Qbb_p[T^{\pm 1}]_{\square}}^{\Lbb}\Qbb_p[T^{\pm 1}]/\mfrak=0$.
    For simplicity, we prove in the case where $\mfrak=(T^{-1}-p^au)$ for $a\in \Zbb$ and $u\in \Zbb_p^{\times}$.
    By the resolution (\ref{eq5}), it is enough to show that $\id-p^auf \colon L \to L$ is an isomorphism.
    For simplicity, we assume $a=-1$.
    Then $\id-p^auf$ induces an endomorphism $$\prod_{n\geq 2}\Zbb_p \to \prod_{n\geq 2}\Zbb_p;\; (a_2,a_3,a_4,\ldots)\mapsto (a_2, a_3-pua_2,a_4-p^2ua_3,\ldots),$$
    which is an isomorphism because it becomes an isomorphism modulo $p$.
    Therefore we get the following commutative diagram:
    $$
    \xymatrix{
        0 \ar[r] & \left(\prod_{n\geq 2}\Zbb_p\right)[1/p]\ar[r]\ar[d]^-{\id-p^{-1}uf}& \left(\prod_{n\geq 0}\Zbb_p\right)[1/p]\ar[r]\ar[d]^-{\id-p^{-1}uf} & \Qbb_p^2\ar[r]\ar[d] &0\\
        0 \ar[r] & \left(\prod_{n\geq 2}\Zbb_p\right)[1/p]\ar[r]& \left(\prod_{n\geq 0}\Zbb_p\right)[1/p]\ar[r] & \Qbb_p^2\ar[r] &0,
    }
    $$
    where the left vertical morphism is an isomorphism and the right vertical morphism is given by the invertible matrix
    $\begin{pmatrix}
       1 & 0 \\
       p^{-1}u & 1 
    \end{pmatrix}.$
    Therefore, $\id-p^{-1}uf \colon L \to L$ is an isomorphism.
\end{example}


\bibliography{phi-Gamma-cohomology}

\providecommand{\bysame}{\leavevmode\hbox to3em{\hrulefill}\thinspace}
\providecommand{\MR}{\relax\ifhmode\unskip\space\fi MR }
\providecommand{\MRhref}[2]{%
  \href{http://www.ams.org/mathscinet-getitem?mr=#1}{#2}
}
\providecommand{\href}[2]{#2}
\begin{thebibliography}{BKKN67}

\bibitem[And74]{Andre74}
Michel Andr\'e, \emph{Localisation de la lissit\'e{} formelle}, Manuscripta
  Math. \textbf{13} (1974), 297--307.

\bibitem[And21]{And21}
Grigory Andreychev, \emph{Pseudocoherent and perfect complexes and vector
  bundles on analytic adic spaces}, arXiv preprint arXiv:
  \href{https://arxiv.org/abs/2105.12591}{2105.12591} (2021).

\bibitem[And23]{And23}
\bysame, \emph{K-theorie adischer r\"aume}, Ph.D. thesis, Rheinische
  Friedrich-Wilhelms-Universit{\"a}t Bonn, September 2023,
  \url{https://hdl.handle.net/20.500.11811/11040}.

\bibitem[BC08]{BC08}
Laurent Berger and Pierre Colmez, \emph{Familles de repr\'{e}sentations de de
  {R}ham et monodromie {$p$}-adique}, no. 319, 2008, Repr\'{e}sentations
  $p$-adiques de groupes $p$-adiques. I. Repr\'{e}sentations galoisiennes et
  $(\phi,\Gamma)$-modules, pp.~303--337.

\bibitem[Bel24]{Bel24}
Rebecca Bellovin, \emph{Cohomology of ({$\varphi,\; \Gamma$})-modules over
  pseudorigid spaces}, Int. Math. Res. Not. IMRN (2024), no.~4, 2999--3051.

\bibitem[Ber02]{Ber02}
Laurent Berger, \emph{Repr\'esentations {$p$}-adiques et \'equations
  diff\'erentielles}, Invent. Math. \textbf{148} (2002), no.~2, 219--284.

\bibitem[Ber08]{Ber08B-pair}
\bysame, \emph{Construction de {$(\phi,\Gamma)$}-modules: repr\'{e}sentations
  {$p$}-adiques et {$B$}-paires}, Algebra Number Theory \textbf{2} (2008),
  no.~1, 91--120.

\bibitem[Ber16]{Ber16}
\bysame, \emph{Multivariable {$(\varphi,\Gamma)$}-modules and locally analytic
  vectors}, Duke Math. J. \textbf{165} (2016), no.~18, 3567--3595.

\bibitem[BKKN67]{BKKN67}
R.~Berger, R.~Kiehl, E.~Kunz, and Hans-Joachim Nastold,
  \emph{Differentialrechnung in der analytischen {G}eometrie}, Lecture Notes in
  Mathematics, vol. No. 38, Springer-Verlag, Berlin-New York, 1967.

\bibitem[BLR95]{FRG3}
Siegfried Bosch, Werner L\"utkebohmert, and Michel Raynaud, \emph{Formal and
  rigid geometry. {III}. {T}he relative maximum principle}, Math. Ann.
  \textbf{302} (1995), no.~1, 1--29.

\bibitem[CC98]{CC98}
Fr\'ed\'eric Cherbonnier and Pierre Colmez, \emph{Repr\'esentations
  {$p$}-adiques surconvergentes}, Invent. Math. \textbf{133} (1998), no.~3,
  581--611.

\bibitem[Col08]{Col08}
Pierre Colmez, \emph{Espaces vectoriels de dimension finie et
  repr\'{e}sentations de de {R}ham}, no. 319, 2008, Repr\'{e}sentations
  $p$-adiques de groupes $p$-adiques. I. Repr\'{e}sentations galoisiennes et
  $(\phi,\Gamma)$-modules, pp.~117--186.

\bibitem[Col10]{Col10}
\bysame, \emph{Repr\'esentations de {${\rm GL}_2(\mathbf Q_p)$} et
  {$(\phi,\Gamma)$}-modules}, Ast\'erisque (2010), no.~330, 281--509.

\bibitem[Col16]{Col16}
\bysame, \emph{Repr\'esentations localement analytiques de {${\bf GL}_2({\bf
  Q}_p)$} et {$(\varphi,\Gamma)$}-modules}, Represent. Theory \textbf{20}
  (2016), 187--248.

\bibitem[CS22]{CC}
Dustin Clausen and Peter Scholze, \emph{Condensed mathematics and complex
  geometry}, \url{https://people.mpim-bonn.mpg.de/scholze/Complex.pdf}, 2022.

\bibitem[EGH25]{EGH23}
Matthew Emerton, Toby Gee, and Eugen Hellmann, \emph{An introduction to the
  categorical {$p$}-adic {L}anglands program}, arXiv preprint arXiv:
  \href{https://arxiv.org/abs/2210.01404}{22210.01404} (2025).

\bibitem[FF18]{FF18}
Laurent Fargues and Jean-Marc Fontaine, \emph{Courbes et fibr\'{e}s vectoriels
  en th\'{e}orie de {H}odge {$p$}-adique}, Ast\'{e}risque (2018), no.~406,
  xiii+382, With a preface by Pierre Colmez.

\bibitem[Fon90]{Fontaine90}
Jean-Marc Fontaine, \emph{Repr\'esentations {$p$}-adiques des corps locaux.
  {I}}, The {G}rothendieck {F}estschrift, {V}ol. {II}, Progr. Math., vol.~87,
  Birkh\"auser Boston, Boston, MA, 1990, pp.~249--309.

\bibitem[FS24]{FS24}
Laurent Fargues and Peter Scholze, \emph{Geometrization of the local
  {L}anglands correspondence}, arXiv preprint arXiv:
  \href{https://arxiv.org/abs/2102.13459}{2102.13459} (2024).

\bibitem[Her98]{Herr98}
Laurent Herr, \emph{Sur la cohomologie galoisienne des corps {$p$}-adiques},
  Bull. Soc. Math. France \textbf{126} (1998), no.~4, 563--600.

\bibitem[Ked04]{Ked04}
Kiran~S. Kedlaya, \emph{A {$p$}-adic local monodromy theorem}, Ann. of Math.
  (2) \textbf{160} (2004), no.~1, 93--184.

\bibitem[Ked19]{AWS17Ked}
\bysame, \emph{Sheaves, stacks, and shtukas}, Mathematical Surveys and
  Monographs, vol. 242, American Mathematical Society, Providence, RI, 2019, In
  Perfectoid spaces: Lectures from the 2017 Arizona Winter School, held in
  Tucson, AZ, March 11--17, Edited and with a preface by Bryden Cais, With an
  introduction by Peter Scholze.

\bibitem[KM25]{KM25}
Kiran~S. Kedlaya and Yutaro Mikami, \emph{On the relative nullstellensatz in
  nonarchimedean geometry}, arXiv preprint arXiv:
  \href{https://arxiv.org/abs/2503.18183}{2503.18183} (2025).

\bibitem[KPX14]{KPX14}
Kiran~S. Kedlaya, Jonathan Pottharst, and Liang Xiao, \emph{Cohomology of
  arithmetic families of {$(\varphi,\Gamma)$}-modules}, J. Amer. Math. Soc.
  \textbf{27} (2014), no.~4, 1043--1115.

\bibitem[Lou17]{Lou17}
Jo{\~a}o N.~P. Louren{\c c}o, \emph{The riemannian hebbarkeitss\"atze for
  pseudorigid spaces}, arXiv preprint arXiv:
  \href{https://arxiv.org/abs/1711.06903}{1711.06903} (2017).

\bibitem[Lur09]{HTT}
Jacob Lurie, \emph{Higher topos theory}, Annals of Mathematics Studies, vol.
  170, Princeton University Press, Princeton, NJ, 2009.

\bibitem[Lur11]{DAG8}
\bysame, \emph{Quasi-coherent sheaves and {T}annaka duality theorems},
  \url{https://www.math.ias.edu/~lurie/papers/DAG-VIII.pdf}, 2011.

\bibitem[Lur17]{HA}
\bysame, \emph{Higher algebra},
  \url{https://www.math.ias.edu/~lurie/papers/HA.pdf}, 2017.

\bibitem[Man22a]{Mann22b}
Lucas Mann, \emph{The 6-functor formalism for {$\mathbb{Z}_\ell$}- and
  {$\mathbb{Q}_\ell$}-sheaves on diamonds}, arXiv preprint arXiv:
  \href{https://arxiv.org/abs/2209.08135}{2209.08135} (2022).

\bibitem[Man22b]{Mann22}
\bysame, \emph{A p-adic 6-functor formalism in rigid-analytic geometry}, Ph.D.
  thesis, Rheinische Friedrich-Wilhelms-Universit{\"a}t Bonn, August 2022,
  \url{https://hdl.handle.net/20.500.11811/10125}.

\bibitem[Mat80]{Mat80}
Hideyuki Matsumura, \emph{Commutative algebra}, second ed., Mathematics Lecture
  Note Series, vol.~56, Benjamin/Cummings Publishing Co., Inc., Reading, MA,
  1980.

\bibitem[Mik]{Mikami25m}
Yutaro Mikami, \emph{$p$-adic monodromy theorem for families over relatively
  discrete algebras}, In preparation.

\bibitem[Mik23]{Mikami23}
\bysame, \emph{Fppf-descent for condensed animated rings}, arXiv preprint
  arXiv: \href{https://arxiv.org/abs/2311.13408}{2311.13408} (2023).

\bibitem[Mik25]{Mikami25}
\bysame, \emph{Finiteness and duality of cohomology of
  {$(\varphi,\Gamma)$}-modules and the 6-functor formalism of locally analytic
  representations}, arXiv preprint arXiv:
  \href{https://arxiv.org/abs/2504.01780}{2504.01780} (2025).

\bibitem[Nak13]{Nak13}
Kentaro Nakamura, \emph{Deformations of trianguline {$B$}-pairs and {Z}ariski
  density of two dimensional crystalline representations}, J. Math. Sci. Univ.
  Tokyo \textbf{20} (2013), no.~4, 461--568.

\bibitem[Por]{Poraterrata}
Gal Porat, \emph{Errata}, Available at
  \url{https://sites.google.com/view/galporatmath/home/}.

\bibitem[Por24]{Por24}
\bysame, \emph{Locally analytic vector bundles on the {F}argues--{F}ontaine
  curve}, Algebra Number Theory \textbf{18} (2024), no.~5, 899--946.

\bibitem[Por25]{Por22}
\bysame, \emph{Overconvergence of \'etale {$(\varphi,\Gamma)$}-modules in
  families}, Compos. Math. \textbf{161} (2025), no.~3, 555--593.

\bibitem[RC24]{RC24}
Juan~Esteban Rodr\'{\i}guez~Camargo, \emph{The analytic de {R}ham stack in
  rigid geometry}, arXiv preprint arXiv:
  \href{https://arxiv.org/abs/2401.07738}{2401.07738} (2024).

\bibitem[RJRC22]{RJRC22}
Joaqu\'{\i}n Rodrigues~Jacinto and Juan~Esteban Rodr\'{\i}guez~Camargo,
  \emph{Solid locally analytic representations of {$p$}-adic {L}ie groups},
  Represent. Theory \textbf{26} (2022), 962--1024.

\bibitem[RJRC25]{RJRC23}
\bysame, \emph{Solid locally analytic representations}, arXiv preprint arXiv:
  \href{https://arxiv.org/abs/2305.03162}{2305.03162} (2025).

\bibitem[Sch19]{CM}
Peter Scholze, \emph{Lectures on condensed mathematics},
  \url{https://people.mpim-bonn.mpg.de/scholze/Condensed.pdf}, 2019.

\bibitem[Sch20]{AG}
\bysame, \emph{Lectures on analytic geometry},
  \url{https://people.mpim-bonn.mpg.de/scholze/Analytic.pdf}, 2020.

\bibitem[Sch24]{GRLLC24}
\bysame, \emph{Geometrization of the real local {L}anglands correspondence},
  \url{https://people.mpim-bonn.mpg.de/scholze/RealLocalLanglands.pdf}, 2024.

\bibitem[SW20]{SW20}
Peter Scholze and Jared Weinstein, \emph{Berkeley lectures on {$p$}-adic
  geometry}, Annals of Mathematics Studies, vol. 207, Princeton University
  Press, Princeton, NJ, 2020.

\bibitem[SP]{stacks-project}
The Stacks~Project Authors, \emph{Stacks project},
  \url{https://stacks.math.columbia.edu}.

\end{thebibliography}
\end{document}